%% file: main.tex
\documentclass[12pt]{article}
\usepackage{graphicx}
\newcommand{\bt}{\begin{tikzcd}}
\newcommand{\et}{\end{tikzcd}}
\usepackage[numbered,framed]{matlab-prettifier}
\usepackage{enumitem}
\usepackage{amsthm}
\usepackage{amsmath}
\usepackage{tgtermes}
\theoremstyle{definition}
\usepackage{array}
\makeatletter
\newcommand{\thickhline}{%
    \noalign {\ifnum 0=`}\fi \hrule height 1pt
    \futurelet \reserved@a \@xhline
}
\newcolumntype{"}{@{\hskip\tabcolsep\vrule width 1pt\hskip\tabcolsep}}
\makeatother
\usepackage{hyperref}
\usepackage{tikz-cd}
\usepackage{xurl}
\usepackage{float}
\usepackage{amssymb}
\usepackage{mathtools}
\usepackage{mathrsfs}
\usepackage{calligra}
\usepackage[left=3cm,right=3cm, top = 3cm, bottom = 3cm]{geometry}
\usepackage{bbm}
\newcommand{\G}{\Gamma}

\newcommand{\M}{\mathcal{M}}
\newcommand{\A}{\mathcal{A}}
\newcommand{\C}{\mathbb{C}} 
\newcommand{\Z}{\mathbb{Z}} 
\newcommand{\pr}{\mathbb{P}}
\newcommand{\F}{\mathcal{F}}
\newcommand{\Hy}{\mathcal{H}}
\newcommand{\tiA}{\tilde{A}}
\newcommand{\tiB}{\tilde{B}}
\newcommand{\tiV}{\tilde{V}}
\newcommand{\tiE}{\tilde{E}}
\newcommand{\Mct}{\M^{ct}}
\newcommand{\uT}{\underline{T}}
\newcommand{\ul}{\underline} 
\newcommand{\ol}{\overline}

\newcommand{\HH}{\mathbb{H}}
\tikzset{labl/.style={anchor=south, rotate=-90, inner sep=.5mm}}
\newtheorem{theorem}{Theorem}[section]
\newtheorem{lemma}[theorem]{Lemma}
\newtheorem{proposition}[theorem]{Proposition}
\newtheorem{corollary}[theorem]{Corollary}
\newtheorem{definition}[theorem]{Definition}

\newtheorem{remark}[theorem]{Remark}
\newtheorem{example}[theorem]{Example}

\usepackage{epigraph}
\usepackage{hyperref}
\usepackage{appendix}
\usepackage{caption}
\usepackage{mathtools}

\begin{document}

\title{Self-intersection of the Torelli map}
\author{Lycka Drakengren}
\date{}
\maketitle

\begin{abstract}
    The Torelli map $t\colon \Mct_g \to \A_g$ is far from an immersion for $g\geq 3$: the self-fiber product of the Torelli map for $g\geq 3$ has several components with nontrivial intersections. We give a stratification of the self-fiber product for arbitrary genus and describe how components in the fiber product intersect. In genus $4$, the Torelli fiber product is nonreduced, which we prove by analyzing the expansion of the period map near a nodal curve. We use the geometry of the Torelli fiber product to:
    \begin{itemize}
        \item Calculate the class of the pullback to $\Mct_4$ of the Torelli cycle $t_*[\Mct_4]$ on $\A_4$
        \item Find the class $t_*[\ol{\M}_4]$ for suitable toroidal compactifications $\ol{\A}_4$
        \item Calculate the class $t^*t_*[\Mct_5]|_{\M_5}$
    \end{itemize}
In the first appendix, we write down a calculation for finding the Chern classes of $\ol{\M}_{g,n}$. In the second, we give a formula for a coefficient occurring in an intersection of excess dimension. 
\end{abstract}

\setcounter{tocdepth}{1}
\tableofcontents

\section{Introduction}

\subsection{The Torelli cycle}
The moduli space of \textit{compact type curves} $\Mct_g\subset \overline{\M}_g$ is the open subset parametrizing stable curves of genus $g$ for which all nodes are disconnecting. The moduli space $\mathcal{A}_g$ parametrizes principally polarized abelian varieties of dimension $g$. For introductions to the moduli spaces of curves and abelian varieties, see \cite{harris} and \cite{BL}. We work over the complex numbers throughout. 

For $[C] \in \Mct_g$, the polarized \textit{Jacobian} $J(C)$ consists of line bundles having degree zero on every irreducible component. The polarization is given by the theta divisor. The two moduli spaces are related by the \textit{Torelli map} 
\begin{equation*}
\begin{aligned}
t\colon &\Mct_g \to \A_g\\
&[C] \mapsto [J(C)].
\end{aligned}
\end{equation*}
The domain $\Mct_g$ is the largest subset of $\overline{\M}_g$ of curves $[C]$ such that $J(C)$ is proper and hence an abelian variety. The Torelli map induces a pullback
$$t^*\colon \mathsf{CH}^*(\A_g) \to \mathsf{CH}^*(\Mct_g)$$ \noindent since $\A_g$ is nonsingular, and a pushforward
$$t_*\colon \mathsf{CH}^*(\Mct_g) \to \mathsf{CH}^*(\A_g)$$ \noindent since $t$ is proper. The study of the image $t(\Mct_g)$ inside $\A_g$ is the famous Schottky problem, see \cite{gru}. We will study the \textit{Torelli cycle} $$T_g = {t}_*[\Mct_g]\in \mathsf{CH}^*(\A_g).$$ 

The class $t^*T_g$ can be calculated using the geometry of the \textit{Torelli fiber product} $\F_g$, defined by the fiber diagram \begin{equation*}
\begin{tikzcd}
     \F_g\ar[d,"p_1"]\ar[r,"p_2"] & \Mct_g\ar[d,"t_2"]\\
    \Mct_g \ar[r,"t_1"] & \A_g
\end{tikzcd}  
\end{equation*}
\noindent where $t_1$, $t_2$ are Torelli maps. The Torelli fiber product consists of tuples $$(C_1, C_2, \begin{tikzcd}
    \alpha \colon J(C_1) \ar[r,"\sim"]& J(C_2)\hspace{-5pt}
\end{tikzcd})$$ where $[C_1],[C_2] \in \Mct_g$. The space $\F_g$ is an interesting geometric object on its own. Due to different choices of automorphisms and from the fact that the Jacobians of any two compact type curves with the same irreducible components are isomorphic, $\F_g$ has many different components and intersections between them when $g\geq 3$. To fully understand the components and their intersections, we describe a stratification of $\F_g$ for any genus $g$ and give a combinatorial description of the strata and their closures.
 
In \cite{Mu}, Mumford proves the following theorem using modular forms:

 \begin{theorem}\label{muthm}
     The Torelli cycle in genus $4$ satisfies $T_4 = 16\lambda_1\in \mathsf{CH}^1(\A_4)$.
 \end{theorem}

We provide an alternative way of calculating the class $t^*{T_4} = 16\lambda_1 \in \mathsf{CH}^1(\Mct_4)$ using excess intersection theory on $\F_g$. Since $\text{Pic}(\A_4) = \mathbb{Q}\lambda_1$, this yields a different proof of Theorem \ref{muthm}, see Section \ref{proofofmuthm}.

A surprising property of $\F_g$ is that the space is nonreduced for $g\geq 4$. We determine the generic structure of this nonreducedness in genus $4$ by analyzing the local expansion of the Torelli map near a given nodal curve using the variational formulas for abelian differentials in \cite{HN}. We calculate the expansion of the period matrix up to second order with respect to a smoothing parameter. This involves computations of contour integrals on fixed curves of genus $2$. We arrive at the following result in Section \ref{proofofthm2}:

\begin{theorem}\label{thm2}\leavevmode
\begin{enumerate}[label=(\roman*)]
    \item The nonreduced locus in $\F_4$ relevant to the calculation of $t^*T_4$ consists of two $8$-dimensional loci inside the component parametrizing pairs of curves $(C_1 \cup_{p\sim q}C_2, C_1 \cup_{r\sim s}C_2)$ where $g(C_1)=g(C_2)=2$ and the automorphism of Jacobians is taken to be the identity. The two nonreduced loci parametrize the pairs of curves for which $(r,s) = (p,\bar{q})$, $(\bar{p},q)$ respectively, where $\bar{p}, \bar{q}$ are the hyperelliptic conjugates of $p,q$.
    \item The local scheme structure at a general nonreduced point in $\F_4$ is given by $$\text{Spec }\frac{\C[[x_1,\dots,x_9,x_{10},x_{11}]]}{(x_9x_{11}, x_{10}x_{11}, x_{11}^2).}$$
\end{enumerate}
\end{theorem}

Consider an extension $t\colon \ol{\M}_g \to \ol{\A}_g$ of the Torelli map to a toroidal compactification $\ol{\A}_g$ with irreducible boundary divisor, such as the perfect cone compactification. In Section $4$, we obtain the following generalization of Theorem \ref{muthm}, also proved by Mumford in \cite{Mu}:

\begin{theorem}\label{thm3}
 Let $\ol{\A}_4$ be a toroidal compactification to which the Torelli map extends, and such that its boundary divisor $D=\ol{\A}_4\backslash \A_4$ is irreducible. Then $t_*[\ol{\M}_4] = 16\lambda_1 - 2D$ in $\mathsf{CH}^1(\ol{\A}_4)$.
\end{theorem}

As can be seen in Section \ref{torfib}, the space $\F_g$ is quite complex to describe in general. The complexity of $\F_g$ is mostly a consequence of the positive dimensional fibers of Torelli map. In genus $4$, the relevant geometry for calculating $t_*T_4$ can be made very explicit, see Section \ref{geofour}. For $g\geq 5$, we determine the restricted class $t^*T_g|_{\M_g}$. For $g=5$, we obtain the following result, proved in Section \ref{fivesection}:

\begin{theorem}\label{thm4}
    The class $t^*T_5|_{\M_5}$ is given by $\frac{48}{5}\kappa_3$ and agrees with $t^*\text{taut}(T_5)|_{\M_5}$.
\end{theorem}

For $g\geq 6$, we prove in Section \ref{fivesection} that $t^*T_g|_{\M_g} = 0$.

It would be interesting to further investigate the following questions:
\begin{itemize}
   \item[]\textbf{Question 1:} What is the precise nonreduced scheme structure of $\F_g$ in general? 
\end{itemize}
\begin{itemize}
   \item[]\textbf{Question 2:} Can we find a birational map $\A'_g \to \A_g$ and a lifting $t'\colon \Mct_g \to \A'_g$ of the Torelli map such that $t'$ is finite, and $\A_g'$ has a natural interpretation as a moduli space?
\end{itemize}

\subsection{Tautological classes} The best understood part of $\mathsf{CH}^*(\A_g)$ is the \textit{tautological ring} $\mathsf{R}^*(\A_g) \subset \mathsf{CH}^*(\A_g)$ generated by the Chern classes $\lambda_i = c_i(\mathbb{E})$ of the Hodge bundle $\mathbb{E} = \pi_*\Omega_\pi$ where $\bt \pi\colon \mathcal{X}_g \to \A_g\et$ is the universal abelian variety \cite{vdg}. In an ongoing project with Samir Canning and Aitor Iribar López, we enlarge the unirational parametrization of $\A_4$ in \cite{izadi} with the aim of showing that $\mathsf{CH}^*(\A_4)= \mathsf{R}^*(\A_4)$. On the other hand, the class $[\A_1\times \A_5]\in \mathsf{CH}^5(\A_6)$ is not tautological \cite{nontaut}.

By \cite{taut}, there is a \textit{tautological projection} $\bt \text{taut}\colon \mathsf{CH}^k(\A_g)\to \mathsf{R}^k(\A_g) \et$ which maps a cycle $\alpha$ to the unique tautological class $\text{taut}(\alpha)$ such that 
$$\int_{\overline{\A}_g}\alpha \cdot\beta\cdot\lambda_g = \int_{\overline{\A}_g}\text{taut}(\alpha) \cdot\beta\cdot\lambda_g$$ for every $\beta\in \mathsf{R}^{\binom{g}{2}-k}(\A_g)$ and any choice of toroidal compactification $\ol{\A}_g$. Each $\alpha\in \mathsf{CH}^*(\A_g)$ can be decomposed uniquely as $\alpha = \text{taut}(\alpha) + \alpha_0$ where $\text{taut}(\alpha)\in \mathsf{R}^*(\A_g)$ and $\text{taut}(\alpha_0)=0$.

While $T_4 = 16\lambda_1$ is tautological, the following question remains open:

\begin{itemize}
   \item[]\textbf{Question 3:} For which $g\geq 5$ is $T_g$ tautological?
\end{itemize}

Comparing the classes $t^*T_g$ and $t^*\text{taut}(T_g)$ could provide a potential obstruction to $T_g$ being tautological, which motivates our interest in determining the class $t^*T_g$. This will only be possible for $g\leq 9$ for dimension reasons. Indeed, we have $\mathrm{dim} (\A_g) = \frac{g(g+1)}{2}$ and $\mathrm{dim}(\Mct_g) = 3g-3$, so
$$\mathrm{dim}(t^*T_g)=\frac{-g^2+11g-12}{2}.$$ Setting $g\geq 10$  thus gives $\mathrm{dim}(t^*T_g)<0$ and hence $t^*T_g = 0$.

In addition, for $g=8$, using Definition \ref{tautring} we have $\mathsf{CH}^*(\Mct_8) = \mathsf{R}^*(\Mct_8)$ and $R^i(\Mct_8) = 0$ for $i>2g-3$ by \cite{CL}, \cite[Theorem 1.10]{CLP} and excision. Since $\mathrm{dim}(t^*T_8) < 8$ we obtain $t^*T_8 = 0$.

The tautological projections of $T_g$ for $g \leq 7$ have been calculated by Carel Faber in \cite{faber}. For $g=5,6,7$ these are
\begin{equation*}
\begin{aligned}
\text{taut}(T_5) &= 2(72\lambda_1\lambda_2 - 48\lambda_3) \\
\text{taut}(T_6) &=  2(384\lambda_1\lambda_2\lambda_3 - 1152\lambda_2\lambda_4 + \frac{474048}{691} \lambda_1\lambda_5 - \frac{248064}{691} \lambda_6)\\
\text{taut}(T_7) &=  2(768\lambda_1\lambda_2\lambda_3\lambda_4 - 6912\lambda_2\lambda_3\lambda_5 + \frac{2209152}{691} \lambda_1\lambda_4\lambda_5 + \frac{7522176}{691} \lambda_1\lambda_3\lambda_6\\
& \hspace{12pt} - \frac{8842752}{691} \lambda_4\lambda_6  + \frac{968832}{691} \lambda_3\lambda_7 - \frac{3276672}{691} \lambda_1\lambda_2\lambda_7)
\end{aligned}
\end{equation*}
\noindent where the factor $2$ arises from the fact that the Torelli map is generically $2:1$ onto its image.

\subsection{Geometry of the Torelli fiber product}

We illustrate the correspondence between geometric and combinatorial strata in the Torelli fiber product.

Recall that a geometric point in the fiber product $\F_g$ corresponds to a pair of curves and a choice of isomorphism between the Jacobians preserving the polarization. 

The automorphisms of the polarized Jacobian $J(C)$ for a smooth curve $C$ are as follows \cite[Theorem 12.1]{Mi}:
\begin{itemize}
    \item If $g(C) > 2$ and $C$ is not hyperelliptic, $\text{Aut}(J(C))=\text{Aut}(C)\times \Z/2\Z$ where the generator of $\Z/2\Z$ is taken to be the sign involution of $J(C)$.

    \item If $C$ is hyperelliptic of genus $g>1$, we have $\text{Aut}(C) = \text{Aut}(J(C))$.

    \item If $g(C)=1$, $\text{Aut}(J(C))=\text{Aut}((C,p))$ where $p$ is any fixed point on $C$.
\end{itemize}
We will denote the sign involution of $J(C)$ by $-$ and the identity automorphism by $+$.

The Torelli theorem states that, for a smooth curve $C$, the polarized Jacobian $J(C)$ determines the curve. The closure of the locus in $\F_g$ parametrizing pairs of smooth curves will therefore consist of two components:
$$\Delta^{+} = \{(C, C, \begin{tikzcd}
    + \colon J(C)\hspace{-5pt} \ar[r,shorten >=5pt, shorten <=5pt,"\sim"]&\hspace{-5pt} J(C))\hspace{2pt}| \hspace{2pt}[C]\in \Mct_g\}\cong \Mct_g
\end{tikzcd}$$
$$\Delta^{-} = \{(C, C, \begin{tikzcd}
    - \colon J(C)\hspace{-5pt} \ar[r,shorten >=5pt, shorten <=5pt,"\sim"]& \hspace{-5pt}J(C))\hspace{2pt}| \hspace{2pt}[C]\in \Mct_g\}\cong \Mct_g.
\end{tikzcd}$$

The Jacobian of a compact type curve $C$ corresponds to the Jacobian of its normalization. By \cite[Corollary 3.23]{CG}, the decomposition of $J(C)$ into a product of Jacobians of irreducible curves of positive genus is unique. This means that two compact type curves have isomorphic Jacobians precisely when their irreducible components of positive genus can be paired up such that the irreducible curves in a pair are isomorphic. Different choices of isomorphisms between the Jacobians (up to isomorphisms of the curves) of the irreducible components will result in different elements of the fiber product. Thus, an element of $\F_g$ consists of: 
\begin{itemize}
    \item A pair of compact type curves $C_1,C_2$ of genus $g$
    \item A chosen pairing between the irreducible components of positive genus of $C_1$ and $C_2$ such that the paired components are isomorphic
    \item A choice of isomorphism of Jacobians for each set of paired irreducible components
\end{itemize}

\begin{figure}[h]
\centering
\scalebox{0.5}{\input{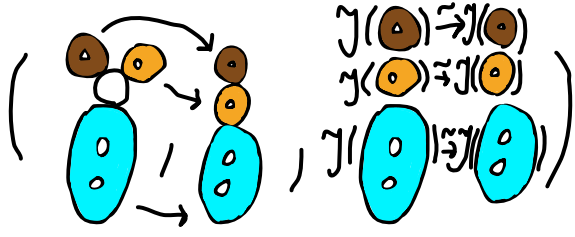}}
\caption{Sample element of the Torelli fiber product for $g=4$.}\label{g0}
\end{figure}

Writing the $\text{Spec}(\C)$-points of $\F_g$ as tuples
$$(C_1, C_2, \begin{tikzcd}
    \alpha \colon J(C_1) \ar[r,"\sim"]& J(C_2)\hspace{-5pt}
\end{tikzcd})$$ where $[C_1],[C_2] \in \Mct_g,$ isomorphisms of objects $(C_1,C_2,\alpha)$ and $(D_1,D_2,\beta)$ in $\F_g$ are given by
\begin{equation*}
\left(\begin{tikzcd}C_1 \ar[d, "i_1"',"\sim" labl]\\ D_1\end{tikzcd}, \begin{tikzcd}C_2 \ar[d, "i_2"' ,"\sim" labl]\\ D_2\end{tikzcd}, \begin{tikzcd}
   \alpha \colon  J(C_1) \ar[d,"t_1 (i_1)"',"\sim"labl]\ar[r,"\sim"]& J(C_2)\ar[d,"t_2 (i_2)"',"\sim" labl]\\
     \beta \colon J(D_1) \ar[r,"\sim"]& J(D_2)
    \hspace{-5pt}
\end{tikzcd}\right).
\label{fiberiso}
\end{equation*}

The components of the fiber product will parametrize pairs of curves with given topological types, with a chosen pairing between irreducible components of equal positive genus, specifying a given isomorphism of Jacobians generically.

Combinatorially, a component will be described by a pair of trees, which are the dual graphs of each pair of curves parametrized by the component. Moreover, the pairing on irreducible components of the curves translates to a pairing between vertices of equal positive genus of the dual graphs. The specified isomorphism of Jacobians corresponds to sign choices for the paired vertices.

We stratify the fiber product according to different choices of dual graphs, pairings and automorphisms. We also distinguish between hyperelliptic/nonhyperelliptic components of the curves, and separate the cases where nodes have certain relations on the paired irreducible components. 

To motivate this choice of stratification, note that the two diagonal components $\Delta^{+}, \Delta^{-}$, both viewed as $\Mct_g$, intersect in the locus $\Hy^{ct}_g$ of hyperelliptic curves. Moreover, a component consisting generically of pairs of $1$-nodal curves with automorphisms $+$ of Jacobians $$(C_1\cup_{p\sim q} C_2, C_1\cup_{r\sim s} C_2,\bt + \colon J(C_i) \hspace{-5pt}\ar[r,shorten >=5pt, shorten <=5pt,"\sim"]& \hspace{-5pt}J(C_i), \text{ $i=1,2$})\et$$
\noindent intersects $\Delta^+$ when the nodes $p=r$, $q=s$ agree on $C_1$, $C_2$. The case when both automorphisms of Jacobians are $-$ is analogous.

\subsection{The Torelli cycle in genus $4$}\label{cygefour}
There are five components of $\F_4$ that play a role in the calculation of $t^*T_4$. These are as follows:

\begin{itemize}
    \item The \textbf{diagonal components} $\Delta^{+}$, $\Delta^{-}$ which parametrize pairs of isomorphic curves $$(C, C, \begin{tikzcd}
    + \colon J(C) \hspace{-5pt}\ar[r,shorten >=5pt, shorten <=5pt,"\sim"]& \hspace{-5pt}J(C)),
\end{tikzcd}\text{ } (C, C, \begin{tikzcd}
    - \colon J(C)\hspace{-5pt} \ar[r,shorten >=5pt, shorten <=5pt,"\sim"]& \hspace{-5pt}J(C))
\end{tikzcd}$$
         \noindent where $+$, $-$ are the identity and sign involutions of Jacobians. We have $\mathrm{dim}(\Delta^\pm) = 9$.

    \item The \textbf{(1,3)-components} $A^+,A^-$ which parametrize pairs of curves consisting of a genus $1$ curve $E$ glued to a genus $3$ curve $C$ where the choice of automorphism of $J(C)$ is $+$ for $A^+$ and $-$ for $A^-$. The automorphism of $J(E)$ can be taken as $+$. The dimensions are $\mathrm{dim}(A^\pm) = 9$.

    \item The \textbf{(2,2)-component} $B$ which consists of pairs of curves with two genus $2$ components glued at a node. Genus $2$ curves are hyperelliptic, so the automorphisms of Jacobians can be taken to be $+$. We have $\mathrm{dim}(B) = 10$.
\end{itemize}

\begin{figure}
\centering
\scalebox{0.7}{\input{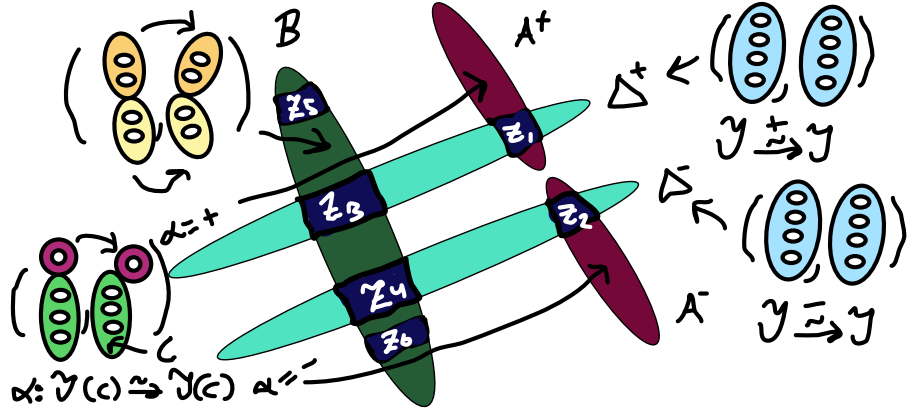}}
\caption{The components and intersections of $\F_4$ which contribute to $t^*T_4$.}\label{fib4}
\end{figure}

The intersections that will contribute to $t^*T_4$ are 
$$ Z_1 = \Delta^+\cap A^+, \hspace{5pt} Z_2 = \Delta^-\cap A^-, \hspace{5pt}  Z_3 = \Delta^+\cap B \hspace{5pt}
\text{and} \hspace{5pt} Z_4 = \Delta^-\cap B.$$ These are all of dimension $\mathrm{dim}(t^*T_4) = 8$ and will therefore contribute to $t^*T_4$ with a multiple of their fundamental class.

Surprisingly, the fiber product $\F_4$ also carries nonreduced structure that will be relevant to the calculation of $t^*T_4$. Precisely, there will be two nonreduced loci $Z_5$, $Z_6$ of dimension $8$ inside $B$, generically disjoint from $B\cap \Delta^{\pm}$.

See Figure \ref{fib4} for an illustration of the components, intersections and the nonreduced loci of $\F_4$ contributing to $t^*T_4$.

Using Proposition \ref{form}, we will be able to decompose the class $t^*T_4$ as
\begin{equation*}
\begin{aligned}
{p_1}_* \Biggl[2c_1\left(\frac{p_1^*N_{\Mct_4/\A_4}}{N_{\Delta^+/\Mct_4}}\right)&+2c_1\left(\frac{p_1^*N_{\Mct_4/\A_4}}{N_{A^+/\Mct_4}}\right)+c_2\left(\frac{p_1^*N_{\Mct_4/\A_4}} {N_{B/\Mct_4}}\right)\\
&-2[Z_1] -2[Z_2] -3[Z_3]-3[Z_4] + [Z_5] + [Z_6] \Biggr].
\end{aligned}
\end{equation*}
The coefficients of $[Z_5]$, $[Z_6]$ will be found using results in Section \ref{nonred2} where we determine the local scheme structure of the nonreducedness. This will be done by analyzing the local structure of the Torelli map near a nodal curve.

We will find in Section \ref{inttheory} that the contributions ${p_1}_*c_{\text{top}}\left(\frac{p_1^*N_{\Mct_4/\A_4}}{N_{X/\Mct_4}}\right)$ for $X = \Delta^+, A^+, B$ are $8\lambda_1 - 2\delta$, $4\delta_A$, $8\delta_B$ respectively, where
\begin{itemize}
    \item $\delta_A$ is the class of the locus in $\Mct_4$ of curves consisting of a genus $3$ curve glued to a genus $1$ curve
    \item $\delta_B$ is the class of the locus of two glued genus $2$ curves
    \item  $\delta = \delta_A + \delta_B$
\end{itemize} Using ${p_1}_*[Z_1] = {p_1}_*[Z_2] = \delta_A$ and ${p_1}_*[Z_i] = \delta_B$ for $3\leq i \leq 6$ we obtain $t^*T_4 = 16\lambda_1$ as claimed.

\subsubsection{Motivation for the nonreducedness}
We will show that the nonreducedness in $\F_4$ contributing to $t^*T_4$ lies in the component $B$, see Theorem \ref{thm2}. This nonreduced structure will be supported on two $8$-dimensional loci $Z_5$ and $Z_6$, given by pairs of isomorphic curves with choices $(+,-)$ resp. $(-,+)$ of automorphisms of Jacobians. This corresponds to having the genus $2$ components $C_1$, $C_2$ glued together at points $p,q$ in $C$, and at $p,\bar{q}$ resp. $\bar{p},q$ in $D$ with notation as in Section \ref{geofour}.

We can see the nonreducedness geometrically as follows: Consider a fixed curve $C$ consisting of two genus $2$ curves glued at a node which is not a Weierstrass point on any of the components. The Torelli map induces a map on normal spaces to the loci of reducible curves in $\Mct_4$ and product abelian varieties in $\A_4$. This map will identify the image of a nonzero vector $v\in N_{\Mct_{2,1}\times \Mct_{2,1}/\Mct_4}$ with the image of $-v$. This will produce an extra dimension of tangent vectors in the fiber product which does neither lie in the diagonal components nor lies parallel to $B$, and must therefore give rise to nonreducedness.

\subsection{Outline of the paper}
\begin{itemize}
    \item The geometry of the fiber product $\F_g$ is described in Section \ref{torfib}. We stratify the fiber product and describe how the strata specialize, allowing us to describe the intersections between components in the fiber product.

\item The calculation of $t^*{T_4}$ is the main content of Section \ref{torellisec}. We determine the geometry of the components and intersections of $\F_4$ contributing to the class, and calculate their contributions using pullbacks and pushforwards of tautological classes on $\Mct_{g,n}$ under gluing and forgetful maps.

\item In Section \ref{nonred2} we determine the precise generic scheme structure of the nonreduced locus in $\F_4$ contributing to the class $t^*T_4$.

\item In Section \ref{abar} we calculate the class $t_*[\overline{\M}_4]$ on $\overline{\A}_4$ for a toroidal compactification with irreducible boundary divisor via pullback to $\overline{\M}_4$.

\item We calculate the class $t^*T_5|_{\M_5}$ in Section \ref{fivesection}. 

\item In Appendix \ref{chern}, we give a formula for calculating the Chern classes of $T\overline{\M}_{g,n}$. We have implemented these formulas in admcycles \cite{admcycles}. A different method for calculating the Chern classes of $\ol{\M}_{g,n}$ is given in \cite{Bini}, but there are mistakes in some formulas. For instance, the coefficient of $\kappa_2$ in $c_2(T\Mct_4)$ should be $-\frac{1}{2}$ as in (\ref{c2calc}) instead of $-\frac{1}{3}$ as claimed in \cite[Theorem 2]{Bini}.

\item Let $X$, $V$ be subvarieties of a smooth ambient variety $Y$ such that $X\cap V = A\cup B$ where $A\cap B$ is smooth and $\mathrm{dim}(A\cap B) = \mathrm{dim}(X\cdot V)$. In Appendix \ref{mf} we give a formula for the contribution of $A\cap B$ to the intersection class $X\cdot_Y V$. The contribution only depends on the dimensions of $A$, $B$ and $A\cap B$.

\item In Appendix \ref{matlab}, we include a MATLAB code for numerical evaluation of contour integrals, used in Section \ref{nonred2}.
\end{itemize}

\subsection*{Acknowledgements}
I would like to thank Rahul Pandharipande for many crucial ideas in this project. His artistic sense for intersection theory has given me insights in ways of thinking about concepts in the field. Johannes Schmitt has contributed significantly to my understanding of the geometry of the fiber product. Most importantly, he has taught me how to think accurately about the automorphisms involved in fiber products of stacks and about the ways of parametrizing strata in moduli spaces. I greatly appreciate the advice I received from Sam Grushevsky to use the expansion of the period matrix to find the precise scheme structure of the Torelli fiber product. I am also grateful for the numerous valuable discussions I have had with Samir Canning, Jeremy Feusi and Aitor Iribar López. Lastly, I thank Barbara Fantechi, Aaron Landesman, Samouil Molcho, Martin Möller, Denis Nesterov and Nicola Tarasca for conversations that gave me new viewpoints on the topic. The grant SNF-200020-219369 supported this project.

\section{The Torelli fiber product}\label{torfib}

The \textit{Torelli fiber product} is defined as the space $\F_g$ in the fiber diagram of stacks

\begin{equation*}
\begin{tikzcd}
     \F_g\ar[d,"p_1"]\ar[r,"p_2"] & \Mct_g\ar[d,"t_2"]\\
    \Mct_g \ar[r,"t_1"] & \A_g
\end{tikzcd}  
\end{equation*}
\noindent where $t_1$, $t_2$ are Torelli maps.

In order to calculate $t^*{T_g}$ using excess intersection theory \cite{Ful}, we would like to understand the geometry of $\F_g$. In particular, we would like to describe the irreducible components and their intersections. We would also like to locate possible nonreduced scheme structure. 

In this section, we give a stratification of the fiber product and provide parametrizations of the strata. We first describe a combinatorial stratification and define specializations of combinatorial strata. Next, we describe the geometric loci associated to the combinatorial strata. A combinatorial specialization will correspond to containment of a geometric stratum inside the closure of another stratum. The containment of a stratum inside the closure of a given geometric stratum can be found from the parametrizations of the closures. As a consequence, we obtain parametrizations for the intersections of closures of strata. Given this, we can deduce the full reduced scheme structure of  $\F_g$. 

On the other hand, the fiber product $\F_g$ turns out to interestingly carry nonreduced structure, which will be explored in Section \ref{nonred2}.

\subsection{The combinatorial strata}\label{combstrata}

In this section, we introduce combinatorial strata which will be shown in Section \ref{geostrata} to classify the geometric strata of the fiber product.

\begin{definition}\label{combtor}
    A \textit{combinatorial Torelli pair (CTP)} of genus $g$ will denote the choice of:

    \begin{enumerate}
        \item \textbf{(Pair of trees)} An ordered pair of stable trees $(T_1,T_2)$ of genus $g$ satisfying the criteria in Appendix A.1 of \cite{GP}.
        \item \textbf{(Vertex pairing)} A bijection between the sets of vertices of positive genus of $T_1,T_2$, denoted $\bt\nu\colon P(T_1) \to P(T_2), \et$ such that $g(v) = g(\nu(v))$ for every $v\in P(T_1)$.
        
        \item \textbf{(Sign choice)} For each vertex $v\in V(T_1)$ with $g(v) \geq 2$, a sign choice $+$, $-$ or $\pm$, denoted 
        \begin{equation*}
            \bt \sigma \colon \{v\in V(T_1)|g(v) \geq 2\} \to \{+,-,\pm\},\et
        \end{equation*}
        \noindent where $g(v)=2$ requires the choice $\sigma(v) = \pm$.
        
        \item \textbf{(Positive and negative half-edge pairings, $g(v)\geq 2$)} For each set of paired vertices $v,\nu(v)$ with $g(v)\geq 2$, a pairing $\gamma^+$ between a subset $H_1^+\subset H(v)$ and $H_2^+\subset H(\nu(v))$ of adjacent half-edges, and another pairing $\gamma^-$ between subsets $H_1^-\subset H(v)$ and $H_2^-\subset H(\nu(v))$. The subsets $H_1^+$ and $H_1^-$ are allowed to have nonempty intersection, and moreover, the pairings $\gamma^+$, $\gamma^-$ need not be compatible on $H_1^+\cup H_1^-$. However, we require that $H_1^- = \emptyset$ if $\sigma(v) = +$, and $H_1^+ = \emptyset$ if $\sigma(v) = -$.

        \hspace{0.5cm}To ensure that we obtain disjoint, nonempty geometric strata from the pairings $\gamma^+$, $\gamma^-$ we impose a further condition called the \textit{bipartite graph condition}:
        
        \textbf{Bipartite graph condition:}
        Consider a bipartite graph whose left vertex set corresponds to the half-edges at $v$, and the right vertex set consists of the half-edges at $\nu(v)$. For each pair $h,h'$ where $h\in H_1^+$ and $\gamma^+(h) = h'$, draw a blue edge between the vertices corresponding to $h$ and $h'$. Repeat for $\gamma^-$ with red edges (multiple edges are allowed). Note that each vertex has at most two adjacent edges, with no repeated color. The graph thus decomposes into disjoint sets of paths and cycles.
        
        \hspace{0.5cm} We impose the following conditions on this bipartite (multi)graph: The components of the graph which are cycles must have length $2k$ where $k=1$ or $2$ (longer cycles will imply that the marked points on the curve corresponding to $v$ are no longer disjoint). Similarly, the components which are paths must have length $0$, $1$, $2$ or $3$. We also require there to be at most $2g(v)+2$ cycles of length $2$ (the vertices in such a cycle will correspond to Weierstrass points).
        
        \hspace{0.5cm} A set of pairings $\gamma^+,\gamma^-$ is equivalent to another set of pairings if and only if the corresponding bipartite graphs are equal after completing the length three paths to cycles with two edges of each color, see Figure \ref{bipartite}. The equivalences will relate to implications of the form 
        \begin{equation*}
            p = r, q = s, p = \bar{s} \implies q = \bar{r}
        \end{equation*}
         \noindent where $p,q$ and $r,s$ are marked points on the curves corresponding to $v$ and $\nu(v)$ respectively, and $\bar{p}$ denotes the hyperelliptic conjugate of a point $p$.

        \begin{figure}[h]
\centering
\scalebox{0.5}{\input{bipartite.tex}}
\caption{Example of two equivalent and admissible bipartite graphs associated to different two different sets of pairings $\gamma^+,\gamma^-$.}\label{bipartite}
\end{figure}
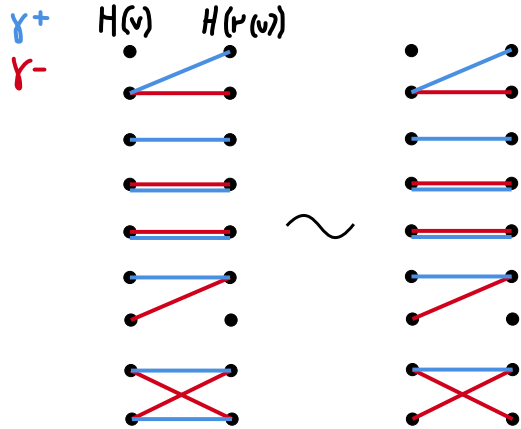
        
        \item \textbf{(Relative half-edge pairings, $g(v)=1$)} \label{relative}For each vertex $v\in V(T_1)$ of genus $1$ and pair of half-edges $h\in H(v)$, $h'\in H(\nu(v))$, a pairing $id_{h,h'}$ between a collection of half-edges $H_{h,h'}\subset H(v)$ and $H'_{h,h'}\subset H(\nu(v))$, and another pairing $\tau_{h,h'}$ between sets of half-edges $K_{h,h'}\subset H(v)$ and $K'_{h,h'}\subset H(\nu(v))$. The pairings must be \textit{admissible} and are subject to an equivalence relation as in Definition \ref{admeq}.

        \hspace{0.5cm} In particular, we can assume the following conditions:
        \begin{itemize}
            \item \textbf{(Reflexivity)} For each $h,h'$ as above, $h$ lies in $H_{h,h'}$ and $id_{h,h'}(h) = h'$. Moreover, $h\in K_{h,h'}$ and $\tau_{h,h'}(h) = h'$.
            
            \item \textbf{(Symmetry)} If $g\in H_{h,h'}$ with $g' = id_{h,h'}(g)$, then $h \in H_{g,g'}$ with $h' = id_{g,g'}(h)$. Similarly, if $g\in K_{h,h'}$ with $g' = \tau_{h,h'}(g)$, then $h \in K_{g,g'}$ with $h' = \tau_{g,g'}(h)$.
            
            \item \textbf{(Transitivity)} If $g\in H_{h,h'}$, $g' = id_{h,h'}(g)$ and $k \in H_{g,g'}$ with $id_{g,g'}(k) = k'$, then $k \in H_{h,h'}$ with $id_{h,h'}(k) = k'$. If $g\in K_{h,h'}$, $g' = \tau_{h,h'}(g)$ and $k \in K_{g,g'}$ with $\tau_{g,g'}(k) = k'$, then $k \in K_{h,h'}$ with $\tau_{h,h'}(k) = k'$.
        \end{itemize}

        \item \textbf{(Half-edge pairings at vertices of genus $0$)} For each pair of vertices \newline $v_1\in V(T_1)$, $v_2\in V(T_2)$ of genus $0$, a collection $S_{v_1,v_2}$ of tuples $$(\Gamma,\Gamma_i,\bt \alpha_i\colon \Gamma \to\Gamma_i \et \text{ for }i=1,2)$$ where $\Gamma$ is either a stratum of $\Hy^{ct}_{0,t,2c}$ (defined in \cite[Notation 4.14]{SvZ}) for some $t,c$ or a stratum of $\Mct_{0,l}$ for some $l$, $\Gamma_i$ is a stable graph obtained by taking the vertex $v_i$ and adding a subset of its adjacent half-edges, and $\alpha_i$ for $i=1,2$ are specializations of stable graphs induced by forgetful morphisms. These pairings are admissible and satisfy an equivalence relation as in Definition \ref{admeq0}.

    \end{enumerate}
\end{definition}
    To describe which relative half-edge pairings are equivalent, consider a set of relative half-edge pairings as in \ref{combtor}. View the half-edges as elements of the complex torus $\mathbb{C}/\langle 1,i \rangle$, letting $z_1,\dots ,z_m$ be the points corresponding to the half-edges $h_1,\dots, h_m$ of $v$, and $w_1,\dots, w_n$ those corresponding to $k_1,\dots,k_n$ of $\nu(v)$. Consider the system of equations 
    \begin{equation*}
    \begin{cases}
      (z_i-z_j) = (w_r-w_s) \text{ for } i,j,r,s \text{ such that } id_{h_i,k_r}(h_j)=k_s \\
      (z_i-z_j) = -(w_r-w_s) \text{ for } i,j,r,s \text{ such that } \tau_{h_i,k_r}(h_j)=k_s
    \end{cases}\
\end{equation*}
 in $\mathbb{C}/\langle 1,i \rangle$. Note that the solutions for $z_j,w_k$ in $\mathbb{C}/\langle 1,i \rangle$ correspond to those for $z_j',w_k'$ in $\mathbb{C}/\langle 1,\tau \rangle$ for $\text{im}(\tau) > 0$ via $$z_j = a_j + b_ji\iff z'_j = a_j + b_j\tau$$ and similarly for $w_k,w_k'$.
    
    An equation $(z_i-z_j) = -(w_r-w_s)$ in this system corresponds to, after a transformation $i \mapsto \tau$ of the complex torus, the equation $q_s = q_r - p_i + p_j$ where $p_i$ are the points corresponding to half-edges $h_i$ and $q_r$ the points corresponding to $k_r$ for the elliptic curve $\mathbb{C}/\langle 1,\tau \rangle$. This equation says precisely that $q_r$ is mapped to $q_s$ after first sending $p_i$ to $p_j$, since the translation taking $p_i$ to $p_j$ is given by $x \mapsto x - p_i + p_j$. Similarly $(z_i-z_j) = -(w_r-w_s)$ corresponds to $q_r$ being sent to $q_s$ after taking $p_i$ to $p_j$ and applying the involution $x \mapsto -(x-p_j)$ in $p_j$.

    \begin{definition}\label{admeq}
    We consider two sets of relative half-edge pairings \textit{equivalent} if the corresponding linear systems give the same solutions.
    We say that an equivalence class of relative half-edge pairings is \textit{admissible} if the linear system has a solution where the $(z_i)_{1\leq i \leq m}$ resp. $(w_r)_{1\leq r \leq n}$ are distinct, and moreover if $$(z_i-z_j) = (w_r-w_s)$$ then there is a representative in the equivalence class such that $id_{h_i,k_r}(h_j)=k_s$, whereas if $$-(z_i-z_j) = (w_r-w_s)$$ there is a representative such that $\tau_{h_i,k_r}(h_j)=k_s$.

    In particular, for a fixed pair of half-edges $h\in H(v)$, $h'\in H(\nu(v))$, viewing the pairing $id_{h,h'}$ as $\gamma^+$ and $\tau_{h,h'}$ as $\gamma^-$, admissibility implies that the bipartite graph condition holds.
\end{definition}

    \begin{definition}\label{admeq0}
    We consider two sets of half-edge pairings at vertices of genus $0$ \textit{equivalent} if the corresponding geometric strata $P_0$ as described in Definition \ref{zero} are the same. An equivalence class of relative half-edge pairings is said to be \textit{admissible} if the corresponding geometric stratum $P_0$ is nonempty.
\end{definition}

A combinatorial Torelli pair with the above data will be denoted by a tuple \begin{equation}\label{eqn:Teqn}
    (T_1,T_2,\nu,\sigma, \gamma^+,\gamma^-,(id_{h,h'})_{(h,h')}, (\tau_{h,h'})_{(h,h')},(S_{v,v'})_{(v,v')}).
\end{equation}
Two such pairs are considered equivalent if and only if the choices in $1-6$ agree/are equivalent according to Definitions \ref{admeq}, \ref{admeq0} up to relabelling vertices. Nonequivalent CTPs will parametrize disjoint loci in the fiber product.

\begin{remark}
    Due to the many choices in 1.1.5 and 1.1.6, and the difficulty in parametrizing these loci in a practical way, it is more convenient to restrict to studying the loci of the Torelli fiber product which either do not have or admit very few nontrivial conditions on half-edges imposed by 1.1.5 or 1.1.6.
\end{remark}

\subsection{The geometric strata}\label{geostrata}
 Given a combinatorial Torelli pair $\underline{T}$ as in Definition \ref{combtor}, we associate a geometric stratum, denoted by $\mathcal{F}(\underline{T})$, of the fiber product $\F_g$. We denote the closure of this stratum in $\F_g$ by $\mathcal{F}^{ct}(\underline{T})$. We will find a space, denoted $\mathcal{P}^{ct}(\underline{T})$, parametrizing  $\mathcal{F}^{ct}(\underline{T})$, and indicate which dense open subset $\mathcal{P}(\underline{T})\subset \mathcal{P}^{ct}(\underline{T})$ parametrizes $\mathcal{F}(\underline{T})$. We will explain why different combinatorial strata correspond to different geometric strata and why every geometric stratum is associated to a combinatorial stratum.

Pick a CTP $\underline{T}$ of the form \eqref{eqn:Teqn}. We describe a procedure to construct the space $P^{ct}(\ul{T})$. Consider the set $V^+$ of vertices of $T_1$ of positive genus. Let $Z_i$ be the set of vertices of $T_i$ of genus $0$ for $i=1,2$. We will define $P^{ct}(\ul{T})$ as a subset of a product \begin{equation}\label{prodct} P^{ct}(\uT) \subset \prod_{v\in V^+} P^{ct}_v \times \prod_{v_1\in Z_1} \M^{ct}_{0,d(v_1)} \times \prod_{v_2\in Z_2} \M^{ct}_{0,d(v_2)}\end{equation} for some spaces $P^{ct}_v$ to be defined.

In order to obtain a map to $\F_g$, we will consider two different maps $f,g$ from $P^{ct}(\uT)$ to $\M_g^{ct}$ and an identification of the images in $\A_g$.

Composing with the gluing maps $$\bt\xi_{T_i}\colon \mathcal{M}^{ct}_{T_i}=\prod_{v\in V(T_i)} \M^{ct}_{v,d(v)}\to \M_g^{ct}, \et$$ we can instead construct $P^{ct}(\uT)$ as a space mapping to $\prod_{v\in V(T_i)} \M^{ct}_{v,d(v)}$ for $i=1,2$, using maps $f_{v}$ for $v\in V(T_1)$ and $g_{v}$ for $v\in V(T_2)$. The map to $\F_g$ is induced by the maps $\xi_{T_1}\circ \prod_{v\in V(T_1)} f_v$ and $\xi_{T_2}\circ \prod_{v\in V(T_2)} g_v$.

For each $v\in V(T_1)$, we will construct maps, also denoted $f_v$, $g_v$, from $P_v^{ct}$ to  $\M^{ct}_{v,d(v)}$ and $\M_{\nu(v),d(\nu(v))}^{ct}$ respectively. We let $f_{v_1}$, $g_{v_2}$ be projection maps from $\prod_{v_1\in Z_1} \M^{ct}_{0,d(v_1)} \times \prod_{v_2\in Z_2} \M^{ct}_{0,d(v_2)}$ to $\M^{ct}_{v_i,d(v_i)}$ for $v_i\in Z_i$, $i=1,2$.

The map from $P^{ct}(\uT)$ will be induced by the composition of the projections to a factor of (\ref{prodct}) with the lastmentioned maps $f_v,g_v$, followed by the gluing morphisms $\xi_{T_1}$, $\xi_{T_2}$.

We start by defining $P^{ct}_v$ and $f_v,g_v$ for $v \in V^+$.

\begin{definition}\label{Pv}
Consider a vertex $v\in V(T_1)$ of $g(v)\geq 2$. 

\textbf{Case 1 (positive sign):} Assume $\sigma(v) = +$. Let $m=|H_1^+|=|H_2^+|$ be the number of paired adjacent half-edges. Let $n_1=|H(v)|-m$, $n_2=|H(\nu(v)|-m$ be the number of half-edges adjacent to $v$, $\nu(v)$ which are not paired. Define $$P_v^{ct} = \Mct_{g(v),m+n_1+n_2}$$ and let $$P_v = \M_{g(v),m+n_1+n_2}\backslash \Hy\subset P_v^{ct}$$ where $\Hy$ denote the sublocus of hyperelliptic curves.

We define $$\bt f_v\colon P_v^{ct}\to \Mct_{g(v),d(v)}\et$$ by forgetting the $n_2$ last markings, and $$\bt g_v\colon P_v^{ct}\to \Mct_{g(v),d(\nu(v))}\et$$ by forgetting the markings $m+1,\dots,m+n_1$. The identification $$\bt t\circ f_v(C,p_1,\dots,p_{m+n_1+n_2}) \ar[r,"\sim"]&t\circ g_v(C,p_1,\dots,p_{m+n_1+n_2}) \et$$ is given by $\bt +\colon J(C)\ar[r,"\sim"]&J(C).\et$ Here we are using the fact that the Torelli map from $\M_T$ factors through the product over $v\in V(T)$ of Torelli maps from $\Mct_{v,d(v)}$. 
   
\textbf{Case 2 (negative sign):} Assume $\sigma(v) = -$. Define $P_v^{ct}$ and $P_v$ similarly for $\sigma(v)=-$. The maps $f_v,g_v$ are defined as before, but the identification of $t\circ f_v$ and $t\circ g_v$ is now given by $\bt -\colon J(C)\ar[r,"\sim"]&J(C).\et$
   
\textbf{Case 3 (both signs):} Assume $\sigma(v) = \pm$. Let $c$ be the number of components of the bipartite graph associated to $\gamma^+,\gamma^-$ as in Definition \ref{combtor}.4 which are not cycles of length $2$. Let $t$ be the number of cycles of length $2$. 

Let $P^{ct}_v = \mathcal{H}^{ct}_{g(v), t, 2c}$ be the partial compactification defined in \cite[Notation 4.14]{SvZ} of the stack of hyperelliptic curves of genus $g(v)$ with $t$ fixed distinct Weierstrass points $w_1,\dots, w_t$ and $2c$ distinct points $(p_1, \bar{p}_1,\dots p_c, \bar{p}_c)$ where $\bar{p}$ denotes the hyperelliptic conjugate of a point $p$. Moreover we have restricted to the locus of compact type curves. An element in $\mathcal{H}^{ct}_{g(v), t, 2c}$ is a stable compact type curve which has an admissible cover of degree $2$ onto a genus $0$ curve. The markings $w_1,\dots,w_t$ on the curve correspond to points which are fixed under the  involution whereas the pairs $(p_i,\bar{p}_i)$ correspond to pairs of interchanged points. In particular, the components of $g> 0$ are not exchanged under the involution. See Figure \ref{samphyp} for an example.

\begin{figure}[h]
\centering
\scalebox{0.7}{\input{samplehyp}}
\caption{Sample element of $\mathcal{H}^{ct}_{12, 3, 2\cdot 5}$.}\label{samphyp}
\end{figure}
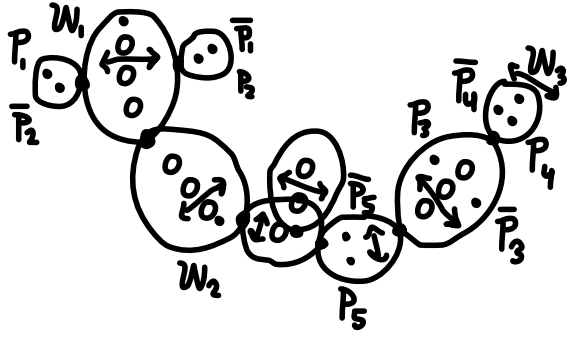
  
We define $$P_v =\mathcal{H}_{g(v), t, 2c} \subset P^{ct}_v.$$ The maps $\bt f_v\colon P^{ct}_v \to \Mct_{g(v),d(v)}\et$ and $\bt g_v \colon P^{ct}_{v} \to \Mct_{g(\nu(v)),d(\nu(v))} \et $ are defined as follows: An element $$(C,w_1,\dots,w_t, (p_1,\bar{p}_1),\dots, (p_c,\bar{p}_c))$$ is sent by $f_v$ to the curve $(C,q_1,\dots, q_{d(v)})$ where in each step, starting from $i=1$, $q_i$ is determined as follows: 

\begin{itemize}
    \item If the vertex $h_i$ corresponding to the $i^{th}$ half-edge at $v$ lies in a $2$-cycle, let $q_i = w_k$ for the least $k$ such that $w_k$ has not been chosen before.
    \item  If $h_i$ lies on a component different from $h_1,\dots, h_{k-1}$, and does not lie in a $2$-cycle, let $q_i = p_k$ for the least $k$ such that $p_k$ has not been chosen.
    \item If $h_i$ is connected to $h_j$ for $j<i$ by a path of length $2$, and $q_j = p_k$, then we let $q_i = \bar{p}_k$.
\end{itemize} 

The map $g_v$ instead sends $$(C,w_1,\dots,w_t, (p_1,\bar{p}_1),\dots, (p_c,\bar{p}_c))$$ to $(C,r_1,\dots, r_{d(\nu(v))})$, where $r_i$ are defined for $i=1,2,\dots,d(\nu(v))$ as follows: 

\begin{itemize}
    \item If the vertex $h_i'$ corresponding to the $i^{th}$ half-edge at $\nu(v)$ is not connected to any $h_j$, then $r_i = p_k$ for the least $k$ such that $p_k$ is not among the $q_i$ or among the $r_j$ for $j<i$.
    
    \item If $h'_i$ is connected by a blue edge with $h_j$ where $q_j = p_k$, then $r_i = q_j = p_k$.

    \item If $h'_i$ is connected by a red edge with $h_j$ where $q_j=p_k$, then $r_j =\bar{q}_j = \bar{p}_k$.

    \item If $h'_i$ is connected with $h_j$ where $q_j = w_k$ we let $r_i = w_k$.
\end{itemize}

See Figure \ref{bipmap} for an example. We choose the identification $\bt t\circ f_v \ar[r,"\sim"] &t\circ g_v \et$ to be $+$ on Jacobians.

\end{definition}

\noindent Note that, if the vertex corresponding to $r_i$ has two edges, the choice of $r_i$ is independent of the choice of edge used to define it.

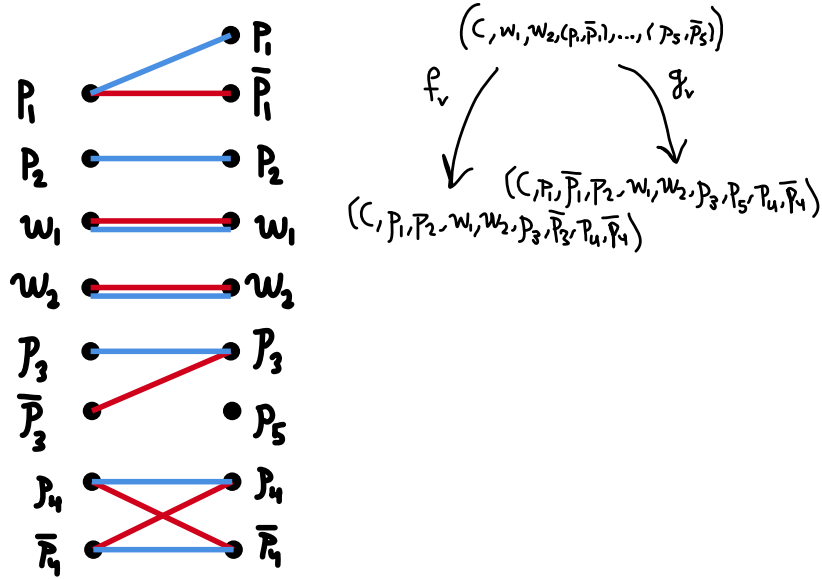
\begin{figure}[h]
    \centering
    \scalebox{0.7}{\input{bipartmap.tex}}
    \caption{An illustration of the maps $f_v,g_v$ inferred from the bipartite graph associated to $\gamma^+,\gamma^-$.}\label{bipmap}
\end{figure}

\begin{definition}\label{one}
For $v\in V(T_1)$ with $g(v) = 1$, we describe parametrizations of the spaces $P_v$ and $P_v^{ct}$, given a set of relative half-edge pairings as in Definition \ref{combtor}.5. Choose a pair of half-edges $h\in H(v)$, $h'\in H(\nu(v))$ and construct a bipartite graph as in Definition \ref{Pv} by taking the left set of vertices to correspond to the set $H(v)$, the right set to be $H(\nu(v))$, creating a blue edge between $g,g'$ for each relation $id_{h,h'}(g) = g'$ and a red edge for each $\tau_{h,h'}(g) = g'$. In particular, $h,h'$ are joined by both a red and a blue edge. Consider the corresponding space $\Hy^{ct}_{1,t,2c}$ constructed from this graph as in Definition \ref{Pv}. Let $P_v\subset \Hy_{1,t,2c}$ be the subset of curves $(C,w_1,\dots,w_t,(p_1,\bar{p}_1),\dots, (p_c,\bar{p}_c))$ such that:

\begin{itemize}
    \item For each set $(g,k)\in H(v)$, $(g',k')\in H(\nu(v))$ of half-edges such that $id_{g,g'}(k) = k'$, the marked points $p,p',q,q'$ associated to $g,g',k,k'$ respectively satisfy $q'-q = p'-p$, and for each $\tau_{g,g'}(k) = k'$ instead $q'-q = p-p'$.

    \item For $(g,k)\in H(v)$, $(g',k')\in H(\nu(v))$, if $q'-q = p'-p$ there is a representative in the given equivalence class of relative half-edge pairings such that $id_{g,g'}(k) = k'$. If $q'-q = p-p'$, there is a representative such that $\tau_{g,g'}(k) = k'$.
\end{itemize}

 Let $P^{ct}_v\subset \Hy^{ct}_{1,t,2c}$ be the closure of $P_v$, and let the maps $f_v$, $g_v$ be induced from the maps from $\Hy^{ct}_{1,t,2c}$ defined in Definition \ref{Pv}. The identification $\bt t\circ f_v \ar[r,"\sim"] &t\circ g_v \et$ is taken to be $+$ on Jacobians with respect to the hyperelliptic involution taking the point corresponding to $h$ to the point corresponding to $h'$.
\end{definition}

\begin{example}
    For $v\in V(T_1)$ with $g(v) = 1$, $|H(v)| = m$ and $|H(\nu(v))| = n$, where the only conditions imposed by the relative half-edge pairings are $$id_{h,h'}(h) = \tau_{h,h'}(h) = h'$$ for $h\in H(v), h'\in H(\nu(v))$, then $P_v^{ct} = \Hy^{ct}_{1,t,2c}$ where $t = 1, c = m+n-2$. Equivalently, $P_v^{ct} = \M^{ct}_{1,m+n-1}$.
\end{example}

\begin{definition}\label{zero}
    Let $$P_0 \subset \prod_{v_1\in Z_1} \M_{0,d(v_1)} \times \prod_{v_2\in Z_2} \M_{0,d(v_2)}$$ be the largest subset such that the following conditions hold:

\begin{itemize}
    \item For $v_1\in Z_1$, $v_2\in Z_2$, let $(\Gamma,\Gamma_i,\bt \alpha_i\colon \Gamma \to \Gamma_i \et \text{ for }i=1,2)$ be a tuple in $S_{v_1,v_2}$ where $\Gamma$ is a stratum of $\Hy^{ct}_{0,t,2c}$. Let $\ul{C}_1 =(C_1, p_1,\dots, p_k)$,  $\ul{C}_2 = (C_2, q_1,\dots, p_m)$ be curves corresponding to the graphs $\Gamma_1$, $\Gamma_2$ respectively. Then there is a curve in $\Hy^{ct}_{0,t,2c}$ which maps $\ul{C}_i$ to under the forgetful maps induced by $\alpha_i$ for $i=1,2$.
    
    \item Let $\ul{C}_1$, $\ul{C}_2$ be curves corresponding to $v_1\in Z_1$ and $v_2\in Z_2$. If there is a curve $\ul{C}$ in some $\Hy^{ct}_{0,t,2c}$ which maps $\ul{C}_i$ to under forgetful maps, then there is a representative in the equivalence class of half-edge pairings containing the tuple $$(\Gamma,\Gamma_i,\bt \alpha_i\colon \Gamma \to \Gamma_i \et \text{ for }i=1,2)\in S_{v_1,v_2}$$ \noindent where $\Gamma_i$ is the stable graph of $\ul{C}_i$ and $\Gamma$ the stable graph of $\ul{C}$.
\end{itemize}
     
Let $$P^{ct}_0 \subset \prod_{v_1\in Z_1} \M^{ct}_{0,d(v_1)} \times \prod_{v_2\in Z_2} \M^{ct}_{0,d(v_2)}$$ be the closure of $P_0$. The maps $f_v$, $g_v$ to spaces $\Mct_{v,d(v)}$ for $v\in Z_1$, $Z_2$ from $P^{ct}_0$ are defined by inclusion into $\prod_{v_1\in Z_1} \M^{ct}_{0,d(v_1)} \times \prod_{v_2\in Z_2} \M^{ct}_{0,d(v_2)}$ followed by projection onto the relevant factor.
\end{definition}

\begin{example}
    Assume $Z_1 =\{v\}$ and $Z_2 =\{w\}$. If $d(v) = d(w) = 3$, then $P_0 = P_0^{ct} = \{\text{pt}\}$. Assume next that $d(v) = d(w) =4$, where $H(v) = \{h_1,\dots,h_4\}$ and $H(w) = \{h'_1,\dots,h'_4\}$. Assume that $S_{v,w}$ has one element, corresponding to a graph $\Gamma$ with one vertex and four marked non-Weierstrass pairs $(p_j,\bar{p}_j)$ for $j=1,\dots, 4$, with $\alpha_i$ forgetting the points $\bar{p}_j$ and mapping $p_j$ to $h_j$ resp. $h'_j$ for $i=1,2$ and $j=1,\dots,4$.
    
    Then $P_0 \subset \M_{0,4}$ consists of crossratios $$\lambda = (h_1,h_2,h_3,h_4) = (h_1',h_2',h_3',h_4')\in \C\backslash\{0,1\}$$ for which $$\lambda \notin S_{\lambda} =\Bigl\{\frac{1}{\lambda}, 1-\lambda, \frac{1}{1-\lambda}, \frac{\lambda}{\lambda-1}, \frac{\lambda-1}{\lambda}\Bigr\}.$$ Indeed, $\lambda \in S_\lambda$ corresponds to imposing further pairings between the half-edges with different orderings: Moreover, if we had added e.g. the assumption that two points out of $h_1,\dots,h_4$ are Weierstrass with respect to the same involution, we would obtain $\lambda \in S_{-1}$ (noting that $(0,1,\infty,-1) = -1$). This, however, would also imply that $\lambda \in S_\lambda$. 
    
    In our case, we have $P^{ct}_0 = \Mct_{0,4}$.
    
\end{example}

We are now ready to define $P(\ul{T})$ and $P^{ct}(\ul{T})$.

\begin{definition}
   Using the definitions \ref{Pv}, \ref{one} and \ref{zero}, we define \begin{equation*}P(\uT) = \left(\prod_{v\in V^+} P_v\right) \times P_0\end{equation*} and \begin{equation*}P^{ct}(\uT) = \left(\prod_{v\in V^+} P^{ct}_v \right)\times P^{ct}_0,\end{equation*} viewing $P(\ul{T})$ as a subset of $P^{ct}(\ul{T})$.
\end{definition}

The maps $f,g$ from $P^{ct}(\ul{T})$ to $\prod_{v\in V(T_i)} \M^{ct}_{g(v),d(v)}$ for $i=1,2$ are as described before Definition $\ref{Pv}$. The identifications of the images $\bt t\circ f \ar[r,"\sim"] &t\circ g\et$ are as described in Definitions \ref{Pv}, \ref{one}.

Thus, having obtained a map from $P^{ct}(\ul{T})$ to the fiber product $\F_g$, we can define $\F(\ul{T})$, $\F^{ct}(\ul{T})$ to be the images of  $P(\ul{T})$, $P^{ct}(\ul{T})$ in $\F_g$ using the maps $f,g$.

\begin{proposition}
    The combinatorial strata correspond to different geometric strata. Moreover, every point in the fiber product $\F_g$ lies in a geometric stratum associated to a combinatorial stratum.
\end{proposition}  

\begin{proof}
Note that the choices of graphs $T_1$, $T_2$ correspond to the topological shapes of the pairs of curves $(C_1,C_2)\in \F_g$. The vertex pairings correspond to the choices of pairings of Jacobians of the irreducible components of genus $\geq 1$. The sign choice $\pm$ on a vertex indicates that the curve correspinding to this vertex is hyperelliptic, whereas $+$ or $-$ says that the corresponding curve is not hyperelliptic, and moreover the automorphism of Jacobians of the curves corresponding to $v,\nu(v)$ is given by $+$ resp. $-$. Different equicalence classes of half-edge pairings at a vertex $v\in V(T_1)$ correspond to different choices of identifications of the nodes at the curve associated to $v$ with those at the curve associated to $\nu(v)$ (if $g(v)\geq 1$) or to various $v'\in Z_2$ (if $g(v) =0$).\end{proof}

\subsection{Specializations of combinatorial strata}
Next, we describe the specializations of combinatorial strata, which will correspond to classifying the strata contained in the closure of a given geometric stratum.

\begin{definition}\label{spec}
    A \textit{specialization} of a CTP $$\underline{T} = (T_1,T_2,\nu,\sigma, \gamma^+,\gamma^-,(id_{h,h'})_{(h,h')}, (\tau_{h,h'})_{(h,h')},(S_{v,v'})_{(v,v')})$$ is a CTP $\underline{T}' = (T'_1,T'_2,\nu',\sigma', \gamma'^+,\gamma'^-,(id'_{h,h'})_{(h,h')}, (\tau'_{h,h'})_{(h,h')},(S'_{v,v'})_{(v,v')})$, together with a pair of specializations of stable graphs $\bt \varphi_i \colon T_i' \to T_i \et$ for $i=1,2$, obtained from $\underline{T}$ by a composition of the following operations:
    \begin{enumerate}

    \item \textbf{(Add a sign)} For a vertex $v\in V(T_1)$ of $g(v)>2$, we can replace a sign $+$ or $-$ by $\pm$.
    
    \item \textbf{(Add half-edge pairings, $g(v)\geq 2$)} For $v\in V(T_1)$, $g(v)\geq 2$, $h\in H(v)\backslash H_1^+$, $h' \in H(\nu(v))\backslash H_2^+$, we can add $h$ to $H_1^+$, $h'$ to $H_2^+$ and pair $h,h'$ under $\gamma^+$ (unless $\sigma(v) = -$) as long as the bipartite graph condition is still satisfied. In the same way, if $\sigma(v) \neq +$, the pairing $\gamma^-$ can be extended by a pair of half-edges such that the bipartite graph condition still holds. 

    \item \textbf{(Add half-edge pairings, $g(v)=1$)} For a vertex $v\in V(T_1)$ of genus $1$, we can extend a pairing $id_{h,h'}$ or $\tau_{h,h'}$ by two half-edges given that the obtained stratum is admissible according to Definition \ref{admeq}.

    \item \textbf{(Add half-edge pairings, $g(v_i)=0$)} For vertices $v_i\in V(T_i)$ for $i=1,2$ of genus $0$, we can add a tuple $$(\Gamma,\Gamma_i,\bt \alpha_i\colon \Gamma \to \Gamma_i \et \text{ for }i=1,2)$$ to $S_{v_1,v_2}$ given that the obtained CTP is admissible as in \ref{admeq0}.
    
    \item \textbf{(Insert trees)} We can replace a vertex $v \in V(T_1)$ (together with its adjacent half-edges) of $g(v)\geq 1$ and the corresponding $\nu(v) \in V(T_2)$ (with its half-edges) by marked trees $T_v$ and $T_{\nu(v)}$ respectively, or simultaneously replace each of the vertices $v\in Z_1\cup Z_2$ by a graph $\Gamma_{v}$, such that the following conditions are met and data is added:
    
    \begin{enumerate}

   \item If $g(v)\geq 2$ and $\sigma(v) = +$ or $-$, pick a graph $\Gamma$ associated to a combinatorial stratum of $\Mct_{g(v),m+n_1+n_2}$ as in Definition \ref{Pv}. Then $T_{v}$ should be the graph obtained by forgetting the last $n_2$ half-edges, and $T_{\nu(v)}$ the graph obtained by forgetting the half-edges $m+1,\dots, m+n_1$. If instead $\sigma(v) = \pm$, consider the associated space $\Hy^{ct}_{g,t,2c}$ obtained from the bipartite graph associated to the pairings $\gamma^+,\gamma^-$ as in \ref{combtor}.4. Pick a combinatorial stratum of this space, corresponding to a stable graph $\Gamma$ with $t+2c$ marked half-edges. Then $T_{v}$ should be obtained by forgetting the half-edges which do not come from half-edges adjacent to $v$. Similarly, $T_{\nu(v)}$ should be obtained by forgetting those half-edges which do not come from $\nu(v)$.
   
   \item If $g(v)=1$, pick a choice of half-edges $h\in H(v)$ and $h'\in H(v')$ and consider the corresponding space $\Hy^{ct}_{1,t,2c}$ as in Definition \ref{Pv}. Pick a graph $\Gamma$ corresponding to a combinatorial stratum of $\Hy^{ct}_{1,t,2c}$ whose geometric stratum has a nonempty intersection with $P^{ct}_v$, defined in Definition \ref{one}, inside $\Hy^{ct}_{1,t,2c}$. Then $T_{v}$ should again be obtained by forgetting the half-edges not coming from half-edges adjacent to $v$, and $T_{\nu(v)}$ should be obtained analogously.

    \item We consider specializations of the genus $0$ vertices simultaneously. Pick a set of graphs $(\sqcup_{v_1\in Z_1}\Gamma_{v_1})\bigsqcup (\sqcup_{v_2\in Z_2}\Gamma_{v_2})$ corresponding to a combinatorial stratum of $$\prod_{v_1\in Z_1}\Mct_{0,d(v_1)}\times \prod_{v_2\in Z_2}\Mct_{0,d(v_2)}$$ which has a nonempty intersection with $P_0^{ct}$ as defined in Definition \ref{zero}. For such a stratum, each vertex $v_1\in Z_1$ is replaced by graph of the corresponding stratum in the factor $\Mct_{0,d(v_1)}$ and similarly for $v_2 \in Z_2$.
   
   \item The forgetful maps $\bt \Gamma \to T_{v}\et$ and $\bt \Gamma \to T_{\nu(v)}\et$ induce identifications between the vertices of positive genus of $T_{v}$ and $T_{\nu(v)}$, which is how $\nu'$ is obtained from $\nu$.
   
   \item We impose the relations on half-edges which correspond to taking an element in the corresponding stratum $\Gamma$ (resp. $\sqcup_{v\in Z_1\cup Z_2} \Gamma_v$) of $P^{ct}_v$ (resp. $P^{ct}_0$), and using the forgetful maps $f_v$, $g_v$ to obtain a correspondence between the markings on the curves in the strata $T_v$, $T_{\nu(v)}$ (and the individual $\Gamma_{v}$'s for $v\in Z_1\cup Z_2$). 
   
   \hspace{0.5cm} Precisely, if $g(v)\geq 2$, for $w\in V(T_v)$ and $w'\in V(T_{\nu(v)})$ (if $g(w)\geq 1$, assume $w'=\nu'(w)$), we define pairings $\gamma'^+,\gamma'^-$ between $H(w)$, $H(w')$ as follows: Assume that $\sigma(v) = +$. Let $h\in H(w)$, $h'\in H(w')$ and consider their images (also denoted $h,h'$) in $\Gamma$ under the maps on half-edges $H(T_v) \to H(\Gamma)$, $H(T_{\nu(v)}) \to H(\Gamma)$ induced by the forgetful maps $\Gamma \to T_v$, $\Gamma \to T_{\nu(v)}$. Assume that there is a half-edge $l$ in $\Gamma$ among the $m$ first markings such that there are paths in $\Gamma$ starting at the attaching points of $h$ resp. $h'$ and ending in $l$. Then we pair $h$ with $h'$ using $\gamma'^+$, see Figure \ref{marks} for an example. We define $\gamma'^-$ similarly if $\sigma(v) = -$. If $\sigma(v) = \pm$, we identify $h\in H(w)$, $h'\in H(w')$ with $\gamma'^+$ using the rule just described. Denote the hyperelliptic involution of $\Gamma$ by $\tau$. We identify $h\in H(w)$ and $h'\in H(w')$ by $\gamma'^-$ if there are half-edge $l,l'$ associated to markings in a conjugate pair in $\Gamma$ such that there are paths from the attaching points of $h$, $h'$ ending in $l$, $l'$ respectively.
    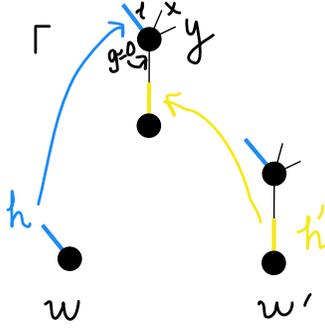
\begin{figure}[h]
    \centering
    \scalebox{0.7}{\input{marks.tex}}
    \caption{Example of a pair of half-edges $h\in H(w)$, $h'\in H(w')$ which are paired under $\gamma'^+$.}\label{marks}
\end{figure}

    \hspace{0.5cm} If $g(v) = 1$, the half-edge pairings between half-edges of $v$ and $\nu(v)$ induce half-edge identifications/pairings on $\Gamma$. Consider a pair of half-edges $(h,h')$ and $(g,g')$ where $h,g \in H(w)$ and $g,g'\in H(w')$. Assume there exist half-edges $f(h),f(h'),f(g),f(g')$ in $\Gamma$ such that $f(g') = id_{f(h),f(h')}(f(g))$ (or $\tau_{f(h),f(h')}(f(g'))$) and such that there is a path in $\Gamma$ starting at the attaching point of $h$ and ending at the tip of $f(h)$ and similarly for $h',g,g'$ (viewed as half-edges in $\Gamma$). Then we impose that $g' = id_{h,h'}(g)$ (or $\tau_{h,h'}(g')$).

    \hspace{0.5cm} If $g(v) = 0$, the relations as in \ref{combtor}.6 induce relations on the half-edges of $\sqcup_{v\in Z_1\cup Z_2} \Gamma_v$. A collection of half-edges at $w$ will then be related to a collection at $w'$ by a tuple $$(\Gamma,\Gamma_i,\bt \alpha_i\colon \Gamma \to \Gamma_i \et \text{ for }i=1,2)$$ as in \ref{combtor}.6 if the corresponding half-edges in $\sqcup_{v\in Z_1\cup Z_2} \Gamma_v$ have the same relation up to paths in $\Gamma$ as in the previous two paragraphs.

    \hspace{0.5cm} If $g(v) \geq 2$ and $g(w) = 1$, we translate the pairings $\gamma'^+$ resp. $\gamma'^-$ to pairings of the form \ref{combtor}.5 by imposing that $h'= \gamma'^+(h)$ and $g' = \gamma'^+(g)$ implies that $id_{h,h'}(g) = g'$, whereas $h'= \gamma'^-(h)$ and $g' = \gamma'^-(g)$ gives $\tau_{h,h'}(g) = g'$. Note that if $\sigma(v) =\pm$, there will always exist a pair $h,h'$ being paired under both $\gamma'^+$ and $\gamma'^-$ since the nodes which do not come from marked points in the specialization must be Weierstrass points.

    \hspace{0.5cm} If $g(v) \geq 2$, $\sigma(v)=\pm$ and $g(w) = 0$, we instead translate the pairings $\gamma'^+,\gamma'^-$ to a suitable map from a stratum $\Gamma'$ in  $\Hy^{ct}_{0,t,2c}$ for some $t$, $c$ where $t$ corresponds to Weierstrass points formed by half-edges which are paired both under $\gamma'^+$ and $\gamma'^-$. half-edges which are paired under $\gamma'^+$ should correspond to the same half-edges in $\Gamma'$, whereas those paired under $\gamma'^-$ should correspond to conjugate half-edges in $\Gamma'$. If instead $\sigma(v) = +$ or $-$, we instead choose $\Gamma'$ in  $\Mct_{0,l}$ and identify half-edges $h\in H(w),h'\in H(w')$ in $\Gamma'$ if they are paired under $\gamma'^+$ resp. $\gamma'^-$.

    \hspace{0.5cm} If $g(v) =1$ and $g(w) = 0$, for each $h\in H(w),h'\in H(w')$ we translate the pairings $id_{h,h'}$, $\tau_{h,h'}$ to relations of the form \ref{combtor}.6 by letting $id_{h,h'},\tau_{h,h'}$ act as $\gamma'^+, \gamma'^-$ as in the previous paragraph.
    
    \item For $w \in V(T_v)$ where $g(w)> 2$, $\sigma'(w) = \sigma(v)$ (if  $g(w)=2$, then $\sigma'(w) = \pm$).
    \end{enumerate}
    \end{enumerate}
    The specialization maps $\varphi_i$ for $i=1,2$ are induced by the chosen composition of operations of type 5.    
    \end{definition}

We say two specializations $(\ul{T}',\phi'_1,\phi'_2)$ and $(\ul{T}'',\phi'_1,\phi'_2)$ of $\ul{T}$ are equivalent if there are isomorphisms $\bt T_1' \ar[r,"\sim"] & T_1''\et $ and $\bt T_2' \ar[r, "\sim"] & T_2''\et $ making the specialization maps and the data associated to $\ul{T}'$ and $\ul{T}''$ commute.

\subsection{Parametrizing specializations of strata}\label{parstrata}
We will explain how the geometric stratum corresponding to a specialization $\underline{T}'$ of $\underline{T}$ is parametrized by a subset of $\mathcal{P}^{ct}(\underline{T})$. Conversely, we will see how each element of $\mathcal{P}^{ct}(\underline{T})$ maps to an element in  $\mathcal{F}(\underline{T}')$ for some specialization $\underline{T}'$ of $\underline{T}$. 

\begin{proposition}
    Let $\ul{T}'$ be a specialization of a CTP $\ul{T}$. Then $\F(\ul{T}')\subset \F^{ct}(\ul{T})$.
\end{proposition}

\begin{proof}
First, we consider the specialization given by $1$ in Definition \ref{spec}. Let $v\in V(T_1)$ be the vertex which receives a $\pm$ sign. Without loss of generality, we assume $\sigma(v) = +$. Let $m = |H_1^+|$, and let $n_1 = |H(v)\backslash H_1^+|$, $n_2 = |H(\nu(v))\backslash H_2^+|$.

The parametrization of $\F(\ul{T}')$ is obtained from that of $\F(\ul{T})$ by replacing $$(P_v = \M_{g(v),m+n_1+n_2}\backslash \Hy_{g(v),m+n_1+n_2},f_v,g_v)$$ by $(\Hy_{g(v),0,m+n_1+n_2},f'_v,g'_v)$ where $f'_v$ maps $\Hy_{g(v),0,m+n_1+n_2}$ to $$\Hy_{g(v),0,d(v)}\subset \M_{g(v),d(v)}$$ and $g_v'$ maps it to $$\Hy_{g(\nu(v)),0,d(\nu(v))}\subset \M_{g(\nu(v)),d(\nu(v))}.$$ Hence $\Hy_{g(v),0,m+n_1+n_2}$ is mapped under $f_v'$, $g_v'$ into the images $\M_{g(v),d(v)}^{ct}$ resp. $\M_{g(\nu(v)),d(\nu(v))}^{ct}$ of $$P_v^{ct} = \M_{g(\nu(v)),m+n_1+n_2}^{ct}$$ under $f_v,g_v$, showing that $\F(\ul{T}')\subset\F^{ct}(\ul{T})$.

Next, we consider type $2$. Let $v$ be the vertex in $V(T_1)$ which obtains an extra half-edge pairing. Let $v'$ be the corresponding vertex in the specialization $T_1'$. Assume first that $\sigma(v) =+$. Then this specialization corresponds to letting marked points among $H(v)\backslash H_1^+$ and $H(\nu(v))\backslash H_2^+$ agree, which corresponds to reducing $n_1,n_2$ by $1$ each and adding $1$ to $m$ (with $m$, $n_1$, $n_2$ defined as in the previous paragraph). A curve $$(C,p_1,\dots,p_{m+1},q_1,\dots,q_{n_1-1},r_1,\dots,r_{n_2-1}) \in P_{v'}$$ is mapped to $$(C,p_1,\dots,p_{m+1},q_1,\dots,q_{n_1-1})$$ under $f_{v'}$ and to $$(C,p_1,\dots,p_{m+1},r_1,\dots,r_{n_2-1})$$ under $g_{v'}$. This gives the same image in the fiber product as the image under $f_v,g_v$ of an element in $P^{ct}_v$ which consists of the curve $C$ with marked points $$(C,p_1,\dots,p_{m+1},q_1,\dots,q_{n_1-1},p_{m+1}, r_1,\dots,r_{n_2-1})$$ where the repeated $p_{m+1}$ corresponds to two markings lying on a bubble (attached genus $0$ curve) at the point $p_{m+1}$.

Similarly, when $\sigma(v)=\pm$, pairing two previously unpaired half-edges $h,h'$ under $\gamma^+$ will correspond to letting the associated pairs of conjugate points $(p,\bar{p}), (q,\bar{q})$ coincide via the identification $p=q$. This will create two bubbles in $P^{ct}_v$, one containing the markings $p,q$ and the other containing $\bar{p},q$. If we instead pair $h,h'$ via $\gamma^-$, the pairs $(p,\bar{p}), (q,\bar{q})$ are identified via $p=\bar{q}$. This corresponds to letting the points $p,\bar{q}$ lie on a bubble and $\bar{p},q$ on another. If paired half-edges $h,h'$ under $\gamma^+$ (or $\gamma^-$) further become identified under $\gamma^-$ (or $\gamma^+$), the corresponding tuple $(p,\bar{p})$ becomes a Weierstrass point and creates a bubble in $P^{ct}_v$ containing the points $p,\bar{p}$. The image of $P_{v'}$ will thus lie in the image of $P^{ct}_v$.

For specialization type $3$ (letting $v\in V(T_1)$ be the vertex in question), unless the specialization gives an empty locus, the image of $P_{v'}$ will be parametrized by elements of $P^{ct}_v$ which have the appropriate bubblings/relations between points imposed by the given additional pairing. For instance, if $id_{h,h'}(g) = g'$ for some $h,h',g,g'$, we impose $q'=p'-p+q$, creating a bubble if this means that $q'$ becomes identified with an existing point. If $\tau_{h,h'}(g) = g'$, we instead impose $q'=p-p'+q$.

For type 4, in the newly obtained stratum we have imposed an additional relation on the markings corresponding to the identifications arising from the forgetful maps $\alpha_i$. The parametrization will be a subset of $$\prod_{v_1\in Z_1} \M_{0,d(v_1)} \times \prod_{v_2\in Z_2} \M_{0,d(v_2)}$$ where the new requirement is imposed. This will lie in the closure $$P_0^{ct}\subset \prod_{v_1\in Z_1} \Mct_{0,d(v_1)} \times \prod_{v_2\in Z_2} \Mct_{0,d(v_2)}$$ of $P_0$ associated to $\ul{T}$ by allowing this relation (which was previously not allowed in $P_0$) for the points corresponding to $v_1,v_2$.

The last specialization type, 5, the choices of $T_v$ resp. $T_{\nu(v)}$ (resp. $\sqcup_{v\in Z_1\cup Z_2}\Gamma_{v}$) come from a topological stratum of $P_v^{ct}$ (resp. $P_0^{ct}$) by construction. The imposed conditions on the half-edges are induced by the maps $f_v,g_v$ from the corresponding stratum in $P_v^{ct}$ to the spaces $\Mct_{g(v),d(v)}$, $\Mct_{g(\nu(v)),d(\nu(v))}$.

For $w\in V(T_v)$ with $g(w) \geq 2$ we have $\sigma'(w) = \sigma(v)$. Thus the isomorphisms of Jacobians of curves coming from $T_v$, $T_{\nu(v)}$ will agree with the limits of automorphisms of Jacobians of the curves corresponding to $v,\nu(v)$.\end{proof}

\begin{proposition} For any $x\in \F^{ct}(\ul{T})$, $x$ lies in the image of $P(\ul{T}')$ for a specialization $\ul{T}'$ of $\ul{T}$ (and so $\F(\ul{T})\subset \F^{ct}(\ul{T})$).
\end{proposition}

\begin{proof} Since $x$ lies in the image of $P^{ct}(\ul{T})$, the point correspond to a pair of curves which have dual graphs and some specializations as constructed as in Definition \ref{spec}.5. Let $\ul{T}'$ be a CTP such that $x\in \F(\ul{T}')$. The pairings of Jacobians of the irreducible components of genus $\geq 2$ must be induced by those of the curves in $\F(\ul{T})$, hence the choice of $\nu'$ and $\sigma'$ as in \ref{spec}.5. Some components of the pair of curves in $x$ might become hyperelliptic, corresponding to the operation \ref{spec}.1. Moreover, additional identifications of markings in the closure of $\F(\ul{T})$ correspond to the operations \ref{spec}.2-4 (existing relations between markings will remain when taking the closure of $\F(\ul{T})$, which also means that the indentifications between half-edges in \ref{spec}.5 are imposed). In other words, all the conditions for $\ul{T}'$ to be a specialization of $\ul{T}$ are satisfied. \end{proof}

\subsection{Components of $\F_g$ and their intersections}\label{compint}

Given the specializations in \ref{spec}, we can now deduce what the irreducible components of $\F_g$ are. In particular, the open stratum $\ul{T}$ associated to an irreducible component contains no half-edge pairings and no $\pm$ signs on vertices of genus $g\geq 3$. Moreover, there are no vertices of genus $0$ in $T_1$ resp. $T_2$ and no two neighboring genus $1$ vertices in $T_1$ identified with neighboring vertices in $T_2$. In other words, the irreducible components can be described as follows:

An irreducible component $\mathcal{F}_{(T_1,T_2,\nu,\sigma)}$ of $\F_g$ is associated to a tuple 
$(T_1,T_2,\nu,\sigma)$ such that:

\begin{enumerate}
\item \textbf{(Positive genera)} $T_1$, $T_2$ are (stable) trees with vertices of positive genus.
\item \textbf{(Vertex pairing)} \begin{tikzcd}
    \nu \colon V(T_1) \ar[r,"\sim"]& V(T_2)\hspace{-5pt}
\end{tikzcd} is a bijection preserving genus.
\item \textbf{(Sign function)} $\sigma$ is defined as a function $$\begin{tikzcd}\sigma \colon \{v\in V(T_1)|g(v) \geq 2\}
\to \{-,\pm,+\}\hspace{-5pt}
\end{tikzcd}$$ such that $\sigma(v) = \pm $ if and only if $g(v)= 2$.
\item \textbf{(Elliptic pairs)} No two neighboring vertices of genus $1$ in $T_1$ become neighbors in $T_2$ under $\nu$.
\end{enumerate}

\begin{remark}
The positive genus condition ensures that the space $\mathcal{F}_{(T_1,T_2,\nu,\sigma)}$ does not lie in the closure of another component defined by graphs obtained by repeatedly contracting edges having one endpoint at a genus $0$ vertex in $T_1$ or $T_2$.

The vertex pairing describes how the factors of the Jacobians corresponding to different curves match up under the automorphism of Jacobians.

The sign function describes whether the automorphism of Jacobians is generically identity ($\sigma(v)= +)$, the involution ($\sigma(v)=-)$ or can be both because of hyperellipticity ($\sigma(v)=\pm$) on the factor corresponding to a given vertex $v$ of $T_1$.

The elliptic pairs condition ensures that the space $\mathcal{F}_{(T_1,T_2,\nu,\sigma)}$ does not lie in the closure of the component obtained by replacing two such neighboring genus $1$ vertices by a vertex of genus $2$.
\end{remark}

An irreducible component $\mathcal{F}_{(T_1,T_2,\nu,\sigma)}$ is precisely the closure $\F^{ct}(\ul{T})$ of a stratum associated to a CTP $\ul{T}$ which has the prescribed vertex pairing $\nu$ and sign choice $\sigma$, and no half-edge pairings (up to equivalence).

The dimension of an irreducible component $\F_{(T_1, T_2, \nu, \sigma)}$ is therefore

$$\text{dim } \F_{(T_1, T_2, \nu, \sigma)} = \sum_{v \in V(T_1)}(3g(v)-3 + 2d(v)-\mathbbm{1}(g(v) = 1))$$ where $g(v)$ is the genus and $d(v)$ the degree of the vertex $v$.

\subsubsection{Intersections of strata closures}\label{intstratacl}

Denote by $\F_{T_1,T_2}^{ct}(\ul{T})$ the image of $P^{ct}(\ul{T})$ inside the fiber product of Torelli maps from $\M^{ct}_{T_i}$ for $i=1,2$ and define $\F_{T_1,T_2}(\ul{T})$ analogously. The results from Section \ref{parstrata} can be refined as follows.

\begin{proposition}\label{intstrata}
    The fiber product of two strata closures $\F_{T_1,T_2}^{ct}(\ul{T})$ and $\F_{\Gamma_1,\Gamma_2}^{ct}(\ul{\Gamma})$ over $\F_g$ consists of a disjoint union of strata $$\bigsqcup_{(\ul{T}',\varphi_1,\varphi_2,\psi_1,\psi_2)}\F_{T'_1,T_2'}(\ul{T'})$$ for specializations $(\ul{T}',\varphi_1,\varphi_2)$ and $(\ul{T}',\psi_1,\psi_2)$ of $\ul{T}_1$ resp. $\ul{T}_2$.
\end{proposition}

\begin{proof} In addition to the results in Section \ref{parstrata}, an $S$-point of $\F^{ct}_{T_1,T_2}(\ul{T})$ for a connected scheme $S$ with a geometric point $x$ with image in $\F(\ul{T'})$ and generic point $\eta \in \F_{T_1,T_2}(\ul{T})$ determines specialization morphisms $\varphi_i$ for $i=1,2$; these are induced by the associated maps from $S$ to $\M^{ct}_{T_i}$ for $i=1,2$.
    
Conversely, given a specialization $(\ul{T}',\varphi_1,\varphi_2)$ of $\ul{T}$, we obtain morphisms to $\F^{ct}_{T_1,T_2}(\ul{T})$ resp. $\F^{ct}_{\Gamma_1,\Gamma_2}(\ul{\Gamma})$ determined by $\varphi_1,\varphi_2$.\end{proof}

In particular, Proposition \ref{intstrata} allows us to compute the intersections between irreducible components of $\F_g$, as well as self-intersections.

\section{The class of the Torelli cycle in genus $4$}\label{torellisec}
In this section, we use the geometry of the Torelli fiber product to show that the pullback $t^*T_4$ of the Torelli cycle in genus $4$ is given by $16\lambda_1$. Since $\text{Pic}(\A_g)$ is generated by $\lambda_1$, this verifies that the class of the Torelli cycle on $\A_4$ is given by $T_4 = 16\lambda_1$.

\subsection{Geometry of the Torelli fiber product in genus $4$}\label{geofour}

Since $\mathrm{dim} (\Mct_g) = 3g-3$ and $\mathrm{dim} (\A_g) = \frac{g(g+1)}{2}$, we have $\mathrm{dim} (\Mct_4) = 9$ and $\mathrm{dim} (\A_4) = 10$. The class $t^*T_4$ will thus be of codimension $1$ on $\Mct_4$. Hence, to calculate this class, we only need to consider contributions from strata of the fiber product $\F_4$ whose images in $\Mct_4$ are supported in codimension $\leq 1$. Consequently, we disregard strata of the fiber product arising from pairs of curves having more than one node. Since the hyperelliptic locus $\Hy_g$ inside $\M_g$ has codimension $g-2$, the closure of $\Hy_4$ inside $\Mct_4$ has codimension $2$. We can therefore also disregard strata arising from pairs of hyperelliptic curves of genus $4$.

The components contributing to the intersection class $t^*T_4$, with their reduced structure, are thus:

\begin{itemize}
    \item The diagonal components $\Delta^{+}$, $\Delta^{-}$ which parametrize pairs of isomorphic curves $$(C, C, \begin{tikzcd}
    + \colon J(C) \hspace{-5pt}\ar[r,shorten >=5pt, shorten <=5pt,"\sim"]& \hspace{-5pt}J(C)),
\end{tikzcd}\text{ } (C, C, \begin{tikzcd}
    - \colon J(C)\hspace{-5pt} \ar[r,shorten >=5pt, shorten <=5pt,"\sim"]& \hspace{-5pt}J(C))
\end{tikzcd}$$
         \noindent where $+$, $-$ are the identity and sign involutions of Jacobians.

 \hspace{0.5cm} Using the description in Section \ref{combtor}, these components are the closures of the strata corresponding to the combinatorial Torelli pairs $\ul{T}_{\Delta^+}$ and $\ul{T}_{\Delta^-}$ consisting of two paired vertices of genus $4$, with sign $+$ resp. $-$ for $\sigma$. There are no half-edge pairings in this case. 
    
    \hspace{0.5cm} The open stratum in $\Delta^+$ resp. $\Delta^-$ corresponding to pairs of smooth nonhyperelliptic curves is parametrized by $\M_4\backslash\Hy_4$. The full components $\Delta^+$ and $\Delta^-$ are both parametrized by $\M_4^{ct}$. In fact, projection to $\Mct_4$ via $p_1$ defines isomorphisms \newline $\Delta^{+},\Delta^-\cong \Mct_4$ and hence $\mathrm{dim}(\Delta^\pm)= 9$.

    \item The (1,3)-components $A^+,A^-$ which parametrize pairs of curves consisting of a genus $1$ curve glued to a genus $3$ curve. These components can be parametrized by \newline $\Mct_{1,1}\times \Mct_{3,2}$ via $$([(E,p)],[(C,q,r)])\mapsto (E\cup_{p\sim q}C, E\cup_{p\sim r}C, \begin{tikzcd}
    \pm \colon J(C) \ar[r,shorten >=5pt, shorten <=5pt,"\sim"]& J(C)),
\end{tikzcd}$$ where the choice of automorphism of $J(C)$ is $+$ for $A^+$ and $-$ for $A^-$. The automorphism of $J(E)$ can be taken as $+$. We have $\dim(A^\pm)= 9$.
    
    \hspace{0.5cm} These components correspond to the combinatorial strata $\ul{T}_{A^+}$, $\ul{T}_{A^-}$ given by two trees which each has a genus $1$ and a genus $3$ vertex adjoined by an edge, equipped with the unique pairing between vertices of same genus, and letting $\sigma$ be $+$ for the genus $3$ vertex for $A^+$ and $-$ for $A^-$. Again, there are no half-edge pairings in these strata.

    \hspace{0.5cm} The geometric strata $\F(\ul{T}_{A^+})$ resp. $\F(\ul{T}_{A^-})$ are parametrized by, and equal to,\newline $\M_{1,1} \times \M_{3,2}\backslash \Hy_{3,2}$. The closures of these strata are parametrized by, but differ at the boundary from, $\M_{1,1}^{ct} \times \M_{3,2}^{ct}$.

    \item The (2,2)-component $B$ which consists of pairs of curves with two genus $2$ components glued at a node. Genus $2$ curves are hyperelliptic, so the automorphisms of Jacobians can be taken to be $+$. 
    
    \hspace{0.5cm} The combinatorial stratum $\ul{T}_{B}$ is given by a pair of graphs of two vertices of genus $2$ connected by an edge, and a given pairing between the vertices. There are no half-edge pairings.

    \hspace{0.5cm} The geometric stratum $\F(\ul{T}_{B})$ is parametrized by $\M_{2,2} \times \M_{2,2}$, and is equal to the quotient of this space by the action of the symmetric group $S_2$ on the factors. Geometric points of $\F(\ul{T}_{B})$ correspond to pairs of curves $$(C\cup_{p\sim r} D, C\cup_{q\sim s}D)$$ with $p\neq q, r\neq s$ and isomorphisms of Jacobians corresponding to the identity on $C,D$. 
    
   \hspace{0.5cm} The closure $B = \F^{ct}(\ul{T}_{B})$ is parametrized by $(\Mct_{2,2} \times \Mct_{2,2})/S_2$, and is $10$-dimensional.
\end{itemize}

The intersections contributing to $t^*T_4$ are:

\begin{itemize}
    \item $ Z_1 = \Delta^+\cap A^+$, which is the union of strata corresponding to common specializations of $\ul{T}_{A^+}$ and $\ul{T}_{\Delta^+}$. The only relevant such stratum is given by the CTP which consists of a pair of trees each having one vertex of genus $1$ adjoined to another of genus $3$, with the sign of $\sigma$ for the genus $3$ curve being $+$, and with a pairing $\gamma^+$ of the half-edges at the respective genus $3$ curves.

    \hspace{0.5cm} The corresponding geometric stratum equals $\M_1\times \M_{3,1}\backslash \Hy_{3,1}$. Its closure sits inside $\Delta^+$ as $\Mct_{1,1}\times \Mct_{3,1}\subset \Mct_4$, and is parametrized by the locus of \newline $\Mct_{1,1}\times \Mct_{3,2}$ of $A^+$ where the marked points in the second factor agree under the two forgetful maps. We have $\mathrm{dim}(Z_1)=8$.

    \item $ Z_2 = \Delta^-\cap A^-$, which has dimension $8$ and admits an analogous parametrization to that of $Z_1$.

    \item $  Z_3 = \Delta^+\cap B$ of dimension $8$, parametrized by $(\Mct_{2,1}\times \Mct_{2,1})/S_2$. The corresponding geometric stratum consists of points $$(C\cup_{p\sim r} D, C\cup_{p\sim r}D).$$

    \item $ Z_4 = \Delta^-\cap B$ of dimension $8$, parametrized by $(\Mct_{2,1}\times \Mct_{2,1})/S_2$. The open stratum consists of points $$(C\cup_{p\sim r} D, C\cup_{\bar{p}\sim \bar{r}}D)$$ where $\bar{p},\bar{r}$ denote the hyperelliptic conjugates of $p,r$.
\end{itemize}

\subsection{Intersection theory}\label{inttheory}
Each of the components $\Delta^\pm, A^\pm, B$ of the Torelli fiber product has a canonical contribution to the $8$-dimensional class $t^*T_4$. This contribution is the top Chern class of a corresponding excess bundle. The remaining contributions come from the intersections $Z_i$ for $1\leq i \leq 4$ and from nonreduced loci $Z_5, Z_6$ in the component $B$. The nonreduced scheme structure is determined in Section \ref{nonred2}.

For a morphism of varieties $f\colon X \to Y$, we denote by $N_f$ or $N_{X/Y}$ the sheaf $f^*T_Y/T_X$ and refer to it as the \textit{normal sheaf} of $f$.

Recall the diagram
$$\begin{tikzcd}
     \F_4\ar[d,"p_1"]\ar[r,"p_2"] & \Mct_4\ar[d,"t_2"]\\
    \Mct_4 \ar[r,"t_1"] & \A_4
\end{tikzcd}$$ where $t_1$, $t_2$ are Torelli maps.

Proposition \ref{cancomp} gives a canonical decomposition of $t^*T_4$.

\begin{proposition}\label{cancomp} The class $t^*T_4$ decomposes as
\begin{equation}\label{class}
\begin{aligned}
{p_1}_*\left[2c_1\left(\frac{p_1^*N_{\Mct_4/\A_4}}{N_{\Delta^+/\Mct_4}}\right)+2c_1\left(\frac{p_1^*N_{\Mct_4/\A_4}}{N_{A^+/\Mct_4}}\right)+c_2\left(\frac{p_1^*N_{\Mct_4/\A_4}} {N_{B/\Mct_4}}\right) + \sum_{i=1}^6 m_i[Z_i] \right]   
\end{aligned}
\end{equation}
\noindent for some multiplicities $m_i$.
\end{proposition}

\begin{proof} Consider the composite fiber diagram

$$\bt \F_4 \ar[r,"(p_1\text{,}p_2)"]\ar[d,"p_1"]& \Mct_4\times \Mct_4\ar[d,"(\text{id,}t_2)"]\\
\Mct_4 \ar[d,"\text{id}"]\ar[r, hookrightarrow,"\hat{\gamma}_{t_1}"] & \ar[d,"i"]\Mct_4\times t_2(\Mct_4) \\
\Mct_4 \ar[r,hookrightarrow,"\gamma_{t_1}"] &\Mct_4\times \A_4  \et$$ where $i$ is the inclusion and $\gamma_{t_1}$ is the graph morphism of $t_1$. Note the equality ${t_2}_*[\Mct_4]=2[t_2(\Mct_4)]$ since the Torelli map is generically $2:1$ on stacks.

We are interested in the class $t^*T_4 = t^*t_*[\Mct_4] = [\Mct_4]\cdot_{t}t_*[\Mct_4]$, which by \cite[Corollary 8.1.2]{Ful} is same as $$[\Gamma_{t_1}]\cdot ([\Mct_4]\times {t_2}_*[\Mct_4])=\gamma_{t_1}^!(\text{id},{t_2})_*([\Mct_4]\times[\Mct_4])$$ where $\Gamma_{t_1}$ is the graph of $t_1$ inside $\Mct_4\times \A_4$. By \cite[Theorem 6.2]{Ful}, using that $N_{\gamma_{t_1}} = {t_1}^*T\A_4$,  this class equals $${p_1}_*(\gamma_{t_1}^!([\Mct_4]\times[\Mct_4])) = {p_1}_*\left[c(p_1^*t_1^*T\A_4)\cap s(C_{\F_4/\Mct_4\times \Mct_4})\right]_8$$ where $C$ denotes the normal cone.

Let $Z$ be the closed subscheme  $\bigcup_{1\leq i \leq 6} Z_i$ of $ \F_4$. Over the locus $U= \F_4\backslash Z$, letting $\bt j\colon U \to \F_4 \et$ be the open immersion,
the normal cone $C_{U/\Mct_4\times \Mct_4} = j_U^*C_{\F_4/\Mct_4\times \Mct_4}$ splits as a union $$C_{U/\Mct_4\times \Mct_4} = \bigcup_{X\in \{\Delta^\pm,A^\pm,B\}} N_{X\cap U/ \Mct_4\times \Mct_4}$$ of normal bundles of the irreducible components of $\F_4$ restricted to $U$.

The pullback of the Segre class $s(C_{\F_4/\Mct_4\times \Mct_4})$ to $U$ equals $\sum_{X\in \{\Delta^\pm,A^\pm,B\}} c(N_{X\cap U/ \Mct_4\times \Mct_4}^{-1})$. Viewing the sheaves as elements in the Grothendieck group $K_0$ with operations $+$ and $-$, we have 
$$
N_{X\cap U/ \Mct_4\times \Mct_4}^{-1} = -p_1^*T\Mct_4 - p_2^*T\Mct_4 + T(X\cap U).
$$
Thus, we obtain
\begin{align*}
j_U^*{p_1}_*&\left[c(p_1^*t_1^*T\A_4)\cap s(C_{\F_4/\Mct_4\times \Mct_4})\right]_8 =\\
&= \sum_{X\in \{\Delta^\pm,A^\pm,B\}} {p_1}_*\left[c(-p_1^*T\Mct_4 -p_2^*T\Mct_4 + T(X\cap U) + p_1^*t_1^*T\A_4)\right]_8\\
&= \sum_{X\in \{\Delta^\pm,A^\pm,B\}} {p_1}_*c_\text{top}(N_{p_2}/p_1^*N_{t_1})
\end{align*}

\noindent with maps as in the diagram
$$\begin{tikzcd}
     X\cap U \ar[d,"p_1"]\ar[r,"p_2"] & \Mct_4\ar[d,"t_2"]\\
    \Mct_4 \ar[r,"t_1"] & \A_4.
\end{tikzcd}$$

We have thus shown that $j_U^*t^*T_4$ breaks up into a sum $$j_U^*t^*T_4=\sum_{X\in \{\Delta^\pm,A^\pm,B\}}  {p_1}_*c_\text{top}(N_{p_2}/p_1^*N_{t_1}).$$ The class $c_\text{top}(N_{p_2}/p_1^*N_{t_1})$ is the same for $\Delta^+$ and $\Delta^-$ and similarly for $A^+$ and $A^-$. Since the classes $[Z_i]$ for $1\leq i \leq 6$ all push forward to the expected dimension $8$ in $\Mct_4$ (showed for $Z_5, Z_6$ in Section \ref{nonred2}), we thus obtain the decomposition of $t^*T_4$ as in (\ref{class}).\end{proof}

\begin{proposition}\label{mcoeffs}
    The coefficients $m_i$ for $1\leq i \leq 4$ are given by $m_1 = m_2 = -2$ and $m_3 = m_4 = -3$.
\end{proposition}

\begin{proof} Proposition \ref{form} tells us that the coefficients $m_i$ only depend on the dimensions in the geometry of the fiber product. The $8$-dimensional intersections $Z_1,Z_2$ are each formed by intersecting two components of genus $9$. Cutting down the dimensions by $8$, we can use a model where two $1$-dimensional components in the fiber product intersect in a point. The intersection is then transverse in a $2$-dimensional ambient space. Thus, the multiplicities $m_1=m_2$ can be calculated as in Example \ref{Acoeff} to be $-2$. The $8$-dimensional intersections $Z_3,Z_4$, are each formed by intersecting a $10$-dimensional component with a $9$-dimensional in the fiber product. Cutting down the dimension by $8$ and considering a surrounding $3$-dimensional space where the intersection is transverse, we obtain dimensions as in \ref{Bcoeff}. This gives us $m_3=m_4 = -3$.\end{proof}

Using results from Section \ref{nonred2}, we show in Section \ref{noncont} that $m_5 = m_6 = 1$. Afterwards, we proceed to calculate the canonical contributions from the components $\Delta^\pm$, $A^\pm$ and $B$.

\subsubsection{Contributions from the nonreduced parts of the component $B$}\label{noncont}

The nonreduced part of the fiber product will be shown in Section \ref{nonred2} to lie inside $B$, and will consist of two components $Z_5$ and $Z_6$, given by pairs of curves \newline $(C_1\cup_{p\sim q} C_2, C_1\cup_{p\sim \bar{q}}C_2)$ resp. $(C_1\cup_{p\sim q} C_2, C_1\cup_{\bar{p}\sim q}C_2)$ equipped with isomorphisms $+$ of Jacobians. Here $C_1,C_2$ are genus $2$ curves, and $\bar{p}$ denotes the hyperelliptic conjugate of $p$.

We will find in Section \ref{nonred2} that the scheme structure of the fiber product near a general point of $Z_5$ or $Z_6$ is of the form 
$$W = \text{Spec }\frac{\C[[x_1,\dots,x_9,x_{10},x_{11}]]}{(x_9x_{11},x_{10}x_{11}, x_{11}^2).}$$

Since $Z_5, Z_6$ have the expected dimension $8$, the contribution from $Z_i$ for $i=5,6$ to the intersection will be a multiple of $\delta_B$, using that ${p_1}_*[Z_i] = \delta_B\in \mathsf{CH}^1(\Mct_4)$ for $i=5,6$.

A residual calculation using \cite[Corollary 9.2.3]{Ful} with the residual scheme to $Z = V(x_{11})$ in $W$ being $R = V(x_9,x_{10},x_{11})$ and with the surrounding space being $V = \C[[x_1,\dots,x_{18}]]$, we obtain the residual intersection class as $[R]$.

The contribution to $t^*T_4$ from $[Z_5]$, $[Z_6]$ each is therefore $\delta_B$.\newline

\noindent \textbf{Remark.} A simplified local model for the residual intersection in the previous paragraph is given by a section of the rank $3$ vector bundle $E = \mathcal{O}(2)^{\oplus 3}$ over $\mathbb{P}^3$ cutting out the locus $W = V(xz,yz,z^2)$. Using the Cartier divisor $D = V(z)$ in $\mathbb{P}^3$ and the residual scheme $R = V(x,y,z)$, it can be verified directly that the residual intersection class is $[R]$.

\subsubsection{Contributions from components of the Torelli fiber product}

We now turn to calculating the classes ${p_1}_*c_1\left(\frac{p_1^*N_{\Mct_4/\A_4}}{N_{X/\Mct_4}}\right)$ for  $X = \Delta^+,A^+,B$.

We will calculate the relevant Chern classes using the geometry of the components of the fiber product. For this, we find it useful to calculate the Chern classes of $T\Mct_{g,n}$ and $T\A_g$, respectively. This is done in Appendix \ref{chern}. We also put together a list of relations for pullbacks and pushforwards of tautological classes in $\Mct_{g,n}$ (more generally in $\ol{\M}_{g,n}$) under forgetful maps and gluing maps (from \cite[Chapter 17]{ACG}).

\subsubsection{Pullbacks and pushforwards of tautological classes under gluing and forgetful maps}
Consider the space $\ol{\M}_{g,P}$ where $P=\{p_1,\dots, p_n\}$ are marked points.

For a stable graph $\Gamma$ with $n$ legs, we let $\bt \xi_{\G} \colon \ol{\M}_{\G} \to\ol{\M}_{g,P}\et$ be the gluing map. Let $\bt \pi_x \colon \ol{\M}_{g,P\cup \{x\}} \to \ol{\M}_{P}\et$ be the forgetful map. For $p\in P$, let $D_{p}$ be the image of the section $\bt p \colon \ol{\M}_{g,P} \to \ol{\mathcal{C}}_{g,P}\et$ of the universal curve $\bt \pi\colon \ol{\mathcal{C}}_{g,P} \to \ol{\M}_{g,P}.\et$

Define the classes (for $\ol{\M}_{g,n}$, or for $\Mct_{g,n}$ by restriction)

\begin{itemize}
\item $\delta  = \sum_{\Gamma \text{ s.t. } |E(\Gamma)| = 1}\frac{1}{|\text{Aut}(\Gamma)|}{\xi_{\G}}_*(1)$
\item $D = \sum_{p\in P}  D_{p}$, $K_i = c_1(\omega_\pi(D))$ 
\item $\psi_i = c_1(p_i^*\omega_\pi) =  c_1(p_i^*\Omega_\pi)$, $\kappa_i = \pi_*(K^{i+1})$, $\lambda_i = c_i(\pi_*\omega_\pi)$.
\end{itemize}

\begin{definition}\label{tautring}
    The \textit{tautological ring} $R^*(\ol{\M}_{g,n})\subset CH^*(\ol{\M}_{g,n})$ is the $\mathbb{Q}$-algebra generated by pushforwards under gluing maps $\xi_{\G}$ of products of classes $p^*_v\kappa_i$ and $p^*_v\psi_i$ for $v\in V(\Gamma)$ where $p_v$ denotes the projection map from $\ol{\M}_{\G}$ to the factor corresponding to the vertex $v$ \cite[Proposition 11]{GP}.
\end{definition}

Let $H(L)$ be the set of half-edges being glued under $\xi_{\G}$ and let $\delta_{0,\{p,x\}}$ be the divisor class of curves where $x,p$ belong to the same rational tail. We have the following relations \cite[Chapter 17]{ACG}:

\begin{center}\label{relations}
\begin{tabular}{ | m{0.5cm} " m{3cm}| m{5.5cm} |m{1cm} | }
  \hline
    & $\pi_{x}^*$ & $\xi_{\G}^*$ & ${\pi_{x}}_*$\\ 
  \thickhline
  $\lambda_1$ &  $\lambda_1$ & $\sum_{v \in V(\Gamma)} p_v^*\lambda_1$& $0$ \\ 
  \hline
  $\kappa_i$ & $\kappa_i - \psi_x^i$ &  $\sum_{v \in V(\Gamma)} p_v^*\kappa_i$ & $\kappa_{i-1}$\\ 
  \hline
  $\delta$ & $\delta - \sum_{p \in P} \delta_{0,\{p,x\}}$ &  $\sum_{v \in V(\Gamma)} p_v^*\delta-\sum_{l \in H(L)} p_{v(l)}^*\psi_l$ & $n$\\ 
  \hline
  $\psi_p$ & $\psi_p-\delta_{0,\{p,x\}}$ & $\psi_p$ & $1$\\ 
  \hline
  $\psi_x$ & --- & --- & $\kappa_0$\\ 
  \hline
\end{tabular}
\captionof{table}{Pullbacks and pushforwards of tautological classes. The rows are labeled by the classes being acted upon. The columns correspond to given pullback/pushforward operations.}
\end{center}

Moreover, we have $\kappa_1 = 12\lambda_1 + \sum_i \psi_{p_i} - \delta$ and $\kappa_0 = 2g-1$.

\subsubsection{Contributions from the components $\Delta^+,\Delta^{-}$}

Since $\bt p_i\colon \Delta^+ \to \Mct_4 \et$ are isomorphisms for $i=1,2$, the class $c_1\left(\frac{p_1^*N_{\Mct_4/\A_4}}{N_{\Delta^+/\Mct_4}}\right)$ equals $$c_1(N_{\Mct_4/\A_4}) = c_1(t_1^*T\A_4) - c_1(\Mct_4).$$

From Appendix \ref{chern}, we have $c_1(T\Mct_4) = 2\delta-13\lambda_1$ and $c_1(t_1^*T\A_4) = -5\lambda_1$. Thus the contribution from $\Delta^+$ (and similarly for $\Delta^-$) is $${p_1}_*c_1(N_{\Mct_4/\A_4}) = 8\lambda_1 - 2\delta.$$

\subsubsection{Contributions from the components $A^\pm$}

Since $A^+$ is isomorphic to to $\Mct_{1,1}\times \Mct_{3,2}$ away from a locus that has images under $p_1,p_2$ in $\Mct_4$ of codimension $\geq 2$, we can assume that these spaces are equal for our intersection calculations. Similarly, we abuse notation and write $B = (\Mct_{2,2}\times \Mct_{2,2})/S_2$.

Consider the (non-Cartesian) diagram 
$$\begin{tikzcd}
     \Mct_{1,1}\times \Mct_{3,2} \ar[d,"p_1"]\ar[r,"p_2"] & \Mct_4\ar[d,"t_2"]\\
    \Mct_4 \ar[r,"t_1"] & \A_4.
\end{tikzcd}$$

We aim to calculate the class ${p_1}_*c_1\left(\frac{p_1^*N_{t_1}}{N_{p_2}}\right)$ which equals $${p_1}_*c_1(p_1^*t_1^*T\A_4) -  {p_1}_*c_1(p_1^*T\Mct_4) - {p_1}_*c_1(p_2^*T\Mct_4) + {p_1}_*c_1(T\Mct_{1,1}\times T\Mct_{3,2}).$$

Denote the marking at the first factor $\Mct_{1,1}$ by $p$, and the markings at the second factor by $q,y$ in that order. Denote the map forgetting the marking $p$ by $\pi_p$ and similarly for $q,y$.

Note that $p_1$ can be written as a composition
$$\begin{tikzcd}
     \Mct_{1,1}\times \Mct_{3,2}\ar[r,"\text{$(id,\pi_y)$}"] &\Mct_{1,1}\times \Mct_{3,1} \ar[r,"\xi_1"] & \Mct_4
\end{tikzcd}$$
\noindent where $\xi_1$ is the gluing map. 

Similarly, the map $p_2$ factors as 

$$\begin{tikzcd}
     \Mct_{1,1}\times \Mct_{3,2}\ar[r,"\text{$(id,\pi_q)$}"] &\Mct_{1,1}\times \Mct_{3,1} \ar[r,"\xi_2"] & \Mct_4
\end{tikzcd}$$
where $\xi_2$ is the gluing map.

We have $c_1(T\Mct_4) = 2\delta- 13\lambda_1$ and $c_1(t_1^*T\A_4)= -5\lambda_1$. Moreover, $$c_1(T\M_{1,1}) = (2\delta- 13\lambda_1 -\psi_p)\otimes 1$$ and $$c_1(T\M_{3,2}) =  1\otimes(2\delta- 13\lambda_1 -\psi_q-\psi_y)$$ on $\Mct_{1,1}\times \Mct_{3,1}$.

We use the relations in \ref{relations} and the decompositions of $p_1,p_2$ to calculate
\begin{align*}
c_1\left(\frac{p_1^*N_{t_1}}{N_{p_2}}\right) = 8(\lambda_1 \hspace{2pt}\otimes \hspace{2pt} 1 + 1\otimes \lambda_1)-2(\delta\otimes 1+ 1\otimes \delta) + 3\psi_p\otimes 1 + 1\otimes \psi_q + 1\otimes \psi_y.
\end{align*}

Pushing forward via $p_1$ we obtain the contribution from $A^+$ as ${p_1}_*c_1\left(\frac{p_1^*N_{t_1}}{N_{p_2}}\right) = 4\delta_A$ and similarly for $A^-$.

\subsubsection{Contribution from the component $B$}\label{Bcont}

Writing $B = (\Mct_{2,2}\times \Mct_{2,2})/S_2$, we consider the  diagram

$$\begin{tikzcd}
     \Mct_{2,2}\times \Mct_{2,2} \ar[d,"p_1"]\ar[r,"p_2"] & \Mct_4\ar[d,"t_2"]\\
    \Mct_4 \ar[r,"t_1"] & \A_4.
\end{tikzcd}$$

We denote by $p,x$ the markings on the first factor $\Mct_{2,2}$ and by $q,y$ the markings on the second factor.

In this case, the map $p_1$ factors as
$$\begin{tikzcd}
     \Mct_{2,2}\times \Mct_{2,2}\ar[r,"\text{$(\pi_x,\pi_y)$}"] &\Mct_{2,1}\times \Mct_{2,1} \ar[r,"\xi_1"] & \Mct_4
\end{tikzcd}$$
\noindent
whereas $p_2$ factors as
$$\begin{tikzcd}
     \Mct_{2,2}\times \Mct_{2,2}\ar[r,"\text{$(\pi_p,\pi_q)$}"] &\Mct_{2,1}\times \Mct_{2,1} \ar[r,"\xi_2"] & \Mct_4
\end{tikzcd}$$ where $\xi_1$, $\xi_2$ are gluing maps.

For the class ${p_1}_*c_2\left(\frac{p_1^*N_{\Mct_4/\A_4}}{N_{B/\Mct_4}}\right)$, we need to compute $${p_1}_*c_2\left(p_1^*t_1^*T\A_4 -p_1^*T\Mct_4 + T(\Mct_{2,2}\times \Mct_{2,2})-p_2^*T\Mct_4\right),$$
which will be equal to $2{p_1}_*c_2\left(\frac{p_1^*N_{\Mct_4/\A_4}}{N_{B/\Mct_4}}\right)$ using $B = (\Mct_{2,2}\times \Mct_{2,2})/S_2$.

We note that any class of the form ${p_1}_*p_1^*\alpha$ for $\alpha \in \mathsf{CH}^*(\Mct_4)$ will vanish, since the pair of forgetful morphisms $(\pi_x,\pi_y)$ gives a flat map with positive fiber dimension. It thus suffices to calculate ${p_1}_*c_2\left(T(\Mct_{2,2}\times \Mct_{2,2})-p_2^*T\Mct_4\right)$. Denoting $\Mct_{2,2}\times \Mct_{2,2}$ by $X$, this class can be expanded as $${p_1}_*\left(c_2(TX)-p_2^*c_2(T\Mct_4)+ p_2^*c_1(T\Mct_4)(p_2^*c_1(T\Mct_4)-c_1(TX))\right).$$
The constituents of $c_2(TX)$ coming from classes on one of the $\Mct_{2,2}$-factors  will vanish under pushforward by $p_1$. Moreover, the class $\lambda_1$ vanishes under pushforwards of forgetful maps. Thus, we have $${p_1}_*c_2(TX) = {p_1}_*(c_1(\Mct_{2,2})\otimes c_1(\Mct_{2,2})) = {p_1}_*((2\delta - \psi_p-\psi_x)\otimes (2\delta - \psi_q-\psi_y)).$$

We obtain ${p_1}_*c_2(TX) = 8\delta_B$ where $\delta_B={\xi_1}_*(1)/2$ using the relations in \ref{relations}.

 We proceed with $-{p_1}_*p_2^*c_2(T\Mct_4)$. The Chern class can be calculated from Appendix \ref{chern} as
 \begin{equation}\label{c2calc}
    c_2(T\Mct_4) = -\frac{1}{2}\kappa_2 + \frac{1}{2}(13\lambda_1 -2\delta)^2 + \frac{1}{2}{\xi_{1,3}}_*(\psi\otimes 1+1\otimes \psi)+\frac{1}{4}{\xi_{2,2}}_*(\psi\otimes 1+1\otimes \psi) 
 \end{equation}
 \noindent where $\xi_{h,4-h}$ denotes the map gluing a genus $h$ curve to a curve of genus $4-h$, and where $\psi\otimes 1$, $1\otimes \psi$ are the psi classes corresponding to the respective glued markings.

The relations in \ref{relations} allow us to calculate ${p_1}_*p_2^* \lambda_1^2 = {p_1}_*p_2^* \lambda_1\cdot \delta = {p_1}_*p_2^*\kappa_2 = 0$.

We also obtain ${p_1}_*p_2^*\delta^2 = 16\delta_B$ and 
\begin{equation*}
\begin{aligned}
&{p_1}_*p_2^*\left( \frac{1}{2}{\xi_1}_*(\psi\otimes 1+1\otimes \psi) + \frac{1}{4}{\xi_2}_*(\psi\otimes 1+1\otimes \psi)\right) =\\
& =\frac{1}{2}{p_1}_*(\pi_p,\pi_q)^*(\psi_x\otimes 1+1\otimes \psi_y)(-\psi_x\otimes 1 -1\otimes\psi_y) =\\
&= -4{\xi_1}_*(1) = -8\delta_B.
\end{aligned}
\end{equation*}

Putting this together, we obtain $-{p_1}_*p_2^*c_2(T\Mct_4) = -24\delta_B.$

Finally, we compute
\begin{equation*}
\begin{aligned}
&{p_1}_* p_2^*c_1(T\Mct_4)\left(p_2^*c_1(T\Mct_4)-c_1(TX)\right)= 32\delta_B.
\end{aligned}
\end{equation*}

We conclude that
\begin{equation*}
\begin{aligned}
{p_1}_*c_2\left(\frac{p_1^*N_{\Mct_4/\A_4}}{N_{B/\Mct_4}}\right) = 8\delta_B
\end{aligned}
\end{equation*}
\noindent is the contribution from the component $B$.

\subsubsection{The classes $t^*T_4$ and $T_4$}\label{proofofmuthm}

We finally prove Theorem \ref{muthm}.

\begin{proof}[Proof of Theorem \ref{muthm}] Putting together the calculations from Section \ref{inttheory} into (\ref{class}), using that $\delta = \delta_A + \delta_B$, we obtain $t^*T_4 = 16\lambda_1.$

Noting that $T_4 \in \text{Pic}(\A_4)$, we can write $T_4 = c \lambda_1$ for some multiple $c$. Since $t^*\lambda_1 =\lambda_1$, knowing that $t^*T_4 = 16\lambda_1$ and $\lambda_1\neq 0$ in $\text{Pic}(\Mct_4)$ gives us $c = 16$ and therefore $T_4 = 16\lambda_1$.\end{proof}

\subsection{Expansion of the period map and nonreduced scheme structure in the genus $4$ Torelli fiber product }\label{nonred2}

We determine the scheme structure of the nonreduced locus in the Torelli fiber product in genus $4$ by studying the local analytic equations defining the period map.

\subsubsection{The nonreduced loci $Z_5$ and $Z_6$ in the (2,2)-component}
Consider the Torelli map on neighborhoods around a general nodal curve $C=C_1\cup_{p\sim q} C_2$ where $g(C_1)=g(C_2)=2$ and its Jacobian $J(C)$. To remove the additional automorphisms of the product abelian varieties in $\A_2 \times \A_2$ given by interchanging the sign of one of the corresponding abelian factors, we work directly on the Siegel upper half space $\HH_4$. Considering a small analytic neighborhood $V$ around $J(C)$ in $\A_4$, elements in the locus $V\cap (\A_2\times \A_2)$ have an extra $\Z/2\Z$-factor in their automorphism group. The preimage of $V$ in $\HH_4$ therefore consists of an infinite number of disjoint copies of schemes $V'$ such that $V' \to V$ gives an étale double cover. A lift of the Torelli map to $\HH_4$ at $U = t^{-1}(V) \subset \Mct_4$ will land in one such $V'$. Label this map by $t_1\colon U\to V'$. Applying an involution given by a matrix $M\in \text{Sp}(8,\Z)$ to $\HH_4$ which fixes precisely $\A_2\times \A_2$ will induce an involution of $V'$ whose composition with $t_1$ gives a map $t_2\colon U \to V'$. We will choose this $M$ to be the $8\times 8$ matrix
$$
M = \begin{pmatrix}
   A & B\\
    C & D\\
\end{pmatrix}
$$

\noindent such that
$$
A = D = \begin{pmatrix}
    -1 & 0 & 0 & 0\\
    0 & -1 & 0 & 0\\
    0 & 0 & 1 & 0\\
    0 & 0 & 0 & 1\\
\end{pmatrix}
$$
and $B=C$ is the zero matrix.

The fiber product of $t\colon U \to V$ with itself consists of $4$ connected components, which are two copies each of the fiber products of the maps $t_1,t_1$ and $t_1,t_2$ respectively. 

We start by finding the precise nonreduced scheme structure of the fiber product of $t_1,t_2$. Later, in Proposition \ref{genred}, we also show that the fiber product of $t_1,t_1$ is reduced.

Consider local analytic coordinates $\C[[x_1,\dots,x_8,s]]$ for $U$ where $x_1,x_2,x_3,x_7$ and $x_4,x_5,x_6,x_8$ are coordinates for the different factors of $(\Mct_{2,1}\times \Mct_{2,1})\cap U$ and $s$ is a smoothing parameter. Precisely, the parameter $s$ is defined to be the plumbing parameter as in \cite{HN}. Moreover, we let $x_7,x_8$ correspond to the marked points on the respective factors $\Mct_{2,1}$. We choose coordinates on $V'\subset \HH_4$ to be $\C[[y_1,\dots,y_{10}]]$ corresponding to the matrix entries as follows:
$$
 \begin{pmatrix}
    y_1 & y_2 & y_7 & y_8\\
    y_2 & y_3 & y_9 & y_{10}\\
    y_7 & y_9 & y_4 & y_5\\
    y_8 & y_{10} & y_5 & y_6\\
\end{pmatrix} \in V' \subset \HH_4.
$$

To distinguish between the domains of $t_1$ and $t_2$, we label them $U_1$, $U_2$ and denote the coordinates of $U_2$ by  $\C[[x'_1,\dots,x'_8,s']]$.

Denote $\ul{x} = (x_1,\dots, x_8)$ and $\ul{x}' =(x'_1,\dots, x_8')$. Let $C_{\ul{x},s}$ be the curve associated to $(\ul{x},s)$ in $U_1$. When $s=0$, let $(C_{1})_{\ul{x}}$, $(C_{2})_{\ul{x}}$ be the irreducible components of $C_{\ul{x},0}$. Denote the points on $(C_{1})_{\ul{x}}$, $(C_{2})_{\ul{x}}$ corresponding to the node in $C_{\ul{x},0}$ by $p_1 = p_1(\ul{x})$, $p_2 = p_2(\ul{x})$. Let $z_1 = z_1(\ul{x})$, $z_2 = z_2(\ul{x})$ be local coordinates around these markings. Denote a basis of $H^0(\omega_{(C_{1})_{\ul{x}}})$ given in local coordinates near $p_1$ as $v_i = \tilde{v}_idz_1$ for $i=1,2$. Similarly, let $v_i = \tilde{v}_idz_2$ for $i=3,4$ be a basis for $H^0(\omega_{(C_{2})_{\ul{x}}})$.

By \cite[Corollary 1.3]{HN}, the period matrix $t(\ul{x},s)$ in this basis of differentials has expansion
\begin{equation}\label{expansion}
    t(\ul{x},s) = \tau_0(\ul{x}) + \tau_1(\ul{x}) s + \mathcal{O}(s^2)
\end{equation}
\noindent where
$$
\tau_1(\ul{x}) = \begin{pmatrix}
    0 &  0 & -\tilde{v}_1(p_1)\tilde{v}_3(p_2) & -\tilde{v}_1(p_1)\tilde{v}_4(p_2)\\
    0 & 0 & -\tilde{v}_2(p_1)\tilde{v}_3(p_2) & -\tilde{v}_2(p_1)\tilde{v}_4(p_2)\\
    -\tilde{v}_1(p_1)\tilde{v}_3(p_2) & -\tilde{v}_2(p_1)\tilde{v}_3(p_2) & 0 & 0\\
    -\tilde{v}_1(p_1)\tilde{v}_4(p_2) & -\tilde{v}_2(p_1)\tilde{v}_4(p_2) & 0 & 0\\
\end{pmatrix}
$$
\noindent and $\tau_0(\ul{x})$ is the lift of the period matrix of $J(C_{\ul{x},0})$ to $V'$ given by $t_1$.

The transformation $M$ of $V'$ is given by
$$
 \begin{pmatrix}
    y_1 & y_2 & y_7 & y_8\\
    y_2 & y_3 & y_9 & y_{10}\\
    y_7 & y_9 & y_4 & y_5\\
    y_8 & y_{10} & y_5 & y_6\\
\end{pmatrix} \text{\hspace{5pt} $\mapsto$ \hspace{5pt}}
\begin{pmatrix}
    y_1 & y_2 & -y_7 & -y_8\\
    y_2 & y_3 & -y_9 & -y_{10}\\
    -y_7 & -y_9 & y_4 & y_5\\
    -y_8 & -y_{10} & y_5 & y_6\\
\end{pmatrix}.
$$

We choose the curves $(C_1)_{\ul{x}}$, $(C_2)_{\ul{x}}$ to be 
$$(C_1)_{\ul{x}} : Y_1^2 = (x-1)(x-2)(x-3)(x-4-x_1)(x-5-x_2)(x-6-x_3)$$
$$(C_2)_{\ul{x}} : Y_2^2 = (x-1)(x-2)(x-3)(x-4-x_4)(x-5-x_5)(x-7-x_6).$$

We choose a marked point $(x_7,Y_1(x_7))$ on $C_1$ and $(x_8,Y_2(x_8))$ on $C_2$ near $(0,Y_i(0))$ for a specified branch of $Y_i$ for $i=1,2$, referred to as the \textit{upper branch}.

\begin{proposition}\label{6eqns}
    In the fiber product of $t_1,t_2$, we have $$x_i' = x_i +\alpha_i(\ul{x},s)s^2 - \alpha_i(\ul{x}',s')s'^2$$ for $1\leq i \leq 6$ and some $\alpha_i \in \C[[x_1,\dots,x_8,s]]$.
\end{proposition}

\begin{proof}
By the Torelli theorem, we know that
$$\tau_0(\ul{x}) = \begin{pmatrix}
    f_1 &  f_2 & 0 & 0\\
    f_2 & f_3 & 0 & 0\\
   0 & 0 & f_4 & f_5\\
  0 &0 & f_5 & f_6\\
\end{pmatrix}$$

\noindent for some $f_i \in \C[[x_1,\dots,x_6]]$ where the map $(x_1,\dots, x_6) \mapsto (f_1,\dots,f_6)$ is injective (in fact, an isomorphism since $\Mct_2\cong \A_2$).

Taking the fiber product with $t_2$ over the entries $(i,j) = (1,1), (1,2), (2,1), (2,2)$, noting that $\tau_1(\ul{x})_{i,j} = 0$ for these indices, we obtain the desired relations.\end{proof}

\begin{proposition}\label{remaining}
    By a suitable change of variables, the remaining relations in the fiber product of $t_1,t_2$ are $s+s'=0$, $s(x_7-x_7') = 0$, $s(x_8-x_8') = 0$ and $s^2$ = 0. 
\end{proposition}

\begin{remark}
    Combining Propositions \ref{6eqns} and \ref{remaining} we obtain $x_i = x_i'$ in the fiber product for $1\leq i \leq 6$.
\end{remark}

\begin{proof}[Proof of Proposition \ref{remaining}]
 We glue the Riemann surface $(C_1)_{\ul{x}}$ using two branches of $Y$ each, where the branch cuts go from $1$ to $2$, from $3$ to $4+x_1$ and from $5+x_2$ to $6+x_3$. Glue $(C_2)_{\ul{x}}$ analogously. We choose a symplectic basis for $H_1((C_1)_{\ul{x}},\mathbb{Z})$ given by $A_1,A_2,B_1,B_2$ where $A_1$ is a loop going around the branch cut $[1,2]$ clockwise, $A_2$ a clockwise loop around $[3,4+x_1]$, $B_1$ the upper semicircle directed from the segment $[1,2]$ to $[5+x_2,6+x_3]$ joined with the corresponding oppositely directed lower semicircle in the other branch, and $B_2$ the upper semicircle from $[4+x_4,5+x_5]$ to $[5+x_5,6+x_6]$ joined with the lower in the opposite branch (see Figure \ref{branchcuts}). Here, the cycles $A_1,A_2$ and the upper semicircles of $B_1,B_2$ are taken to lie in the upper branch of $Y$. The basis for $H_1((C_2)_{\ul{x}},\mathbb{Z})$ is chosen similarly.

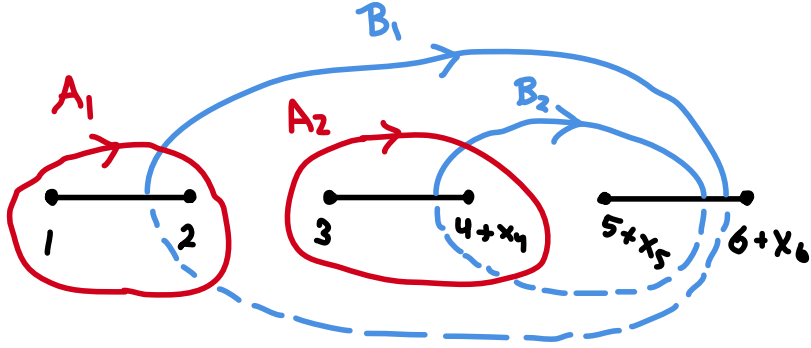
\begin{figure}[h]
\centering
\scalebox{0.7}{\input{branchcuts.tex}}
\caption{Branch cuts and homology basis for $(C_1)_{\ul{x}}$.}\label{branchcuts}
\end{figure}

We first work modulo $s^2, x_1,\dots,x_6$. We chose bases $v_1,v_2$ of $H^0(\omega_{C_1})$ and $v_3,v_4$ of $H^0(\omega_{C_2})$ which are dual to $A_1,A_2$ and $A_3,A_4$ respectively. Since $C_1$, $C_2$ are hyperelliptic with the ``same" loops $A_i$ in the upper branch, we can write $$v_1 = a\frac{dz_1}{Y_1} + bz_1\frac{dz_1}{Y_1},\hspace{6pt} v_2 = c\frac{dz_1}{Y_1} + dz_1\frac{dz_1}{Y_1}$$ $$v_3 = a\frac{dz_2}{Y_2} + bz_2\frac{dz_2}{Y_2},\hspace{6pt} v_4 = c\frac{dz_2}{Y_2} + dz_2\frac{dz_2}{Y_2}$$ for some $a,b,c,d\in \C$ such that the matrix
$\begin{pmatrix}
    a&  b\\
    c & d\\
\end{pmatrix}$ is invertible. Thus, we can write the upper right $2\times 2$-block of $\tau_1(\ul{x})$ as

$$\frac{-1}{Y_1(x_7)Y_2(x_8)}\begin{pmatrix}
    (a+bx_7)(a+bx_8) &  (a+bx_7)(c+dx_8) \\
    (c+dx_7)(a+bx_8) & (c+dx_7)(c+dx_8)\\
\end{pmatrix}$$
\noindent where $Y_1,Y_2$ correspond to the upper branches.

Writing this matrix as a column vector of the entries indexed by $(1,1)$, $(1,2)$, $(2,1)$, $(2,2)$ in that order, we can view this matrix as a sum of vectors

$$\frac{-1}{Y_1(x_7)Y_2(x_8)}\left(\begin{pmatrix}
    a^2\\
    ac\\
    ac \\
    c^2\\
\end{pmatrix}+\begin{pmatrix}
    ab\\
    bc\\
    ad \\
    cd\\
\end{pmatrix}x_7+\begin{pmatrix}
    ab\\
    ad\\
    bc \\
    cd\\
\end{pmatrix}x_8+\begin{pmatrix}
    b^2\\
    bd\\
    bd \\
    d^2\\
\end{pmatrix}x_7x_8\right).$$

Since $ad-bc\neq 0$, we can use row operations on this vector to turn it to the form

$$\frac{-1}{Y_1(x_7)Y_2(x_8)}\left(\begin{pmatrix}
    1\\
    0\\
    0 \\
    0\\
\end{pmatrix}+\begin{pmatrix}
    0\\
    1\\
    0 \\
    0\\
\end{pmatrix}x_7+\begin{pmatrix}
    0\\
    0\\
    1 \\
    0\\
\end{pmatrix}x_8+\begin{pmatrix}
    0\\
    0\\
    0\\
    1\\
\end{pmatrix}x_7x_8\right).$$

If we use the same row operations on the vector $\ul{a}(\ul{x})s + \ul{b}(\ul{x},s)s^2$ corresponding to the upper right block of the matrix $t(\ul{x},s)$, the first entry is given by $\ul{a}_1(\ul{x})s + \ul{b}_1(\ul{x},s)s^2$ where $\ul{a}_1(\ul{0})\neq 0$. We can thus make an invertible substitution $$\mathfrak{s} = (\ul{a}_1(\ul{x}) + \ul{b}_1(\ul{x},s)s)s$$ in the coordinates for $U$. From now on, we use the same notation $s$ for this new coordinate $\mathfrak{s}$ which we work with unless stated otherwise. 

By taking the fiber product with $t_2$ and considering the upper right $2\times 2$ block of $t(\ul{x},s)$ in this new basis, we obtain the equations 
\begin{equation*}
\begin{aligned}
    &s+s' = 0\\
    &s(r_2x_7+a_2x_8+b_2x_7x_8-r_2'x_7'-a_2'x_8'-b_2'x_7'x_8') = \rho_2(\ul{x},\ul{x}',s)s^2\\
    &s(r_3x_8+a_3x_7+b_3x_7x_8-r_3'x_8'-a_3'x_7'-b_3'x_7'x_8') = \rho_3(\ul{x},\ul{x}',s)s^2\\
    &s(r_4x_7x_8+a_4x_7+b_4x_8-r_4'x_7'x_8'-a_4'x_7'-b_4'x_8') = \rho_4(\ul{x},\ul{x}',s)s^2
\end{aligned}
\end{equation*}
\noindent for some units $r_i(x_1,\dots,x_6)$ and nonunits $a_i(x_1,\dots,x_6)$, $b_i(x_1,\dots,x_6)$ in $\C[[x_1,\dots,x_6]]$, where $a_i' = a_i(x_1',\dots,x_6')$, $b_i' = b_i(x_1',\dots,x_6')$, and $\rho_i(\ul{x},s) \in \C[[x_1,\dots,x_8,s]]$ for $2\leq i \leq 4$.

Working modulo $s^2x_1,\dots,s^2x_8,s^3$, using Proposition \ref{6eqns} and dividing the three last equations by $r_i$ for $i =2,3,4$, we obtain the new equations
\begin{equation*}\label{neweqs}
\begin{aligned}
    &s[(x_7-x_7')+a_2(x_8-x_8')+b_2(x_7x_8-x_7'x_8')] = \rho_2s^2\\
    &s[(x_8-x_8')+a_3(x_7-x_7')+b_3(x_7x_8-x_7'x_8')] = \rho_3s^2\\
    &s[(x_7x_8-x_7'x_8')+a_4(x_7-x_7')+b_4(x_8-x_8')] = \rho_4s^2
\end{aligned}
\end{equation*}
\noindent where $\rho_i = \rho_i(\ul{0},\ul{0},0)$ for $i =2,3,4$.

Solving this system for $s(x_7-x_7')$, $s(x_8-x_8')$, $s(x_7x_8-x_7'x_8')$ gives, since we are working modulo $s^2x_1,\dots,s^2x_8,s^3$,
\begin{equation*}
\begin{aligned}
    s(x_7-x_7') &= \rho_2s^2\\
    s(x_8-x_8') &= \rho_3s^2\\
    s(x_7x_8-x_7'x_8') &= \rho_4s^2.
\end{aligned}
\end{equation*}

Substituting the first two equations into the last gives $$\text{nonunit}\cdot s^2 = \rho_4s^2.$$ To obtain $s^2 = 0$ and hence the remaining relations, it thus suffices to show that $\rho_4 \neq 0$.

When showing that $\rho_4 \neq 0$, we work modulo $x_1,\dots,x_8,s^3$. We obtain the coefficient of $s^2$ in the expansion of $t(\ul{0},s)$ using \cite[Equation (3.9)]{HN}.

\begin{remark}
    The second equation in the formula of $v_{k,s}(z)$ on p.30 in \cite{HN} holds only up to $O(s^2)$ and not to $O(s^3)$ as stated. This is also indicated in \cite[Corollary 1]{Y}.
\end{remark}

The upper $2\times 2$ matrix of $t(\ul{0},s)$, viewed as a vector, is given in the old basis (before changing the variable $s$ to $\mathfrak{s}$) by

$$\frac{-1}{Y_1(0)Y_2(0)}\begin{pmatrix}
    a^2\\
    ac\\
    ac \\
    c^2\\
\end{pmatrix}s+\begin{pmatrix}
    G_1D_1\\
    G_1D_2\\
    G_2D_1\\
    G_2D_2\\
\end{pmatrix}s^2$$
\noindent for $G_i,D_i$ to be defined shortly.

Applying the matrix $$N = \begin{pmatrix}
    a^2& ab & ab & b^2\\
    ac & bc & ad & bd\\
    ac & ad & bc & bd \\
    c^2 & cd & cd & d^2\\
\end{pmatrix}^{-1}$$ which accounts for the change of basis and replacing $s$ by $\mathfrak{s}$, the matrix $t(\ul{0},s)$ is given by 

$$\begin{pmatrix}
    1\\
    0\\
    0 \\
    0\\
\end{pmatrix}s + \frac{-(Y_1(0)Y_2(0))^2}{(ad-bc)^2}\begin{pmatrix}
    0\\
   *\\
    *\\
    -c^2G_1D_1 +acG_1D_2 + acG_2D_1 -a^2G_2D_2\\
\end{pmatrix}s^2.$$

Taking the fiber product with $t_2$ over the fourth entry, noting that $s^2 = s'^2$, we obtain $$\rho_4 \neq 0 \iff -c^2G_1D_1 +acG_1D_2 + acG_2D_1 -a^2G_2D_2 \neq 0.$$

We now define the $D_i$ and $G_i$ for $i=1,2$. Let $I$ be the transformation taking $z_1$ to $s/z_2$ (see \cite[Definition 2.2]{HN}). The $s^2$-term of $I^*v_3$ is given by $D_1/z_1^3$ where $D_1 = (aY_2'(0)-b)/Y_2(0)$. Similarly, $D_2 = (cY_2'(0)-d)/Y_2(0)$ is obtained from $I^*v_4$.

Let $K_1(z,z_1)$ be the $\{A_1,A_2\}$-normalized \textit{Cauchy kernel} for $C_1$, see \cite[Section 3.1]{HN}. Then 
$$G_i = \int_{z\in B_i}\int_{|z_1| = \epsilon}\frac{K_1(z,z_1)}{z_1^3}dz_1$$
\noindent where $\epsilon > 0$ is sufficiently small.

By \cite[Equation (17)]{Z}, the normalized Cauchy kernel is given by, for points $(z_1,y_1)$ and $(z,y)$ on $C_1$, 
$$K_1 = \frac{(y_1+y)dz}{2(z-z_1)y} + h(z_1)\frac{dz}{y} + k(z_1)\frac{zdz}{y}$$ for suitable holomorphic functions $h,k$. 

Using MATLAB, we find the $z_1^2$-terms of $h$, $k$ and calculate the values of $a,b,c,d$. The inner integral of $G_i$ is calculated by taking the residue at $z_1 = 0$ since $z$ lies away from $0$ in the outer integral. The outer integral is calculated along the segments $[2,3]$ and $[4,5]$ on the respective branches of $Y_1$. Evaluating the contour integrals in MATLAB, we find the values of $G_1,G_2$ and thus of $\rho_4$ (the code can be found in Appendix \ref{matlab}). The value $\rho_4$ is found to be nonzero as desired.\end{proof}

\begin{corollary}\label{nrscheme}
    The scheme structure at a general nonreduced point in the genus $4$ Torelli fiber product is given by $$\text{Spec }\frac{\C[[x_1,\dots,x_6,x_7,x_8,x_7',x_8',s]]}{(s^2, s(x_7-x_7'), s(x_8-x_8')).}$$
\end{corollary}

\begin{proof}
    Apply Propositions \ref{6eqns} and \ref{remaining}.
\end{proof}

\begin{remark}
    Note that Corollary \ref{nrscheme} is consistent with the fact that the nonreduced locus is supported on the $8$-dimensional subscheme $x_7 = x_7'$, $x_8=x_8'$ of a surrounding space of dimension $10$.
\end{remark}

\subsubsection{Generic reducedness of $Z_1, Z_2, Z_3, Z_4$}\label{proofofthm2}
It remains to show that $Z_5$, $Z_6$ are the only nonreduced loci that will contribute to the class $t^*T_4$. Away from the intersections of components in the fiber product, the reduced structure of each component is nonsingular. Thus, eventual nonreducedness can be detected by excess dimension of the tangent space. By looking at the Torelli map on tangent spaces, we find that the tangent space at all points away from $Z_1,\dots,Z_6$ in the fiber product (away from loci projecting to codimensions $\geq 2$ in $\Mct_4$) has the same dimension as the component it lies on. Thus, we only need to consider the scheme structure in the intersections $Z_1$, $Z_3$ (the arguments for $Z_2$, $Z_4$ are the same).

\begin{proposition}\label{genred}
    The scheme structure of $Z_3$ is generically reduced.
\end{proposition}

\begin{proof}
    This is saying that the fiber product of $t_1$ with itself is reduced. Indeed, the self-fiber product over the upper right $2\times 2$ matrix of $t_1$ is defined by the equations \begin{equation*}
\begin{aligned}
    &s-s' = 0\\
    &s(r_2x_7+a_2x_8+b_2x_7x_8 + k_2s -r_2'x_7'-a_2'x_8'-b_2'x_7'x_8' - k_2's) = 0\\
    &s(r_3x_8+a_3x_7+b_3x_7x_8+k_3s -r_3'x_8'-a_3'x_7'-b_3'x_7'x_8'- k_3's) =  0\\
    &s(r_4x_7x_8+a_4x_7+b_4x_8+k_4s -r_4'x_7'x_8'-a_4'x_7'-b_4'x_8' - k_4's) = 0
\end{aligned}
\end{equation*}
\noindent for units $r_i$ and nonunits $a_i$, $b_i$ as before, where $k_i = k_i(\ul{x},s)$ for $i=2,3,4$ and $k_i' = k_i(\ul{x}',s)$. By making the invertible substitutions
\begin{equation*}
\begin{aligned}
z_7 &= r_2x_7+a_2x_8+b_2x_7x_8 + k_2s\\
z_8 &= r_3x_8+a_3x_7+b_3x_7x_8+k_3s,\\
\end{aligned}
\end{equation*}
\noindent and defining $z_i'$ similarly for $i=7,8$ the equations thus read 
\begin{equation*}
\begin{aligned}
    s-s' &= 0\\
    s(z_7-z_7') &= 0\\
    s(z_8-z_8') &=  0.
\end{aligned}
\end{equation*}
Similarly, after a change of variables, we obtain the relations $x_i = x_i'$ for $1\leq i \leq 6$ in the fiber product. This indeed gives a reduced, but reducible fiber product.
\end{proof}

\begin{proposition}\label{z1prop}
    The scheme structure of $Z_1$ is generically reduced.
\end{proposition}

\begin{proof}
    We consider a local neighborhood $U\subset {\Mct_{4}}^{(n)}$ around a curve $E\cup_{p\sim q} C$ where $g(E) = 1$, $g(C) = 3$ and ${\Mct_{4}}^{(n)}$ is the moduli space of genus $4$ compact type curves with a level $n$ structure for some $n\geq 3$. We pick coordinates $\C[[x_1,\dots,x_8,s]]$ for this neighborhood where $x_1$ corresponds to varying the moduli of $E$, $x_2,\dots,x_7$ to varying the moduli of $C$, $x_8$ to varying the marking $q$ at $C$ and $s$ is a plumbing parameter. For $(\ul{x},0)\in U$, let $E_{\ul{x}}\cup_{p\sim q}C_{\ul{x}}$ be the corresponding curve. Let $z_1$ be the local parameter at $p$ and $z_2$ the parameter at $q$. Assume that $q$ is not a Weierstrass point of $C$. Then, a basis of holomorphic differentials for $E_{\ul{x}}$, $C_{\ul{x}}$ at $p$, $q$ respectively are given by $$v_1 = u_1dz_1, v_2 = u_2dz_2, v_3 =  u_3z_2dz_2, v_4 = u_4z_2^2dz_2$$ 
\noindent where $u_i$ are units in $\C[[z_1]]$, $\C[[z_2]]$ respectively for $1\leq i \leq 4$.
In this case, the expansion (\ref{expansion}) of the period matrix reads
$$
\tau_0(\ul{x}) = \begin{pmatrix}
    f_1 &  0 & 0 & 0 \\
   * & f_2 & f_3 & f_4\\
    * & * &  f_5& f_6\\
    * & * & * & f_7\\
\end{pmatrix}
$$
\noindent and 
$$
\tau_1(\ul{x}) = \begin{pmatrix}
    0 &  -u_1(0)u_2(x_8) & -u_1(0)u_3(x_8)x_8 & -u_1(0)u_4(x_8)x_8^2\\
    * & 0 & 0 & 0\\
    * & * & 0 & 0\\
    * & * & * & 0\\
\end{pmatrix}
$$
\noindent for some $f_i\in \C[[x_1,\dots,x_7]]$ such that the map $(x_1,\dots,x_7)\mapsto (f_1,\dots, f_7)$ is invertible.

Thus, after making invertible substitutions in both the domain and the image, the Torelli map is locally given by
$$t(\ul{x},s) = \begin{pmatrix}
x_1 &  s & x_8s & x_8^2s\\
* & x_2 & x_3 & x_4\\
* & * & x_5 & x_6\\
* &*& * & x_7\\
\end{pmatrix}.$$

The reducedness of the self-fiber product follows.\end{proof}

We now have the results to conclude Theorem \ref{thm2}.
\begin{proof}[Proof of Theorem \ref{thm2}]
Combine Corollary \ref{nrscheme} and Propositions \ref{genred}, \ref{z1prop}.
\end{proof}

\section{The class of the Torelli locus in $\overline{\A}_4$}\label{abar}
For any toroidal compactification $\overline{\A}_g$ such that the tropical Torelli map (see \cite{viviani}) sends cones of $M_g^{\text{trop}}$ to cones in the associated admissible cone decomposition of $A^{\text{trop}}_g$, the Torelli map extends to a map $t\colon \ol{\M}_g \to \ol{\A}_g$ \cite[Theorem 7.29]{N}. Examples of such toroidal compactifications are the perfect cone and second Voronoi compactifications (\cite{alexeev}, \cite[\S 18]{N76}).

Let $\overline{\A}_4$ be a toroidal compactification to which the Torelli map extends such that only one of its irreducible boundary divisors meets the image of $\ol{\M}_4$ (e.g. perfect cone, 2nd Voronoi). We will extend our results in Section \ref{torellisec} to determine the class $t^*{t_*}[\overline{\M}_4]$. This class has also been calculated by Mumford \cite{Mu}.

We start by considering the sublocus $\A_g\subset \ol{\A}_g^{\leq 1}\subset \ol{\A}_g$ of \textit{torus rank $\leq 1$ degenerations} as defined in \cite[\S 1]{Mu}. An element of $\ol{\A}_g^{\leq 1}\backslash \A_g$ is a pair $(X,\Theta_X)$ of a proper $g$-dimensional variety $X$ with an ample divisor $\Theta_X$ constructed as follows: Start with a $(g-1)$-dimensional principally polarized abelian variety $A$ with theta-divisor $\Theta_A$. Consider an element $x\in \text{Ext}^1(A,\mathbb{G}_m)\cong A^\vee$ given by a short exact sequence 
$$\bt  0\ar[r] & \mathbb{G}_m \ar[r]& G \ar[r] & A \ar[r]& 0.\et$$
Viewing $G$ as a $\mathbb{G}_m$-bundle over $A$, we can compactify $G$ to a $\pr^1$-bundle by adding the $0$-section $G_0$ and $\infty$-section $G_\infty$ to $G$. Denote this $\pr^1$-bundle by $\pi\colon \tilde{G}\to A$. Using the isomorphism $A\cong A^\vee$ induced by the principal polarization, we can view $x$ as an element of $A$. We glue together $G_0$ and $G_\infty$ in $\tilde{G}$ by identifying each point $(a,0)$ for $a\in A$ with $(a-x,\infty)$. Take $X$ to be the resulting variety and $\Theta_X$ to be the unique effective divisor on $X$ descending from the linear system $|G_\infty + \pi^{-1}(\Theta_A)|$ on $\tilde{G}$.

An automorphism of $(X,\Theta_X)$ is an isomorphism $\bt \hspace{-4pt} j \colon X  \hspace{-5pt}\ar[r,shorten >=5pt, shorten <=5pt,"\sim"]& \hspace{-5pt}X\hspace{-5pt}\et$ preserving $\Theta_X$ together with a group isomorphism $\bt \hspace{-5pt}h \colon G \hspace{-5pt}\ar[r,shorten >=5pt, shorten <=5pt,"\sim"]& \hspace{-5pt}G\hspace{-5pt}\et$ such that the natural action $G \curvearrowright X$ is compatible with $j,h$ \cite{ale}.

The subset $\ol{\M}_g^{\leq 1}= t^{-1}(\ol{\A}_g^{\leq 1}) \subset \ol{\M}_g$ consists of curves whose dual graph $\Gamma$ satisfies $h^0(\Gamma,\mathbb{Z}) \leq 1$. The restricted Torelli map $\bt t\colon\ol{\M}_g^{\leq 1}\to \ol{\A}_g^{\leq 1}\et$ sends a curve $C\in \ol{\M}_g^{\leq 1} \backslash \M^{ct}_g$ to the pair $(X,\Theta_X)$ defined using the short exact sequence
$$\bt 0\ar[r] & \mathbb{G}_m \ar[r]& J(C) \ar[r]& J(\tilde{C}) \ar[r]& 0\et$$
\noindent where $\tilde{C}$ is the normalization of $C$.

\begin{lemma}\label{autlem}
    Let $C$ be an irreducible curve with one self-node such that $\text{Aut}(J(\tilde{C})) = \mathbb{Z}/2\mathbb{Z}$. Let $C$ be such that $(X,\Theta_X)= t(C)$ is glued from $x\in J(\tilde{C})$ where $2x\neq 0$. The automorphism group of $(X,\Theta_X)$ is $\mathbb{Z}/2\mathbb{Z}$ where a generator is the involution.
\end{lemma}

The following proof of Lemma \ref{autlem} is due to Aitor Iribar López.
\begin{proof}[Proof of Lemma \ref{autlem}.]
The automorphism of $X$ induces a group isomorphism of $\mathbb{G}_m$, which is either $+$ or $-$. Moreover, to preserve the polarization, the induced automorphism of $J(\tilde{C})$ is $+$ or $-$. Assume first that the automorphism of $J(\tilde{C})$ is $+$.

To preserve the gluing along the zero/infinity sections as in \cite[\S 1]{Mu} under the action of the given group isomorphism $\sigma$ of $\mathbb{G}_m$, we need $(\sigma(a-x),\infty) = (\sigma(a)-x,\infty)\in J(C)_\infty$ for each $a\in J(\tilde{C})$. Since $2x\neq 0$, we obtain $\sigma = +$.

To prove that the corresponding automorphism of $J(C)$ is also trivial, we consider the induced automorphism of short exact sequences
$$\bt 0\ar[r] & \mathbb{G}_m \ar[r,"f",]\ar[d,equal]& J(C) \ar[r,"g"]\ar[d,"\sim" labl, "h"'] & J(\tilde{C}) \ar[d,equal]\ar[r]& 0\\ 
 0\ar[r] & \mathbb{G}_m \ar[r,"f"]& J(C) \ar[r,"g"] & J(\tilde{C}) \ar[r]& 0.
\et$$ There is a well-defined homomorphism from $J(\tilde{C})$ to $\mathbb{G}_m$ sending $a \in J(\tilde{C})$ to $c \in \mathbb{G}_m$ such that $f(c) = b - h(b)$ for $b\in J(C)$ such that $g(b)=a$.
Since $J(\tilde{C})$ is an abelian variety, $\text{Hom}(J(\tilde{C}),\mathbb{G}_m)=0$ and thus $h = \text{id}$. Since the automorphism of $X$ depends only on its restriction to $J(C)$, which is dense in $X$, the automorphism of $X$ is the identity.

Similarly, if the automorphism of $\mathbb{G}_m$ is the involution, the induced automorphism of $J(\tilde{C})$ is also the involution. Composing with the involution on $J(C)$ and applying the above argument shows that the automorphism of $X$ has to be the involution.\end{proof}

\begin{proof}[Proof of Proposition \ref{thm3}] We first find the class $t_*[\ol{\M}^{\leq 1}_4]$ with respect to the map $$t\colon \ol{\M}_4^{\leq 1}\to \ol{\A}^{\leq 1}_4.$$

By Lemma \ref{autlem}, a general irreducible curve $C$ with one self-node has $\text{Aut}(t(C)) = \mathbb{Z}/2\mathbb{Z}$. Moreover, the extended Torelli map is injective when restricted to the locus of irreducible nodal curves \cite[Theorem 7]{Nam}. This means that the only additional contribution from the locus of irreducible nodal curves will come from the corresponding parts of the diagonal components $\Delta^+$ and $\Delta^-$. Letting $\delta_{irr}$ denote the class corresponding to the divisor of irreducible nodal curves in $\overline{\M}^{\leq 1}_4$, we thus want to find the contribution of this class to $c_1(T\overline{\M}^{\leq 1}_4)$ and $c_1(T\ol{\A}_4^{\leq 1})$.

Using the formula $c_1(T\overline{\M}^{\leq 1}_4) = 2\delta-13\lambda_1$, we find that the class $\delta_{irr}$ will contribute to $c_1(T\overline{\M}^{\leq 1}_4)$ with a multiple of $2$.

The logarithmic differentials of $\overline{\A}^{\leq 1}_4$ is given by $S^2\mathbb{E}$ with respect to the boundary divisor $D = \overline{\A}^{\leq 1}_4\backslash \A_4$ \cite[p.225]{FC}. 

To calculate the first Chern class of $\Omega_{\overline{\A}^{\leq 1}_4}$, we use the short exact sequence for logarithmic differentials
$$\bt 0\ar[r]& \Omega_{\overline{\A}^{\leq 1}_4} \ar[r] & \Omega^{\text{log}}_{\overline{\A}^{\leq 1}_4}\ar[r] &\mathcal{O}_{D}\ar[r]& 0\et$$
\noindent together with the divisor exact sequence 
$$\bt 0\ar[r]& \mathcal{O}_{\overline{\A}^{\leq 1}_4}(-D) \ar[r] & \mathcal{O}_{\overline{\A}^{\leq 1}_4}\ar[r] &\mathcal{O}_{D}\ar[r]& 0\et$$
\noindent to find $c_1(\Omega_{\overline{\A}^{\leq 1}_4}) = c_1(S^2(\mathbb{E}))+ [D] = 5\lambda_1 - [D]$. 

Since $D$ is irreducible and meets $t(\ol{\M}^{\leq 1}_4)$ transversally, the class $[D]$ pulls back to $\delta_{irr}$ under the Torelli map. Hence we obtain the contribution $\delta_{irr}$ from the irreducible nodal curves to the class $t^*c_1(T\ol{\A}^{\leq 1}_4)$.

The final contribution of $\delta_{irr}$ from $\Delta^+$ resp. $\Delta^-$ is therefore $-\delta_{irr}$. We conclude that $$t^*t_*[\overline{\M}^{\leq 1}_4]=16\lambda_1 - 2\delta_{irr}.$$ 

For a toroidal compactification for which $D = \ol{\A}_g\backslash \A_g$ is irreducible (e.g. perfect cone), using that $\text{Pic}(\A_g) = \mathbb{Q}\lambda_1$ for $g\geq 2$ \cite[Proposition 8.1]{vdG} and thus $$\text{Pic}(\ol{\A}_g) = \mathbb{Q}\lambda_1\oplus \mathbb{Q}D,$$ we deduce in this case that
$$t_*[\overline{\M}_4]|_{\overline{\A}^{\leq 1}_4}=16\lambda_1 - 2D$$\noindent and thus $$t_*[\overline{\M}_4] = 16\lambda_1 - 2D.$$\end{proof}

\begin{corollary}
    If $\ol{\A}_4$ is a toroidal compactification such that only one of the irreducible components of $D = \ol{\A}_4\backslash \A_4$ meets $t(\ol{\M}_4)$, we have $t^*t_*[\overline{\M}_4] = 16\lambda_1 - 2\delta_{irr}$.
\end{corollary}

\section{Outlook on the Torelli cycle in genus $5$}\label{fivesection}

\begin{proof}[Proof of Theorem \ref{thm4}.] The self-fiber product of the Torelli map $t\colon \M_5\to \A_5$ consists of two diagonal components, both isomorphic to $\M_5$ via projection and thus of dimension $12$:
\begin{itemize}
    \item $\Delta^+ = \{(C, C, \begin{tikzcd}
    + \colon J(C) \ar[r,shorten >=5pt, shorten <=5pt,"\sim"]& J(C))\end{tikzcd}|\hspace{5pt}C\in \M_5\}$
    \item  $\Delta^- = \{(C, C, \begin{tikzcd}
    - \colon J(C) \ar[r,shorten >=5pt, shorten <=5pt,"\sim"]& J(C))\end{tikzcd}|\hspace{5pt}C\in \M_5\}$
\end{itemize}

The intersection $\Delta^+\cap\Delta^-$ is isomorphic to $\Hy_5$ via projection to $\M_5$ and is therefore of dimension $9$. Proposition \ref{transv} shows that this intersection is transverse.

By Appendix \ref{mf}, the class $t^*t_*[\Mct_5]|_{\M_5}$ equals $2c_3(N) + m[\Hy_5]$ where $N = N_{\M_5/\A_5}$ and $m = -20$ as in Example \ref{g5}. In \cite[Theorem 5.5]{SvZ} it is shown that $[\Hy_5] = \frac{31}{30}\kappa_3$. By computing Chern characters of $\M_5$ and $\A_5$, we will find that $2c_3(N) = \frac{454}{15}\kappa_3$ and conclude that $$t^*t_*[\Mct_5]|_{\M_5} = \frac{48}{5}\kappa_3$$ 
\noindent which agrees with the formula for $t^*\text{taut}(T_5)$ in \cite[p.9]{faber}.

The class $$c_3(N) = \frac{1}{3}\left(6\text{ch}_3(N) - \text{ch}_1(N)^3 + 3\text{ch}_1(N)(\frac{1}{2}\text{ch}_1(N)^2-\text{ch}_2(N))\right)$$ is calculated from the values of $\text{ch}_i(N) = t^*\text{ch}_i(\A_5)-\text{ch}_i(\M_5)$ for $i = 1,2,3$.

The Chern characters of $\M_5$ and $\A_5$ can be found in Appendix \ref{chern} as
\begin{equation*}
    \text{ch}_1(\M_5) = -13\lambda_1,\hspace{5pt}
    \text{ch}_2(\M_5) = \frac{1}{2}\kappa_2,\hspace{5pt}
    \text{ch}_3(\M_5) = -\frac{119}{720}\kappa_3,
\end{equation*}
\begin{equation*}
    \text{ch}_1(\A_5) = -6\lambda_1,\hspace{5pt}
    \text{ch}_2(\A_5) = \lambda_2,\hspace{5pt}
    \text{ch}_3(\A_5) = \frac{1}{6}(-12\lambda_1^3+33\lambda_1\lambda_2-27\lambda_3).
\end{equation*}
Using admcycles, the class $2c_3(N)$ is then calculated as $\frac{454}{15}\kappa_3$.\end{proof}

\begin{proposition}\label{transv}
    The intersection $\Delta^+\cap \Delta^-$ is transverse.
\end{proposition}

\begin{proof} We deduce the local model for the Torelli map near a hyperelliptic curve from \cite{OS} and use it to find the scheme structure of the intersection $\Delta^+\cap \Delta^-$ in the Torelli fiber product.

Consider the lifted Torelli map $\bt M_5 ^{(n)}\ar[r,"t^{(n)}"]&A_5^{(n)}\et$ between the moduli spaces of curves and abelian varieties equipped with a level $n$ structure for $n\geq 3$. This gives an étale cover of the usual Torelli map.

As shown in \cite{OS}, the map $t^{(n)}$ factors through the space $V_5^{(n)} = M_5^{(n)}/G$ where $G$ is the action given by sending a pair $(C,\alpha)$ of a curve $C$ with level structure $\alpha$ to $(C,-\alpha)$. The intermediate map $\bt V_5^{(n)}\ar[r,"i"]&A_5^{(n)}\et$ is shown to be a closed immersion. The quotient map $\bt M_5^{(n)}\ar[r,"q"]&V_5^{(n)}\et$ is given in local coordinates by
$$\bt \mathbb{C}[[t_1,\dots,t_9,t_{10}^2,t_{10}t_{11},\dots,t_{12}^2]]\ar[r, hookrightarrow,"q^*"]&\mathbb{C}[[t_1,\dots,t_{12}]]\et$$ where the variables $t_1,\dots,t_9$ correspond to the hyperelliptic locus.

Giving $A^{(n)}_5$ local coordinates $\mathbb{C}[[x_1,\dots,x_{15}]]$ and restricting to the hyperplanes\newline $t_1=\dots=t_9 = 0$ and $x_1=\dots=x_9= 0$, the Torelli map is locally described by the map on coordinate rings \begin{equation}\label{lochyp}
    \bt \mathbb{C}[[x_{10},\dots,x_{15}]] \ar[r, "{t^{(n)}}^*"]&\mathbb{C}[[t_{10},t_{11},t_{12}]]\et
\end{equation} where the variables $x_{10},\dots, x_{15}$ map to the set of quadratic monomials in $t_{10},t_{11},t_{12}$.

The local geometry of fiber product near a point in $\Delta^+\cap \Delta^- = \Hy_5$ is determined, after cutting down dimensions, by the self-fiber product of the map in (\ref{lochyp}). This gives a union of two $3$-dimensional linear spaces intersecting in a point inside a surrounding $6$-dimensional space. This justifies the transversality of the intersection $\Delta^+\cap \Delta^-$.\end{proof}

\begin{remark}
    For genus $g\geq 6$, the class $t^*T_g|_{\M_g}$ vanishes for dimension reasons. Indeed, we have $t^*T_g|_{\M_g} = 2c_{\text{top}}(N) + m[\Hy_g] \in R^k(\M_g)$ for $k=(g^2-5g+6)/2$ and $m\in \mathbb{Z}$. Since $k>g-2$ and $R^i(\M_g)=0$ for $i>g-2$ \cite{Lo} we obtain the desired vanishing.
\end{remark}

\begin{appendices}
\appendixpage

\section{Chern characters of $\overline{\M}_{g,n}$ and $\A_g$}\label{chern}

The Chern characters of $\A_g$ can be computed using $T\A_g = S^2\mathbb{E}^\vee$. Letting $\alpha_1,\dots,\alpha_g$ be the Chern roots of $\mathbb{E}$, the Chern roots of $S^2\mathbb{E}^\vee$ are $-\alpha_i-\alpha_j$ for $1\leq i \leq j \leq g$.

We obtain 
$$\text{ch}(S^2\mathbb{E}^\vee) = \sum_{1\leq i \leq j \leq g}\text{exp}(-\alpha_i-\alpha_j).$$
For example, this gives us
\begin{equation*}
    \begin{aligned}
        \text{ch}_1(T\A_g) &= -(1+g)\lambda_1\\
        \text{ch}_2(T\A_g) &= \frac{1}{2}\left((g+3)\lambda_1^2-2(g+2)\lambda_2\right) = \lambda_2.
    \end{aligned}
\end{equation*}

Next, we compute the Chern characters of $\ol{\M}_{g,n}$. The logarithmic cotangent bundle of $\ol{\M}_{g,n}$ equals $\Omega^{\text{log}}_{\ol{\M}_{g,n}} = \pi_*(\omega_\pi^2(D)),$ and fits into the short exact sequence
\begin{equation}\label{seq}
    \bt 0 \ar[r]& \Omega_{\ol{\M}_{g,n}} \ar[r]& \Omega^{\text{log}}_{\ol{\M}_{g,n}} \ar[r]&\oplus_{\Gamma}\mathcal{O}_{\Gamma}\ar[r]&0\et
\end{equation}

\noindent where the direct sum is taken over boundary divisor strata $\Gamma$ of $\ol{\M}_{g,n}$.

The Chern character of $\Omega^{\text{log}}_{\ol{\M}_{g,n}}$ satisfies $$\text{ch}_m(\Omega^{\text{log}}_{\ol{\M}_{g,n}}) =$$
\begin{equation}\label{eq1}
      =\frac{B_{m+1}(2)}{(m+1)!}\kappa_m - \sum_{i=1}^n \frac{B_{m+1}(1)}{(m+1)!}\psi_i^m-\frac{B_{m+1}(1)}{(m+1)!}\sum_{\Gamma}\frac{1}{|\text{Aut}(\Gamma)|}{\xi_{\Gamma}}_*\left(\frac{\psi^m+(-1)^{m-1}\ol{\psi}^m}{\psi+\ol{\psi}}\right)   
\end{equation}

\noindent by Chiodo's formula (see \cite[\S 2.1]{LP}), where $B_n(x)$ is the $n$-th Bernoulli polynomial and $\psi,\ol{\psi}$ are the psi classes corresponding to the glued markings under $\xi_\Gamma$.

Using GRR for closed embeddings \cite[p.283]{Ful}, the contribution from $\oplus_{\Gamma}\mathcal{O}_{\Gamma}$ equals
\begin{equation}\label{eq2}
\begin{aligned}
    \text{ch}_m(\oplus_{\Gamma}\mathcal{O}_{\Gamma}) = \sum_{\Gamma}\frac{1}{|\text{Aut}(\Gamma)|}{\xi_{\Gamma}}_*(p(-\psi - \ol{\psi})_{m-1})
\end{aligned}
\end{equation}

\noindent where $p$ is the polynomial in $c_1(N_{\xi_{\Gamma}}) = -\psi - \ol{\psi}$ of the inverse Todd class $\text{Td}(N_{\xi_{\Gamma}})^{-1}$.

Combining equations (\ref{eq1}) and (\ref{eq2}) and using the sequence (\ref{seq}) we obtain a formula for $$\text{ch}_m(\Omega_{\ol{\M}_{g,n}}) =  \text{ch}_m(\Omega^{\text{log}}_{\ol{\M}_{g,n}})- \text{ch}_m(\oplus_{\Gamma}\mathcal{O}_{\Gamma}).$$

We obtain the Chern characters for $\Mct_{g,n}$ by removing the contributions from irreducible boundary divisors.

\section{An intersection formula}\label{mf}
Suppose we want to calculate an intersection class $X\cdot_Y V$ for varieties $X,V,Y$. Each irreducible component $A$ of the fiber product $X\times_Y V$ has a canonical contribution to this class, given by $c_{\text{top}}(N_{X,Y}/N_{A,V})$. Subtracting these canonical contributions from $X\cdot_Y V$, the remaining class will be supported on nontrivial intersections between components.

For components $A$, $B$ of the fiber product $X\times_Y V$, we provide a general formula which calculates the contribution of $A\cap B$ to the intersection class $X\cdot_Y V$. 

Let $X,Y$ be varieties equipped with a regular embedding $\bt X\ar[r,hookrightarrow,"i"] & Y\et$ of codimension $d$ and normal bundle $N$. Furthermore, let $V$ be a pure-dimensional variety of dimension $d$ and $\bt V\ar[r,"f"] & Y \et$ a morphism. Assume that $f^{-1}(X)$ consists of a union $A\cup B$ of closed subvarieties $A$, $B$ which are regularly embedded in $V$ and intersect transversely at a point $P$. Let $d_A$, $d_B$ be the dimensions of $A,B$ where $d_A,d_B\geq 1$. Note that $d=d_A+d_B$ since the intersection is transverse. Denote by $N_A$ and $N_B$ the normal bundles of $A$ and $B$ inside $V$. We omit pullback notation for $N$. We thus have a diagram:
$$\bt 
A\cup B \ar[r,hookrightarrow,"j"]\ar[d,"g"]& V\ar[d,"f"]\\
X \ar[r,hookrightarrow,"i"] & Y. \et$$

Proposition \ref{form} calculates the remaining contribution of $A\cap B$ to the intersection product $X\cdot V$ after removing the canonical contributions from $A$ and $B$.

\begin{proposition}\label{form} 

We have $X\cdot V =c_{d_A}\left(\frac{N}{N_A}\right)+c_{d_B}\left(\frac{N}{N_B}\right)+m[P]$, where $m = m(d_A,d_B)$ satisfies the formula
$$m = \sum_{1\leq k\leq d-1}(-1)^{k+1}(2^{d-k}-1)\binom{d}{k}\left[\binom{k-1}{d_A-1}(-1)^{d_A}+\binom{k-1}{d_B-1}(-1)^{d_B}\right]+2^d-2.$$

We give a proof in Section \ref{proofofmf}.
\end{proposition}

\subsection{Example computations}
We first illustrate with example computations, and compare them to suitable local models with the given dimensions.

\begin{example}\label{g5} Let $d_A=d_B=3$. Proposition \ref{form} tells us that $m  =-20.$ 

A local model for this intersection is given by $X=V=\pr^6$ and $Y = \mathcal{O}(2)^{\oplus 6}$ a vector bundle whose zero section $$(x_0^2-x_3^2,x_1^2-x_4^2,x_2^2-x_5^2,x_0x_1-x_3x_4,x_0x_2-x_3x_5,x_1x_2-x_4x_5)$$ cuts out the union of two $\pr^3$'s (denoted $A$ and $B)$ intersecting in a point $P$.

Using $N = \mathcal{O}(2)^{\oplus 6}$, $N_A = N_B = \mathcal{O}(1)^{\oplus 3}$, we calculate
$$c_3\left(\frac{N}{N_A}\right) = \left[\frac{1+12H+60H^2+160H^3}{1+3H+3H^2+H^3}\right]_0 = 42[\text{pt}].$$

This indeed gives us $m = c_6(\mathcal{O}(2)^{\oplus 6}) - 2\cdot 42 = -20$.
\end{example}

\begin{example}\label{Bcoeff} Let $d_A =2$, $d_B=1$. The formula gives $m = -3.$
We compare this with the local model of the vector bundle $\mathcal{O}\oplus \mathcal{O}(2)^{\oplus 2}$ on $\pr^3$, where the section $(0,xz,yz)$ cuts out $A\cup B$ where $A=\pr^2$, $B = \pr^1$ and $A\cap B = P$.

With $N =\mathcal{O}(2)^{\oplus 2}$, $N_A = \mathcal{O}(1)$, $N_B = \mathcal{O}(1)^{\oplus 2}$, we obtain
$$c_2\left(\frac{N}{N_A}\right) = \left[\frac{1+4H+4H^2}{1+H}\right]_0 = 1[\text{pt}]$$
and $$c_1\left(\frac{N}{N_B}\right) = \left[\frac{1+4H+4H^2}{1+2H+H^2}\right]_0 = 2[\text{pt}].$$

Combining gives $m = c_3(\mathcal{O}\oplus \mathcal{O}(2)^{\oplus 2})-1-2 = -3$.
\end{example}

\begin{example}\label{Acoeff} Let $d_A=1, d_B=1$. The formula gives $m = -2.$ This agrees with the local model for the vector bundle $\mathcal{O}(2)\oplus \mathcal{O}$ over $\pr^2$, with the section $(xy,0)$ cutting out the union of two transverse lines.
\end{example}

\subsection{Motivation for the formula}
We indicate why the formula in Proposition \ref{form} is expected.

The difference $X\cdot V -  c_{d_A}\left(\frac{N}{N_A}\right)+c_{d_B}\left(\frac{N}{N_B}\right)$ equals $$\left\{c(N)\cap \left(s(A\cup B,V) - s(A,V)-s(B,V)\right)\right\}_0.$$ The normal cone of $A\cup B$ in $V$ agrees with the one of $A$ in $V$ on the complement $A\backslash P$ and of $B$ in $V$ over $B\backslash P$. Taking Segre classes, the difference $$\delta = s(A\cup B,V) - s(A,V)-s(B,V)$$ is supported on $P$ by the following argument: Restrict to $U = V\backslash P$ and consider the inclusion $\bt U\ar[r,hookrightarrow,"\iota"] & V.\et$ We use the same notation $\iota$ for the restricted map $\bt A\backslash P\cup B \backslash P \hookrightarrow A\cup B.\et$ Since open immersions are flat, the Segre classes pull back: $$\iota^*\delta =s(A\backslash P\cup B\backslash P,U) - s(A\backslash P,U)-s(B\backslash P,U).$$

Due to equality of normal cones on the open set $U$, the class $\iota^*\delta$ vanishes. This in turn means that the class $\delta$ must be supported on $V\backslash U = P$.

That the coefficient $m$ only depends on the dimensions $d_A,d_B$ translates to the statement that it only depends on the normal bundles of $P$ inside $A,B$. The proof of Proposition \ref{form} will indeed show that the coefficient can be calculated using only the local geometry near the exceptional divisor $E$ of $P$ after blowing up $P$ and the dimensions of the linear subspaces corresponding to the intersections of the strict transforms of $A, B$ with $E$.

\subsection{Proof of the intersection formula} \label{proofofmf}
\begin{proof}[Proof of Proposition \ref{form}]
Let $\bt V' \ar[r,"\pi'"]& V\et$ be the blowup of $V$ at $P$. Let $E$ be the exceptional divisor, and $A'$, $B'$ the strict transforms of $A$, $B$.
Furthermore, let $\bt\tilde{V}\ar[r,"\tilde{\pi}"] & V'\et$ be the blowup of $V'$ in $A'\cup B'$. Let $\tilde{A}$, $\tilde{B}$ be the components of the exceptional divisor of $\tilde{\pi}$ corresponding to the preimages of $A'$, $B'$, and $\tilde{E}$ the strict transform of $E$.

We have a fiber diagram
    $$\bt 
\tilde{F}\ar[d,"\tilde{q}"] \ar[r,hookrightarrow,"l"]& \ar[d,"\tilde{\pi}"]\tilde{V}\\
F'\ar[d,"q'"] \ar[r,hookrightarrow,"k"]& \ar[d,"\pi'"]V'\\
F \ar[r,hookrightarrow,"j"]\ar[d,"g"]& V\ar[d,"f"]\\
X \ar[r,hookrightarrow,"i"] & Y, \et$$

\noindent where $F = A\cup B$, $F' = A'\cup B' \cup 2E$ and $\tilde{F} = \tilde{A} + \tilde{B}+ 2\tilde{E}$, the latter being a sum of Cartier divisors on $\tilde{V}$. That the preimage of $E$ under $i$ is $2E$ follows from the transversality of the intersection $A\cap B$ in $V$.

We abbreviate $\pi = \pi'\circ \tilde{\pi}$ and $q = q'\circ \tilde{q}$, and omit all pullback notation for the normal bundle $N$.

The intersection class $X\cdot V$ can be calculated from $q_*(X\cdot \tilde{V})$ by the birational invariance of Segre classes. The class $q_*(X\cdot \tilde{V})$ equals $$q_*\{c(N)\cap s(\tiA + \tiB + 2\tiE,\tiV)\}_0.$$ By birational invariance, the classes $s(\tiA+\tiE,\tiV)$ and $s(\tiB+\tiE,\tiV)$ push forward to $s(A,V)$ and $s(B,V)$ respectively under $q$.

As shown in the prelude, the class $$s(A\cup B,V) - s(A,V) - s(B,V)$$ is supported on $P$, and must therefore be a multiple of $P$. Note that $N$ is trivial on $P$, so capping with $c(N)$ and taking the degree $0$ part will not change this class.

Since $\{c(N)\cap s(A,V)\}_0 = c_{d_A}\left(\frac{N}{N_A}\right)$ and $\{c(N)\cap s(B,V)\}_0= c_{d_B}\left(\frac{N}{N_B}\right)$, we obtain $$X\cdot V = c_{d_A}\left(\frac{N}{N_A}\right)+c_{d_B}\left(\frac{N}{N_B}\right) + m[P]$$ for some multiple $m$.

The above reasoning shows that 
\begin{align*}
    m &= q_*\left(s_d(\tiA + \tiB + 2\tiE,\tiV)-s_d(\tiA+\tiE,\tiV)-s_d(\tiB+\tiE,\tiV)\right)
\end{align*}
\noindent which can be expanded as
\begin{align*}
    m &=  (-1)^{d-1} q_*\left([\tiA + \tiB + 2\tiE]^d-[\tiA+\tiE]^d-[\tiB + \tiE]^d\right)\\
    &= (-1)^{d-1}q_*\sum_{j=1}^{d-1}\binom{d}{j}(\tiA+\tiE)^j(\tiB+\tiE)^{d-j}\\
    &= (-1)^{d-1} q_*\left(\sum_{j=1}^{d-1}\sum_{k=1}^{j}\binom{d}{j}\binom{j}{k}[\tiA^k\tiE^{d-k} + \tiB^k\tiE^{d-k}]\right) + (-1)^{d-1}(2^d-2)q_*\tiE^d. 
\end{align*}

In the last line, we used the fact that $\tiA$ and $\tiB$ are disjoint and therefore intersect trivially.

To simplify this expression, we use the combinatorial identity
$$\sum_{i=1}^{d-1}\binom{d}{j}\binom{j}{k} = (2^{d-k}-1)\binom{d}{k}.$$

Moreover, we have $\pi^*E = \tilde{E}$, giving $\tiE^d = \pi^*E^d = (-1)^{d-1}$.

It remains to calculate the value of $\tiA^k\tiE^{d-k}$ for $1\leq k \leq d-1$.

Let $L=A'\cap E$ and  $\tilde{L}=\tiA\cap \tiE$. Let $q'_A$, $\tilde{q}_A$ be the restrictions of $q'$, $\tilde{q}$ to $A'$, $\tiA$ respectively.

Note that $$\tilde{q}_*\tiA^k = (-1)^{k+1}\tilde{q}_*s_{k-1}(\tiA,\tiV) = (-1)^{k+1}s_{k-d_B}(A',V')$$
and that $N_{A'/V'} \cong q'^*_AN_{A/V}\otimes \mathcal{O}(-L)$ (see \cite[p.437]{Ful}). Moreover, we have $\tiE|_{\tiA} = \tilde{L}= \tilde{q}^*_AL$.

Since $q'^*_AN_{A/V}$ is trivial on $L$ and the codimension of $A$ inside $V$ is $d_B$, we can write
\begin{align*}
   \tilde{q}_*\tiA^k\tiE^{d-k} &=  \tilde{q}_*(\tiA^{k-1}\cdot_{\tiA} \tilde{L}^{d-k})  = {{\tilde{q}}_{A*}}(\tiA^k\cdot {\tilde{q}_A}^*L^{d-k})\\
    & = (-1)^{k+1}s_{k-d_B}(\mathcal{O}(-L)^{\oplus d_B})\cdot_{A'} L^{d-k} = \binom{k-1}{d_B-1}L^{d-d_B}(-1)^{k+1}\\
    & =  \binom{k-1}{d_B-1}(-1)^{d-d_B+k}.
\end{align*}
\noindent where the last equality holds after taking degree. Combining these results, with a similar calculation for $\tiB$, yields the desired formula.
\end{proof}

\subsection{Generalizations of the formula}
The below propositions are natural generalizations of the formula in Proposition \ref{form}.

\begin{proposition}\label{dimgreat}
    Assume that $Z=A\cap B$ has dimension $d_1$ and $X\cdot V$ has dimension $d_2$ where $d_1\geq d_2$. Let $d_A$, $d_B$ be the dimensions of $A$, $B$. We obtain a formula $$X\cdot V = c_{d_A}\left(\frac{N}{N_A}\right)+c_{d_B}\left(\frac{N}{N_B}\right)+R$$ where $R$ is supported on $Z$ and depends only on the normal bundles of $Z$ in $A$, $B$ and $V$.
\end{proposition}

\begin{proof} Here $d = d_A+d_B-d_1$ is the dimension of $V$. The class \begin{align*}
    R &= q_*\left[c(N)\cap\left(s(\tiA + \tiB + 2\tiE,\tiV)-s(\tiA+\tiE,\tiV)-s(\tiB+\tiE,\tiV)\right)\right]_{d_2}\\
    &= \left[c(N)\cap q_*\left(\sum_{n=1}^d(-1)^{n-1}\left([\tiA + \tiB + 2\tiE]^n-[\tiA+\tiE]^n-[\tiB+\tiE]^n\right)\right)\right]_{d_2}
   \end{align*}
\noindent depends only on the classes $c(N)$, $q_*\tiA^i\tiE^j$, $q_*\tiB^i\tiE^j$ for $i\geq 1, j\geq 0$ and $q_*\tiE^n$ for $n\geq 0$. Note that $$E^n = c_1(\mathcal{O}_{\pr(N_{Z,V})}(-1))^{n-1}$$ pushes down to $(-1)^{n-1}s_{d_Z-d+n}(N_{Z,V})$, and $$ \tilde{q}_*\tiA^i\cdot \tiE^{j} = \binom{i-1}{d-d_A-1}(-1)^{i+1}L^{i-d+d_A+j}$$ where $$L^{i-d+d_A+j} = c_1(\mathcal{O}_{\pr(N_{Z,A})}(-1))^{i-d+d_A+j-1}$$ pushes down to $(-1)^{i-d+d_A+j-1}s_{d_Z-d+i+j}(N_{Z,A})$ and similarly for $q_*\tiB^i\cdot \tiE^{j}$. We conclude that $R$ is only dependent on the Chern/Segre classes of $N$, $N_{Z,A}$, $N_{Z,B}$ and $N_{Z,V}$.\end{proof}

\begin{proposition}
    Assume that $X\cdot V$ has dimension $k\geq 0$ where $\mathrm{dim}(A\cap B) = k.$ Then $m = m(d_A-k,d_B-k)$. 
\end{proposition}

\begin{proof}
   Noting that $q_*s_i$ vanishes for each Segre class $s_i$ on $\tilde{F}$ where $i<d-k$ and that the class $X\cdot V$ is $k$-dimensional, we obtain
\begin{align*}
    m[Z] &= (-1)^{d-k}q_*\left(s_{d-k}(\tiA + \tiB + 2\tiE,\tiV)-s_{d-k}(\tiA+\tiE,\tiV)-s_{d-k}(\tiB+\tiE,\tiV)\right)
\end{align*}
\noindent where $d=d_A+d_B-k$.

Letting $d'_A = d_A-k$, $d'_B = d_B-k$ gives $d-k = d'_A + d_B'$. Imitating the proof of Proposition $\ref{form}$ thus yields $m= m(d_A',d_B')$.\end{proof}

\begin{proposition}
    Assume more generally that $d\geq d_A+d_B$, where $P=A\cap B$ is a smooth point. Then the above formula for $m = m(d_A,d_B)$ still holds.
\end{proposition}

\begin{proof}
    Let $\tilde{V}$ be the blowup of $V$ in $B$. Let $\tilde{B}$ be the exceptional divisor and $A' = \text{Bl}_PA$. Let $q$ be the induced map from $A'\cup \tilde{B}$ to $A\cup B$. Denote by $Z$ the preimage of $P$ under this blowup. By birational invariance of Segre classes, we have 
    \begin{equation*}
    \begin{aligned}
        s(A\cup B,V) &= q_*s(A'\cup \tiB,\tiV),\\
        s(A,V) &= q_*s(A',\tiV),\\
        s(B,V) &= q_*s(\tiB,\tiV).\\
    \end{aligned}
    \end{equation*}

In turn, this means that $m[P] = q_*R$ where $$R = q_*s(A'\cup \tiB,\tiV)- q_*s(A',\tiV)- q_*s(\tiB,\tiV)$$ only depends on the normal bundles of $Z$ in $A'$, $\tiB$ and $\tiV$ by Corollary \ref{dimgreat}. The normal bundle of $Z$ in $\tiB$ is trivial since $\tiB= \mathbb{P}(N_{P,B})$, whereas the normal bundle of $Z$ in $A$ is $\mathcal{O}_{\mathbb{P}(N_{P,A})}(-1)$. The normal bundle of $Z$ in $V$ equals, in the Grothendieck group $K_0$, the sum $N_{Z,\tiB} + i_Z^*N_{\tiB,\tiV}$ where $i_Z$ denotes the inclusion of $Z$ into $B$. The normal bundle $N_{\tiB,\tiV}$ equals $\mathcal{O}_{\mathbb{P}(N_{B,V})}(-1)$. Thus, after restricting to $P$, these normal bundles will only be dependent on the normal bundles of $P$ inside $A,B$ and $V$. As a result, the coefficient $m$ depends only on the values $d_A,d_B,d$. Due to invariance of $m$ after taking the product of $V$ resp. $Y$ by $\mathbb{A}^1$, $m$ remains unchanged after replacing $d$ by $d-k$ while keeping $d_A,d_B$ the same. These new dimensions model a transverse intersection, meaning that $m = m(d_A,d_B)$ in the intersection formula.
\end{proof}

\section{MATLAB code for Section \ref{nonred2}}\label{matlab}

\begin{lstlisting}[
  style      = Matlab-editor,
  basicstyle = \mlttfamily,
  caption = {MATLAB code for calculating the value of $\rho_4$ in Section \ref{nonred2}.}
]
f1 = @(x) (x-1).*(x-2).*(x-3).*(x-4).*(x-5).*(x-6);
f2 = @(x) (x-1).*(x-2).*(x-3).*(x-4).*(x-5).*(x-7);

%Sum of arguments for A1, A2 using the interval (-pi,pi) around all
%branch points:
upperargA1 = 5*pi;
lowerargA1 = -5*pi;
upperargA2 = 3*pi;
lowerargA2 = -3*pi;
argnearzero = 6*pi;

%Let e1 = dx/y, e2 = xdx/y and v1 = ae1+be2, v2 = ce1+de2 be normalized
%with respect to A1, A2.

y1_upperzero = @(t) sqrt(abs(f1(t))).*exp(1i*argnearzero/2);
y1_upperA1 = @(t) sqrt(abs(f1(t))).*exp(1i*upperargA1/2);
y1_upperA2 = @(t) sqrt(abs(f1(t))).*exp(1i*upperargA2/2);
intA1ofe1 = 2*integral(@(t) 1./y1_upperA1(t),1,2);
intA1ofe2 = 2*integral(@(t) t./y1_upperA1(t),1,2);
intA2ofe1 = 2*integral(@(t) 1./y1_upperA2(t),3,4);
intA2ofe2 = 2*integral(@(t) t./y1_upperA2(t),3,4);

det = intA1ofe1*intA2ofe2 - intA1ofe2*intA2ofe1;

a = intA2ofe2/det;
b = -intA2ofe1/det;
c = -intA1ofe2/det;
d = intA1ofe1/det;

%Integrals over small circles around branch points of K_1 vanish. Thus
%they have no contribution to the contour integral.

%Assume that z1 lies in a small neigborhood of 0 and hence outside A1. Then the
%integral over A1 of 1/(y(z)(z-z1)) can be made over 2*(upper part of A1). 
%This is because both the integrand and the path direction change sign.
%Use the expansion 1/(t-z1) = 1/t + z1/t^2 + z1^2/t^3 +
%O(z1^3) near z_1 = 0, and that only these first three terms matter. 

%We obtain the integrals (1/2)*y1(z1)*2*int(1/t^i*y1upperA2,1,2) for t=1,2,3.
%As polynomials in z1, let y1(z1) = c0 + c1z1 + c2z1^2 and
% int(1/t^i*y1upperA2,1,2) = d0 + d1z1+ d2z1^2.

d0 = integral(@(t) 1./(y1_upperA1(t).*t),1,2);
d1 = integral(@(t) 1./(y1_upperA1(t).*t.^2),1,2);
d2 = integral(@(t) 1./(y1_upperA1(t).*t.^3),1,2);

c0 = real(y1_upperzero(0));

syms x
symbolicf1 = (x-1).*(x-2).*(x-3).*(x-4).*(x-5).*(x-6);
df1_atzero = subs(diff(symbolicf1),0);

%Implicit differentiation gives 2*y_1(0)*y_1'(0) = f1'(0)
%and 2y'(0)^2 + 2y(0)y''(0) = f1''(0) = 3248

dy1_atzero = df1_atzero/(2*c0);
c1 = dy1_atzero;
%dy1_atzero = 32.8702

ddy1_atzero = (3248-2*(dy1_atzero)^2)/(2*c0);
c2 = (ddy1_atzero)/2;

%Create h1 = (c0+c1*z1+c2*z1^2)*(d0+d1*z1+d2*z1^2) up
%to o(z1^3).
k2 = double(c0*d2+c1*d1+c2*d0); %This is the z1^2 term of h1, which is all we need.

%Repeat for A2:
d0_2 = integral(@(t) 1./(y1_upperA2(t).*t),3,4);
d1_2 = integral(@(t) 1./(y1_upperA2(t).*t.^2),3,4);
d2_2 = integral(@(t) 1./(y1_upperA2(t).*t.^3),3,4);

k2_2 = double(c0*d2_2+c1*d1_2+c2*d0_2);
%The z1^2 term of h2.

%Now K1 = (1/2)*(y(z)+y1(z1))/((z-z1)*y(z)) - k2*z1^2*v1- k2_2*z1^2*v2

%Use k2*v1+k2_2*v2 = (k2*a+k2_2*c)e1 + (k2*b+k2_2*d)e2

m1 = k2*a + k2_2*c;
m2 = k2*b + k2_2*d;

%K1 = (1/2)*(y(z)+y1(z1))/((z-z1)*y(z)) - m1*z1^2*e1 - m2*z2^2*e2

%Have G1 = int over B1 of K1/z1^3, G2 = int over B1 of K1/z1^3.

%The integral over z1 of K1/z1^3 is given by the residue at z1 = 0.
%There is no residue over z1 = z since 0 lies far away from B1,B2.

%The residue at z1 = 0 is given by (1/2)*(d/dz1)^2 ((1/2)*(y(z)+y1(z1))/((z-z1)*y(z)))
%(We have multiplied up by z1^3 before differentiating) 
%This equals Res0 = @(z) (1/4)*(c2*z.^2+2*(c1*z+c0+y1(z)))./(z.^3) where 
%y1 is suitably chosen depending on the branch at the path we integrate over.
 
%Here we define arguments. Note B2 = B2_upper + B2_lower, with same arguments
%arg_45 but y(z) on upper becomes -y(z) on lower.
%B1 = B1_upper23 + B1_upper45 + B1_lower23 + B1_lower45
%arg_23, arg_45.

arg_23 = 4*pi;
arg_45 = 2*pi;

y1_mainbranch23 = @(t) sqrt(abs(f1(t))).*exp(1i*arg_23/2);
y1_otherbranch23 = @(t) -y1_mainbranch23(t);
y1_mainbranch45 = @(t) sqrt(abs(f1(t))).*exp(1i*arg_45/2);
y1_otherbranch45 = @(t) -y1_mainbranch45(t);

int_upper23ofK1 = integral(@(t) double( (-1./y1_mainbranch23(t)).*(m1 + m2.*t) + (1./y1_mainbranch23(t)).*(1/4).*(c2.*t.^2+2.*(c1.*t+c0+y1_mainbranch23(t)))./(t.^3) ),2,3);
%Recall that we are subtracting the m1+m2t part from K1

int_lower23ofK1 = integral(@(t) double( (-1./y1_otherbranch23(t)).*(m1 + m2.*t) + (1./y1_otherbranch23(t)).*(1/4).*(c2.*t.^2+2.*(c1.*t+c0+y1_otherbranch23(t)))./(t.^3) ),2,3);

int_upper45ofK1 = integral(@(t) double( (-1./y1_mainbranch45(t)).*(m1 + m2.*t) + (1./y1_mainbranch45(t)).*(1/4).*(c2.*t.^2+2.*(c1.*t+c0+y1_mainbranch45(t)))./(t.^3) ),4,5);

int_lower45ofK1 = integral(@(t) double( (-1./y1_otherbranch45(t)).*(m1 + m2.*t) + (1./y1_otherbranch45(t)).*(1/4).*(c2.*t.^2+2.*(c1.*t+c0+y1_otherbranch45(t)))./(t.^3) ),4,5);

G1 = int_upper23ofK1 - int_lower23ofK1 + int_upper45ofK1 - int_lower45ofK1; %since intlower is in the opposite path direction we need to alter the sign
G2 = int_upper45ofK1 - int_lower45ofK1;

y2_upperzero = @(t) sqrt(abs(f2(t))).*exp(1i*argnearzero/2);
l0 = real(y2_upperzero(0));

syms x
symbolicf2 = (x-1).*(x-2).*(x-3).*(x-4).*(x-5).*(x-7);
df2_atzero = subs(diff(symbolicf2),0);
dy2_atzero = df2_atzero/(2*l0);
l1 = dy2_atzero;

D1 = (a*l1-b)/l0;
D2 = (c*l1-d)/l0;

disp(double(-c^2*G1*D1 + a*c*G1*D2 + a*c*G2*D1 - a^2*G2*D2))
%= 0.0256i is nonzero as desired.

%rho_4 = (-2)*(-c^2*G1*D1 + a*c*G1*D2 + a*c*G2*D1 - a^2*G2*D2)*(c0^2*l0^2)/(a*d - b*c)^2;
%disp(imag(double(rho_4)))
%rho_4 = -2.5871*10^5
\end{lstlisting}

\end{appendices}

\bigskip
\bigskip

Lycka~Drakengren\par\nopagebreak
\textsc{Department of Mathematics, ETH Zürich}\par\nopagebreak
\textsc{Rämistrasse 101, 8092 Zürich, Switzerland}\par\nopagebreak
\textit{E-mail address}: \texttt{lycka.drakengren@math.ethz.ch}

\end{document}

%% file: bipartite.tex
\tikzset{every picture/.style={line width=0.75pt}} %set default line width to 0.75pt        

\begin{tikzpicture}[x=0.75pt,y=0.75pt,yscale=-1,xscale=1]
%uncomment if require: \path (0,467); %set diagram left start at 0, and has height of 467

%Shape: Circle [id:dp6135404859595076] 
\draw  [fill={rgb, 255:red, 0; green, 0; blue, 0 }  ,fill opacity=1 ] (131.67,52.17) .. controls (131.67,48.76) and (134.43,46) .. (137.83,46) .. controls (141.24,46) and (144,48.76) .. (144,52.17) .. controls (144,55.57) and (141.24,58.33) .. (137.83,58.33) .. controls (134.43,58.33) and (131.67,55.57) .. (131.67,52.17) -- cycle ;
%Shape: Circle [id:dp5279916739816132] 
\draw  [fill={rgb, 255:red, 0; green, 0; blue, 0 }  ,fill opacity=1 ] (131.67,186.17) .. controls (131.67,182.76) and (134.43,180) .. (137.83,180) .. controls (141.24,180) and (144,182.76) .. (144,186.17) .. controls (144,189.57) and (141.24,192.33) .. (137.83,192.33) .. controls (134.43,192.33) and (131.67,189.57) .. (131.67,186.17) -- cycle ;
%Shape: Circle [id:dp09703946253347795] 
\draw  [fill={rgb, 255:red, 0; green, 0; blue, 0 }  ,fill opacity=1 ] (131.67,234.17) .. controls (131.67,230.76) and (134.43,228) .. (137.83,228) .. controls (141.24,228) and (144,230.76) .. (144,234.17) .. controls (144,237.57) and (141.24,240.33) .. (137.83,240.33) .. controls (134.43,240.33) and (131.67,237.57) .. (131.67,234.17) -- cycle ;
%Shape: Circle [id:dp9525497989471352] 
\draw  [fill={rgb, 255:red, 0; green, 0; blue, 0 }  ,fill opacity=1 ] (131.67,280.17) .. controls (131.67,276.76) and (134.43,274) .. (137.83,274) .. controls (141.24,274) and (144,276.76) .. (144,280.17) .. controls (144,283.57) and (141.24,286.33) .. (137.83,286.33) .. controls (134.43,286.33) and (131.67,283.57) .. (131.67,280.17) -- cycle ;
%Shape: Circle [id:dp8705129581884626] 
\draw  [fill={rgb, 255:red, 0; green, 0; blue, 0 }  ,fill opacity=1 ] (131.67,94.17) .. controls (131.67,90.76) and (134.43,88) .. (137.83,88) .. controls (141.24,88) and (144,90.76) .. (144,94.17) .. controls (144,97.57) and (141.24,100.33) .. (137.83,100.33) .. controls (134.43,100.33) and (131.67,97.57) .. (131.67,94.17) -- cycle ;
%Shape: Circle [id:dp484201566440419] 
\draw  [fill={rgb, 255:red, 0; green, 0; blue, 0 }  ,fill opacity=1 ] (132.67,323.17) .. controls (132.67,319.76) and (135.43,317) .. (138.83,317) .. controls (142.24,317) and (145,319.76) .. (145,323.17) .. controls (145,326.57) and (142.24,329.33) .. (138.83,329.33) .. controls (135.43,329.33) and (132.67,326.57) .. (132.67,323.17) -- cycle ;
%Shape: Circle [id:dp6592041904845434] 
\draw  [fill={rgb, 255:red, 0; green, 0; blue, 0 }  ,fill opacity=1 ] (132.67,374.17) .. controls (132.67,370.76) and (135.43,368) .. (138.83,368) .. controls (142.24,368) and (145,370.76) .. (145,374.17) .. controls (145,377.57) and (142.24,380.33) .. (138.83,380.33) .. controls (135.43,380.33) and (132.67,377.57) .. (132.67,374.17) -- cycle ;
%Shape: Circle [id:dp07519998601160549] 
\draw  [fill={rgb, 255:red, 0; green, 0; blue, 0 }  ,fill opacity=1 ] (131.67,141.17) .. controls (131.67,137.76) and (134.43,135) .. (137.83,135) .. controls (141.24,135) and (144,137.76) .. (144,141.17) .. controls (144,144.57) and (141.24,147.33) .. (137.83,147.33) .. controls (134.43,147.33) and (131.67,144.57) .. (131.67,141.17) -- cycle ;
%Shape: Circle [id:dp8772732926888551] 
\draw  [fill={rgb, 255:red, 0; green, 0; blue, 0 }  ,fill opacity=1 ] (232.67,52.17) .. controls (232.67,48.76) and (235.43,46) .. (238.83,46) .. controls (242.24,46) and (245,48.76) .. (245,52.17) .. controls (245,55.57) and (242.24,58.33) .. (238.83,58.33) .. controls (235.43,58.33) and (232.67,55.57) .. (232.67,52.17) -- cycle ;
%Shape: Circle [id:dp1700535948819173] 
\draw  [fill={rgb, 255:red, 0; green, 0; blue, 0 }  ,fill opacity=1 ] (232.67,186.17) .. controls (232.67,182.76) and (235.43,180) .. (238.83,180) .. controls (242.24,180) and (245,182.76) .. (245,186.17) .. controls (245,189.57) and (242.24,192.33) .. (238.83,192.33) .. controls (235.43,192.33) and (232.67,189.57) .. (232.67,186.17) -- cycle ;
%Shape: Circle [id:dp32703299477782954] 
\draw  [fill={rgb, 255:red, 0; green, 0; blue, 0 }  ,fill opacity=1 ] (232.67,234.17) .. controls (232.67,230.76) and (235.43,228) .. (238.83,228) .. controls (242.24,228) and (245,230.76) .. (245,234.17) .. controls (245,237.57) and (242.24,240.33) .. (238.83,240.33) .. controls (235.43,240.33) and (232.67,237.57) .. (232.67,234.17) -- cycle ;
%Shape: Circle [id:dp39338649647112955] 
\draw  [fill={rgb, 255:red, 0; green, 0; blue, 0 }  ,fill opacity=1 ] (232.67,280.17) .. controls (232.67,276.76) and (235.43,274) .. (238.83,274) .. controls (242.24,274) and (245,276.76) .. (245,280.17) .. controls (245,283.57) and (242.24,286.33) .. (238.83,286.33) .. controls (235.43,286.33) and (232.67,283.57) .. (232.67,280.17) -- cycle ;
%Shape: Circle [id:dp1250140774289059] 
\draw  [fill={rgb, 255:red, 0; green, 0; blue, 0 }  ,fill opacity=1 ] (232.67,94.17) .. controls (232.67,90.76) and (235.43,88) .. (238.83,88) .. controls (242.24,88) and (245,90.76) .. (245,94.17) .. controls (245,97.57) and (242.24,100.33) .. (238.83,100.33) .. controls (235.43,100.33) and (232.67,97.57) .. (232.67,94.17) -- cycle ;
%Shape: Circle [id:dp16090886924466907] 
\draw  [fill={rgb, 255:red, 0; green, 0; blue, 0 }  ,fill opacity=1 ] (233.67,323.17) .. controls (233.67,319.76) and (236.43,317) .. (239.83,317) .. controls (243.24,317) and (246,319.76) .. (246,323.17) .. controls (246,326.57) and (243.24,329.33) .. (239.83,329.33) .. controls (236.43,329.33) and (233.67,326.57) .. (233.67,323.17) -- cycle ;
%Shape: Circle [id:dp9558762300402436] 
\draw  [fill={rgb, 255:red, 0; green, 0; blue, 0 }  ,fill opacity=1 ] (233.67,374.17) .. controls (233.67,370.76) and (236.43,368) .. (239.83,368) .. controls (243.24,368) and (246,370.76) .. (246,374.17) .. controls (246,377.57) and (243.24,380.33) .. (239.83,380.33) .. controls (236.43,380.33) and (233.67,377.57) .. (233.67,374.17) -- cycle ;
%Shape: Circle [id:dp6176776008895739] 
\draw  [fill={rgb, 255:red, 0; green, 0; blue, 0 }  ,fill opacity=1 ] (232.67,141.17) .. controls (232.67,137.76) and (235.43,135) .. (238.83,135) .. controls (242.24,135) and (245,137.76) .. (245,141.17) .. controls (245,144.57) and (242.24,147.33) .. (238.83,147.33) .. controls (235.43,147.33) and (232.67,144.57) .. (232.67,141.17) -- cycle ;
%Shape: Circle [id:dp4066689462781211] 
\draw  [fill={rgb, 255:red, 0; green, 0; blue, 0 }  ,fill opacity=1 ] (415.67,51.17) .. controls (415.67,47.76) and (418.43,45) .. (421.83,45) .. controls (425.24,45) and (428,47.76) .. (428,51.17) .. controls (428,54.57) and (425.24,57.33) .. (421.83,57.33) .. controls (418.43,57.33) and (415.67,54.57) .. (415.67,51.17) -- cycle ;
%Shape: Circle [id:dp8694036910222835] 
\draw  [fill={rgb, 255:red, 0; green, 0; blue, 0 }  ,fill opacity=1 ] (415.67,185.17) .. controls (415.67,181.76) and (418.43,179) .. (421.83,179) .. controls (425.24,179) and (428,181.76) .. (428,185.17) .. controls (428,188.57) and (425.24,191.33) .. (421.83,191.33) .. controls (418.43,191.33) and (415.67,188.57) .. (415.67,185.17) -- cycle ;
%Shape: Circle [id:dp9288871713244938] 
\draw  [fill={rgb, 255:red, 0; green, 0; blue, 0 }  ,fill opacity=1 ] (415.67,233.17) .. controls (415.67,229.76) and (418.43,227) .. (421.83,227) .. controls (425.24,227) and (428,229.76) .. (428,233.17) .. controls (428,236.57) and (425.24,239.33) .. (421.83,239.33) .. controls (418.43,239.33) and (415.67,236.57) .. (415.67,233.17) -- cycle ;
%Shape: Circle [id:dp4099960754086288] 
\draw  [fill={rgb, 255:red, 0; green, 0; blue, 0 }  ,fill opacity=1 ] (415.67,279.17) .. controls (415.67,275.76) and (418.43,273) .. (421.83,273) .. controls (425.24,273) and (428,275.76) .. (428,279.17) .. controls (428,282.57) and (425.24,285.33) .. (421.83,285.33) .. controls (418.43,285.33) and (415.67,282.57) .. (415.67,279.17) -- cycle ;
%Shape: Circle [id:dp11479432525339706] 
\draw  [fill={rgb, 255:red, 0; green, 0; blue, 0 }  ,fill opacity=1 ] (415.67,93.17) .. controls (415.67,89.76) and (418.43,87) .. (421.83,87) .. controls (425.24,87) and (428,89.76) .. (428,93.17) .. controls (428,96.57) and (425.24,99.33) .. (421.83,99.33) .. controls (418.43,99.33) and (415.67,96.57) .. (415.67,93.17) -- cycle ;
%Shape: Circle [id:dp9292284960237713] 
\draw  [fill={rgb, 255:red, 0; green, 0; blue, 0 }  ,fill opacity=1 ] (416.67,322.17) .. controls (416.67,318.76) and (419.43,316) .. (422.83,316) .. controls (426.24,316) and (429,318.76) .. (429,322.17) .. controls (429,325.57) and (426.24,328.33) .. (422.83,328.33) .. controls (419.43,328.33) and (416.67,325.57) .. (416.67,322.17) -- cycle ;
%Shape: Circle [id:dp8404621070553461] 
\draw  [fill={rgb, 255:red, 0; green, 0; blue, 0 }  ,fill opacity=1 ] (416.67,373.17) .. controls (416.67,369.76) and (419.43,367) .. (422.83,367) .. controls (426.24,367) and (429,369.76) .. (429,373.17) .. controls (429,376.57) and (426.24,379.33) .. (422.83,379.33) .. controls (419.43,379.33) and (416.67,376.57) .. (416.67,373.17) -- cycle ;
%Shape: Circle [id:dp7160748902713238] 
\draw  [fill={rgb, 255:red, 0; green, 0; blue, 0 }  ,fill opacity=1 ] (415.67,140.17) .. controls (415.67,136.76) and (418.43,134) .. (421.83,134) .. controls (425.24,134) and (428,136.76) .. (428,140.17) .. controls (428,143.57) and (425.24,146.33) .. (421.83,146.33) .. controls (418.43,146.33) and (415.67,143.57) .. (415.67,140.17) -- cycle ;
%Shape: Circle [id:dp28971409513948565] 
\draw  [fill={rgb, 255:red, 0; green, 0; blue, 0 }  ,fill opacity=1 ] (516.67,51.17) .. controls (516.67,47.76) and (519.43,45) .. (522.83,45) .. controls (526.24,45) and (529,47.76) .. (529,51.17) .. controls (529,54.57) and (526.24,57.33) .. (522.83,57.33) .. controls (519.43,57.33) and (516.67,54.57) .. (516.67,51.17) -- cycle ;
%Shape: Circle [id:dp9453695699254968] 
\draw  [fill={rgb, 255:red, 0; green, 0; blue, 0 }  ,fill opacity=1 ] (516.67,185.17) .. controls (516.67,181.76) and (519.43,179) .. (522.83,179) .. controls (526.24,179) and (529,181.76) .. (529,185.17) .. controls (529,188.57) and (526.24,191.33) .. (522.83,191.33) .. controls (519.43,191.33) and (516.67,188.57) .. (516.67,185.17) -- cycle ;
%Shape: Circle [id:dp03699310930315214] 
\draw  [fill={rgb, 255:red, 0; green, 0; blue, 0 }  ,fill opacity=1 ] (516.67,233.17) .. controls (516.67,229.76) and (519.43,227) .. (522.83,227) .. controls (526.24,227) and (529,229.76) .. (529,233.17) .. controls (529,236.57) and (526.24,239.33) .. (522.83,239.33) .. controls (519.43,239.33) and (516.67,236.57) .. (516.67,233.17) -- cycle ;
%Shape: Circle [id:dp8239231303287081] 
\draw  [fill={rgb, 255:red, 0; green, 0; blue, 0 }  ,fill opacity=1 ] (516.67,279.17) .. controls (516.67,275.76) and (519.43,273) .. (522.83,273) .. controls (526.24,273) and (529,275.76) .. (529,279.17) .. controls (529,282.57) and (526.24,285.33) .. (522.83,285.33) .. controls (519.43,285.33) and (516.67,282.57) .. (516.67,279.17) -- cycle ;
%Shape: Circle [id:dp30711705472565065] 
\draw  [fill={rgb, 255:red, 0; green, 0; blue, 0 }  ,fill opacity=1 ] (516.67,93.17) .. controls (516.67,89.76) and (519.43,87) .. (522.83,87) .. controls (526.24,87) and (529,89.76) .. (529,93.17) .. controls (529,96.57) and (526.24,99.33) .. (522.83,99.33) .. controls (519.43,99.33) and (516.67,96.57) .. (516.67,93.17) -- cycle ;
%Shape: Circle [id:dp7529440357814672] 
\draw  [fill={rgb, 255:red, 0; green, 0; blue, 0 }  ,fill opacity=1 ] (517.67,322.17) .. controls (517.67,318.76) and (520.43,316) .. (523.83,316) .. controls (527.24,316) and (530,318.76) .. (530,322.17) .. controls (530,325.57) and (527.24,328.33) .. (523.83,328.33) .. controls (520.43,328.33) and (517.67,325.57) .. (517.67,322.17) -- cycle ;
%Shape: Circle [id:dp3319626872618018] 
\draw  [fill={rgb, 255:red, 0; green, 0; blue, 0 }  ,fill opacity=1 ] (517.67,373.17) .. controls (517.67,369.76) and (520.43,367) .. (523.83,367) .. controls (527.24,367) and (530,369.76) .. (530,373.17) .. controls (530,376.57) and (527.24,379.33) .. (523.83,379.33) .. controls (520.43,379.33) and (517.67,376.57) .. (517.67,373.17) -- cycle ;
%Shape: Circle [id:dp958524868523183] 
\draw  [fill={rgb, 255:red, 0; green, 0; blue, 0 }  ,fill opacity=1 ] (516.67,140.17) .. controls (516.67,136.76) and (519.43,134) .. (522.83,134) .. controls (526.24,134) and (529,136.76) .. (529,140.17) .. controls (529,143.57) and (526.24,146.33) .. (522.83,146.33) .. controls (519.43,146.33) and (516.67,143.57) .. (516.67,140.17) -- cycle ;
%Shape: Circle [id:dp14245998644089686] 
\draw  [fill={rgb, 255:red, 0; green, 0; blue, 0 }  ,fill opacity=1 ] (133.67,423.17) .. controls (133.67,419.76) and (136.43,417) .. (139.83,417) .. controls (143.24,417) and (146,419.76) .. (146,423.17) .. controls (146,426.57) and (143.24,429.33) .. (139.83,429.33) .. controls (136.43,429.33) and (133.67,426.57) .. (133.67,423.17) -- cycle ;
%Shape: Circle [id:dp9252093595735664] 
\draw  [fill={rgb, 255:red, 0; green, 0; blue, 0 }  ,fill opacity=1 ] (234.67,423.17) .. controls (234.67,419.76) and (237.43,417) .. (240.83,417) .. controls (244.24,417) and (247,419.76) .. (247,423.17) .. controls (247,426.57) and (244.24,429.33) .. (240.83,429.33) .. controls (237.43,429.33) and (234.67,426.57) .. (234.67,423.17) -- cycle ;
%Shape: Circle [id:dp6990733162786534] 
\draw  [fill={rgb, 255:red, 0; green, 0; blue, 0 }  ,fill opacity=1 ] (416.67,423.17) .. controls (416.67,419.76) and (419.43,417) .. (422.83,417) .. controls (426.24,417) and (429,419.76) .. (429,423.17) .. controls (429,426.57) and (426.24,429.33) .. (422.83,429.33) .. controls (419.43,429.33) and (416.67,426.57) .. (416.67,423.17) -- cycle ;
%Shape: Circle [id:dp626267535184458] 
\draw  [fill={rgb, 255:red, 0; green, 0; blue, 0 }  ,fill opacity=1 ] (517.67,423.17) .. controls (517.67,419.76) and (520.43,417) .. (523.83,417) .. controls (527.24,417) and (530,419.76) .. (530,423.17) .. controls (530,426.57) and (527.24,429.33) .. (523.83,429.33) .. controls (520.43,429.33) and (517.67,426.57) .. (517.67,423.17) -- cycle ;
%Straight Lines [id:da2669817268128519] 
\draw [color={rgb, 255:red, 208; green, 2; blue, 27 }  ,draw opacity=1 ][line width=3]    (137.83,94.17) -- (238.83,94.17) ;
%Straight Lines [id:da3016109571013108] 
\draw [color={rgb, 255:red, 208; green, 2; blue, 27 }  ,draw opacity=1 ][line width=3]    (137.83,186.17) -- (238.83,186.17) ;
%Straight Lines [id:da4080800188113405] 
\draw [color={rgb, 255:red, 208; green, 2; blue, 27 }  ,draw opacity=1 ][line width=3]    (137.83,234.17) -- (238.83,234.17) ;
%Straight Lines [id:da7380445761189431] 
\draw [color={rgb, 255:red, 208; green, 2; blue, 27 }  ,draw opacity=1 ][line width=3]    (138.83,374.17) -- (240.83,423.17) ;
%Straight Lines [id:da8032196230217655] 
\draw [color={rgb, 255:red, 208; green, 2; blue, 27 }  ,draw opacity=1 ][line width=3]    (139.83,423.17) -- (239.83,374.17) ;
%Straight Lines [id:da7335805919117255] 
\draw [color={rgb, 255:red, 208; green, 2; blue, 27 }  ,draw opacity=1 ][line width=3]    (138.83,323.17) -- (238.83,280.17) ;
%Straight Lines [id:da7823951007863974] 
\draw [color={rgb, 255:red, 208; green, 2; blue, 27 }  ,draw opacity=1 ][line width=3]    (421.83,93.17) -- (522.83,93.17) ;
%Straight Lines [id:da7464043945640564] 
\draw [color={rgb, 255:red, 208; green, 2; blue, 27 }  ,draw opacity=1 ][line width=3]    (421.83,185.17) -- (522.83,185.17) ;
%Straight Lines [id:da35514367507920674] 
\draw [color={rgb, 255:red, 208; green, 2; blue, 27 }  ,draw opacity=1 ][line width=3]    (421.83,233.17) -- (522.83,233.17) ;
%Straight Lines [id:da47495099697731435] 
\draw [color={rgb, 255:red, 208; green, 2; blue, 27 }  ,draw opacity=1 ][line width=3]    (422.83,322.17) -- (522.83,279.17) ;
%Straight Lines [id:da29321888119988904] 
\draw [color={rgb, 255:red, 208; green, 2; blue, 27 }  ,draw opacity=1 ][line width=3]    (422.83,373.17) -- (523.83,423.17) ;
%Straight Lines [id:da9666837473789387] 
\draw [color={rgb, 255:red, 208; green, 2; blue, 27 }  ,draw opacity=1 ][line width=3]    (422.83,423.17) -- (523.83,373.17) ;
%Straight Lines [id:da030257783266588678] 
\draw [color={rgb, 255:red, 74; green, 144; blue, 226 }  ,draw opacity=1 ][line width=3]    (137.83,94.17) -- (238.83,52.17) ;
%Straight Lines [id:da3253400423741817] 
\draw [color={rgb, 255:red, 74; green, 144; blue, 226 }  ,draw opacity=1 ][line width=3]    (137.83,141.17) -- (238.83,141.17) ;
%Straight Lines [id:da12923026950367789] 
\draw [color={rgb, 255:red, 74; green, 144; blue, 226 }  ,draw opacity=1 ][line width=3]    (137.83,192.33) -- (238.83,192.33) ;
%Straight Lines [id:da5869686062608143] 
\draw [color={rgb, 255:red, 74; green, 144; blue, 226 }  ,draw opacity=1 ][line width=3]    (421.83,140.17) -- (522.83,140.17) ;
%Straight Lines [id:da9515269954173391] 
\draw [color={rgb, 255:red, 74; green, 144; blue, 226 }  ,draw opacity=1 ][line width=3]    (421.83,93.17) -- (522.83,50.17) ;
%Straight Lines [id:da5526967517151274] 
\draw [color={rgb, 255:red, 74; green, 144; blue, 226 }  ,draw opacity=1 ][line width=3]    (421.83,191.33) -- (522.83,191.33) ;
%Straight Lines [id:da8347390786079557] 
\draw [color={rgb, 255:red, 74; green, 144; blue, 226 }  ,draw opacity=1 ][line width=3]    (137.83,240.33) -- (238.83,240.33) ;
%Straight Lines [id:da06674927879855219] 
\draw [color={rgb, 255:red, 74; green, 144; blue, 226 }  ,draw opacity=1 ][line width=3]    (421.83,239.33) -- (522.83,239.33) ;
%Straight Lines [id:da05339666886197114] 
\draw [color={rgb, 255:red, 74; green, 144; blue, 226 }  ,draw opacity=1 ][line width=3]    (422.83,373.17) -- (523.83,373.17) ;
%Straight Lines [id:da034708935942093344] 
\draw [color={rgb, 255:red, 74; green, 144; blue, 226 }  ,draw opacity=1 ][line width=3]    (138.83,374.17) -- (239.83,374.17) ;
%Straight Lines [id:da5513944482939577] 
\draw [color={rgb, 255:red, 74; green, 144; blue, 226 }  ,draw opacity=1 ][line width=3]    (139.83,423.17) -- (240.83,423.17) ;
%Straight Lines [id:da9320653583445858] 
\draw [color={rgb, 255:red, 74; green, 144; blue, 226 }  ,draw opacity=1 ][line width=3]    (137.83,280.17) -- (238.83,280.17) ;
%Straight Lines [id:da9926256849967117] 
\draw [color={rgb, 255:red, 74; green, 144; blue, 226 }  ,draw opacity=1 ][line width=3]    (421.83,279.17) -- (522.83,279.17) ;
%Curve Lines [id:da7639360203387562] 
\draw [line width=2.25]    (296,228) .. controls (330.67,190.33) and (329.67,270.33) .. (363.67,226.33) ;
%Shape: Free Drawing [id:dp698370061825836] 
\draw  [line width=3] [line join = round][line cap = round] (109.67,8.33) .. controls (109.67,16.01) and (108.67,23.66) .. (108.67,31.33) ;
%Shape: Free Drawing [id:dp9293661605336049] 
\draw  [line width=3] [line join = round][line cap = round] (110.67,16.33) .. controls (113,18.67) and (118.89,18.33) .. (120.67,18.33) ;
%Shape: Free Drawing [id:dp8698892783084455] 
\draw  [line width=3] [line join = round][line cap = round] (121.67,6.33) .. controls (121.67,16) and (121.67,25.67) .. (121.67,35.33) ;
%Shape: Free Drawing [id:dp4916483832574817] 
\draw  [line width=3] [line join = round][line cap = round] (140.67,7.33) .. controls (132.2,7.33) and (131.04,20.77) .. (131.67,28.33) .. controls (131.89,31.07) and (135.67,31.03) .. (135.67,33.33) ;
%Shape: Free Drawing [id:dp1769551154052843] 
\draw  [line width=3] [line join = round][line cap = round] (139.67,19.33) .. controls (139.67,20.67) and (144.22,28.56) .. (144.67,28.33) .. controls (146.98,27.18) and (146.95,18.33) .. (148.67,18.33) ;
%Shape: Free Drawing [id:dp5393922131696681] 
\draw  [line width=3] [line join = round][line cap = round] (152.67,8.33) .. controls (157.67,8.33) and (157.88,25.51) .. (157.67,28.33) .. controls (157.49,30.63) and (153.67,33.93) .. (153.67,36.33) ;
%Shape: Free Drawing [id:dp14881594318084268] 
\draw  [line width=3] [line join = round][line cap = round] (215.67,10.33) .. controls (215.67,13.59) and (213.03,32.33) .. (212.67,32.33) ;
%Shape: Free Drawing [id:dp5557807832050381] 
\draw  [line width=3] [line join = round][line cap = round] (213.67,21.33) .. controls (216.67,21.33) and (219.67,21.33) .. (222.67,21.33) ;
%Shape: Free Drawing [id:dp19996061763830064] 
\draw  [line width=3] [line join = round][line cap = round] (220.67,34.33) .. controls (220.67,25.76) and (224.67,15.78) .. (224.67,7.33) ;
%Shape: Free Drawing [id:dp5282379807586284] 
\draw  [line width=3] [line join = round][line cap = round] (239.67,8.33) .. controls (239.67,9.98) and (236.14,12.95) .. (235.67,15.33) .. controls (234.85,19.44) and (233.08,36.33) .. (238.67,36.33) ;
%Shape: Free Drawing [id:dp2982370581364091] 
\draw  [line width=3] [line join = round][line cap = round] (242.67,18.33) .. controls (242.67,22.33) and (242.67,26.33) .. (242.67,30.33) .. controls (242.67,33.42) and (242.78,26.11) .. (243.67,24.33) .. controls (245.57,20.53) and (252.67,25.63) .. (252.67,16.33) ;
%Shape: Free Drawing [id:dp18696280387845365] 
\draw  [line width=3] [line join = round][line cap = round] (264.67,16.33) .. controls (260.68,20.32) and (259.42,36.33) .. (264.67,36.33) ;
%Shape: Free Drawing [id:dp3044575575737888] 
\draw  [line width=3] [line join = round][line cap = round] (267.67,22.33) .. controls (267.67,31.57) and (274.67,37) .. (274.67,23.33) ;
%Shape: Free Drawing [id:dp4154924037648722] 
\draw  [line width=3] [line join = round][line cap = round] (279.67,15.33) .. controls (291.72,15.33) and (278.67,31.78) .. (278.67,34.33) ;
%Shape: Free Drawing [id:dp5632701522002703] 
\draw  [line width=3] [line join = round][line cap = round] (288.67,7.33) .. controls (288.67,10.39) and (291.56,12.62) .. (291.67,14.33) .. controls (292,19.66) and (292,25.01) .. (291.67,30.33) .. controls (291.54,32.3) and (288.67,34.67) .. (288.67,37.33) ;
%Shape: Free Drawing [id:dp326663361635695] 
\draw  [color={rgb, 255:red, 208; green, 2; blue, 27 }  ,draw opacity=1 ][line width=3] [line join = round][line cap = round] (20.67,59.33) .. controls (20.67,65.53) and (25.87,74.98) .. (26.67,81.33) .. controls (27,84) and (28.07,90.54) .. (25.67,89.33) .. controls (21.78,87.39) and (24.31,80.66) .. (24.67,76.33) .. controls (24.96,72.8) and (32.88,61.33) .. (37.67,61.33) ;
%Shape: Free Drawing [id:dp03302358941155936] 
\draw  [color={rgb, 255:red, 208; green, 2; blue, 27 }  ,draw opacity=1 ][line width=3] [line join = round][line cap = round] (41.67,70.33) .. controls (45,70.33) and (48.33,70.33) .. (51.67,70.33) ;
%Shape: Free Drawing [id:dp17204928524313645] 
\draw  [color={rgb, 255:red, 74; green, 144; blue, 226 }  ,draw opacity=1 ][line width=3] [line join = round][line cap = round] (19.67,15.33) .. controls (19.67,19.12) and (24.46,24.93) .. (25.67,27.33) .. controls (27.28,30.56) and (28.17,37.81) .. (27.67,41.33) .. controls (27.51,42.45) and (23.89,42.46) .. (23.67,41.33) .. controls (23.36,39.78) and (22.35,31.91) .. (22.67,30.33) .. controls (23.09,28.22) and (30.84,15.33) .. (32.67,15.33) ;
%Shape: Free Drawing [id:dp9316860870016098] 
\draw  [color={rgb, 255:red, 74; green, 144; blue, 226 }  ,draw opacity=1 ][line width=3] [line join = round][line cap = round] (41.67,17.33) .. controls (46.01,17.33) and (50.32,18.33) .. (54.67,18.33) ;
%Shape: Free Drawing [id:dp39757969450791153] 
\draw  [color={rgb, 255:red, 74; green, 144; blue, 226 }  ,draw opacity=1 ][line width=3] [line join = round][line cap = round] (47.67,12.33) .. controls (47.67,16.33) and (47.67,20.33) .. (47.67,24.33) ;

\end{tikzpicture}

%% file: samplehyp.tex
\tikzset{every picture/.style={line width=0.75pt}} %set default line width to 0.75pt        

\begin{tikzpicture}[x=0.75pt,y=0.75pt,yscale=-1,xscale=1]

%Shape: Free Drawing [id:dp154084136129088] 
\draw  [line width=3] [line join = round][line cap = round] (167.67,103.67) .. controls (170.9,100.44) and (173.89,101.44) .. (176.67,98.67) .. controls (182.78,92.55) and (188.41,83.44) .. (191.67,73.67) .. controls (198.69,52.59) and (189.87,25.87) .. (176.67,12.67) .. controls (172.53,8.53) and (155.42,9.45) .. (151.67,9.67) .. controls (147.09,9.94) and (142.15,17.19) .. (138.67,20.67) .. controls (128.23,31.1) and (124.04,55.16) .. (128.67,73.67) .. controls (133.77,94.08) and (147.77,103.67) .. (170.67,103.67) ;
%Shape: Free Drawing [id:dp9157048273799924] 
\draw  [line width=3] [line join = round][line cap = round] (173.67,100.67) .. controls (170.82,103.51) and (167.05,103.51) .. (165.67,107.67) .. controls (158.09,130.4) and (161.63,150.12) .. (175.67,167.67) .. controls (178.48,171.18) and (183.06,176.83) .. (187.67,178.67) .. controls (205.01,185.61) and (220.44,178.28) .. (231.67,172.67) .. controls (245.83,165.58) and (247.97,138.61) .. (238.67,124.67) .. controls (236.2,120.97) and (235.78,114.78) .. (232.67,111.67) .. controls (219.57,98.57) and (216.17,96.46) .. (194.67,94.67) .. controls (184.52,93.82) and (176.94,92.4) .. (170.67,98.67) ;
%Shape: Free Drawing [id:dp35479876968341684] 
\draw  [line width=3] [line join = round][line cap = round] (281.67,144.67) .. controls (260.78,144.67) and (235.97,143.19) .. (241.67,171.67) .. controls (242.8,177.35) and (243.05,182.21) .. (247.67,185.67) .. controls (251.34,188.42) and (255.33,187.94) .. (259.67,188.67) .. controls (261.87,189.03) and (263.44,191.48) .. (265.67,191.67) .. controls (278.68,192.75) and (288.27,187.06) .. (294.67,180.67) .. controls (295.72,179.61) and (295.26,176.89) .. (295.67,175.67) .. controls (301.41,158.43) and (295.98,145.67) .. (280.67,145.67) ;
%Shape: Free Drawing [id:dp08298169377661502] 
\draw  [line width=3] [line join = round][line cap = round] (279.67,166.67) .. controls (296.29,166.67) and (303.13,153.6) .. (308.67,140.67) .. controls (310.19,137.11) and (313.42,133.29) .. (313.67,129.67) .. controls (314.56,116.28) and (308.2,96.92) .. (295.67,95.67) .. controls (277.38,93.84) and (268.88,109.23) .. (263.67,119.67) .. controls (263.24,120.51) and (260.77,122.78) .. (260.67,124.67) .. controls (259.6,143.94) and (263.34,167.67) .. (282.67,167.67) ;
%Shape: Free Drawing [id:dp6272966357840156] 
\draw  [line width=3] [line join = round][line cap = round] (328.67,157.67) .. controls (323.33,157.67) and (317.99,157.33) .. (312.67,157.67) .. controls (310.01,157.83) and (299.86,166.28) .. (298.67,168.67) .. controls (286.77,192.47) and (307.89,205.56) .. (328.67,203.67) .. controls (341.51,202.5) and (356.31,185.46) .. (354.67,170.67) .. controls (353.64,161.45) and (338.49,156.67) .. (329.67,156.67) ;
%Shape: Free Drawing [id:dp35883174381027516] 
\draw  [line width=3] [line join = round][line cap = round] (354.67,168.67) .. controls (352.31,168.67) and (354,164) .. (353.67,161.67) .. controls (352.68,154.73) and (352.33,147.67) .. (352.67,140.67) .. controls (353.57,121.7) and (371.91,103.02) .. (388.67,99.67) .. controls (398.58,97.68) and (418.25,95.84) .. (423.67,106.67) .. controls (436.86,133.06) and (418.59,160.2) .. (399.67,169.67) .. controls (395.84,171.58) and (389.7,177.37) .. (384.67,177.67) .. controls (379.34,177.98) and (373.99,178) .. (368.67,177.67) .. controls (368,177.62) and (355.61,166.67) .. (353.67,166.67) ;
%Shape: Free Drawing [id:dp8008306665223663] 
\draw  [line width=3] [line join = round][line cap = round] (371.67,145.67) .. controls (368.04,145.67) and (365.27,154.97) .. (368.67,156.67) .. controls (376.28,160.48) and (381.91,145.67) .. (372.67,145.67) ;
%Shape: Free Drawing [id:dp2086970949701673] 
\draw  [line width=3] [line join = round][line cap = round] (386.67,131.67) .. controls (383.04,131.67) and (380.27,140.97) .. (383.67,142.67) .. controls (391.28,146.48) and (396.91,131.67) .. (387.67,131.67) ;
%Shape: Free Drawing [id:dp41759734678531246] 
\draw  [line width=3] [line join = round][line cap = round] (400.67,117.67) .. controls (397.04,117.67) and (394.27,126.97) .. (397.67,128.67) .. controls (405.28,132.48) and (410.91,117.67) .. (401.67,117.67) ;
%Shape: Free Drawing [id:dp07690769990353341] 
\draw  [line width=3] [line join = round][line cap = round] (266.67,162.67) .. controls (263.04,162.67) and (260.27,171.97) .. (263.67,173.67) .. controls (271.28,177.48) and (276.91,162.67) .. (267.67,162.67) ;
%Shape: Free Drawing [id:dp5804377132482524] 
\draw  [line width=3] [line join = round][line cap = round] (280.67,141.67) .. controls (277.04,141.67) and (274.27,150.97) .. (277.67,152.67) .. controls (285.28,156.48) and (290.91,141.67) .. (281.67,141.67) ;
%Shape: Free Drawing [id:dp3614911524948702] 
\draw  [line width=3] [line join = round][line cap = round] (285.67,116.67) .. controls (282.04,116.67) and (279.27,125.97) .. (282.67,127.67) .. controls (290.28,131.48) and (295.91,116.67) .. (286.67,116.67) ;
%Shape: Free Drawing [id:dp4764832580656081] 
\draw  [line width=3] [line join = round][line cap = round] (202.67,130.67) .. controls (199.04,130.67) and (196.27,139.97) .. (199.67,141.67) .. controls (207.28,145.48) and (212.91,130.67) .. (203.67,130.67) ;
%Shape: Free Drawing [id:dp8389858543439427] 
\draw  [line width=3] [line join = round][line cap = round] (189.67,113.67) .. controls (186.04,113.67) and (183.27,122.97) .. (186.67,124.67) .. controls (194.28,128.48) and (199.91,113.67) .. (190.67,113.67) ;
%Shape: Free Drawing [id:dp4735671447675237] 
\draw  [line width=3] [line join = round][line cap = round] (215.67,146.67) .. controls (212.04,146.67) and (209.27,155.97) .. (212.67,157.67) .. controls (220.28,161.48) and (225.91,146.67) .. (216.67,146.67) ;
%Shape: Free Drawing [id:dp9475659411276894] 
\draw  [line width=3] [line join = round][line cap = round] (161.67,72.67) .. controls (158.04,72.67) and (155.27,81.97) .. (158.67,83.67) .. controls (166.28,87.48) and (171.91,72.67) .. (162.67,72.67) ;
%Shape: Free Drawing [id:dp048566659388490474] 
\draw  [line width=3] [line join = round][line cap = round] (156.67,49.67) .. controls (153.04,49.67) and (150.27,58.97) .. (153.67,60.67) .. controls (161.28,64.48) and (166.91,49.67) .. (157.67,49.67) ;
%Shape: Free Drawing [id:dp5231308224865836] 
\draw  [line width=3] [line join = round][line cap = round] (155.67,27.67) .. controls (152.04,27.67) and (149.27,36.97) .. (152.67,38.67) .. controls (160.28,42.48) and (165.91,27.67) .. (156.67,27.67) ;
%Shape: Free Drawing [id:dp19086918029959] 
\draw  [line width=3] [line join = round][line cap = round] (204.67,28.67) .. controls (202,31.34) and (198.26,32.09) .. (197.67,35.67) .. controls (194.92,52.17) and (195.88,54.37) .. (209.67,56.67) .. controls (215.39,57.62) and (221.18,57.41) .. (226.67,54.67) .. controls (238.22,48.89) and (231.56,24.66) .. (220.67,23.67) .. controls (217,23.33) and (205.67,22.94) .. (205.67,27.67) ;
%Shape: Free Drawing [id:dp5405283430661184] 
\draw  [line width=3] [line join = round][line cap = round] (124.67,56.67) .. controls (124.67,44.2) and (105.98,40.56) .. (96.67,43.67) .. controls (90.84,45.61) and (91.41,65.12) .. (91.67,69.67) .. controls (92,75.66) and (104.68,78.33) .. (110.67,77.67) .. controls (119.61,76.67) and (124.67,67.33) .. (124.67,56.67) -- cycle ;
%Shape: Free Drawing [id:dp7638908426909672] 
\draw  [line width=3] [line join = round][line cap = round] (98.67,53.67) .. controls (98.67,57.35) and (103,56.33) .. (101.67,53.67) .. controls (101.46,53.26) and (98.67,51.2) .. (98.67,53.67) -- cycle ;
%Shape: Free Drawing [id:dp4793374641118362] 
\draw  [line width=3] [line join = round][line cap = round] (107.67,63.67) .. controls (107.67,67.35) and (112,66.33) .. (110.67,63.67) .. controls (110.46,63.26) and (107.67,61.2) .. (107.67,63.67) -- cycle ;
%Shape: Free Drawing [id:dp10590413168634294] 
\draw  [line width=3] [line join = round][line cap = round] (207.67,42.67) .. controls (207.67,46.35) and (212,45.33) .. (210.67,42.67) .. controls (210.46,42.26) and (207.67,40.2) .. (207.67,42.67) -- cycle ;
%Shape: Free Drawing [id:dp8171494409974087] 
\draw  [line width=3] [line join = round][line cap = round] (217.67,35.67) .. controls (217.67,39.35) and (222,38.33) .. (220.67,35.67) .. controls (220.46,35.26) and (217.67,33.2) .. (217.67,35.67) -- cycle ;
%Shape: Free Drawing [id:dp730574828596749] 
\draw  [line width=3] [line join = round][line cap = round] (429.67,59.67) .. controls (417.29,59.67) and (406.43,99.76) .. (429.67,102.67) .. controls (448.6,105.03) and (458.94,85.49) .. (453.67,69.67) .. controls (450.83,61.16) and (437.88,61.67) .. (429.67,61.67) ;
%Shape: Free Drawing [id:dp17434897522654147] 
\draw  [line width=3] [line join = round][line cap = round] (438.67,71.67) .. controls (438.67,75.35) and (443,74.33) .. (441.67,71.67) .. controls (441.46,71.26) and (438.67,69.2) .. (438.67,71.67) -- cycle ;
%Shape: Free Drawing [id:dp4724361328792315] 
\draw  [line width=3] [line join = round][line cap = round] (423.67,79.67) .. controls (423.67,83.35) and (428,82.33) .. (426.67,79.67) .. controls (426.46,79.26) and (423.67,77.2) .. (423.67,79.67) -- cycle ;
%Shape: Free Drawing [id:dp4523039705205816] 
\draw  [line width=3] [line join = round][line cap = round] (433.67,87.67) .. controls (433.67,91.35) and (438,90.33) .. (436.67,87.67) .. controls (436.46,87.26) and (433.67,85.2) .. (433.67,87.67) -- cycle ;
%Shape: Free Drawing [id:dp3489870087928669] 
\draw  [line width=3] [line join = round][line cap = round] (222.67,159.67) .. controls (222.67,163.35) and (227,162.33) .. (225.67,159.67) .. controls (225.46,159.26) and (222.67,157.2) .. (222.67,159.67) -- cycle ;
%Shape: Free Drawing [id:dp8576871477294348] 
\draw  [line width=3] [line join = round][line cap = round] (377.67,115.67) .. controls (377.67,119.35) and (382,118.33) .. (380.67,115.67) .. controls (380.46,115.26) and (377.67,113.2) .. (377.67,115.67) -- cycle ;
%Shape: Free Drawing [id:dp7463340457405655] 
\draw  [line width=3] [line join = round][line cap = round] (407.67,146.67) .. controls (407.67,150.35) and (412,149.33) .. (410.67,146.67) .. controls (410.46,146.26) and (407.67,144.2) .. (407.67,146.67) -- cycle ;
%Shape: Free Drawing [id:dp5302111042316368] 
\draw  [line width=3] [line join = round][line cap = round] (153.67,15.67) .. controls (153.67,19.35) and (158,18.33) .. (156.67,15.67) .. controls (156.46,15.26) and (153.67,13.2) .. (153.67,15.67) -- cycle ;
%Shape: Free Drawing [id:dp7861152848948819] 
\draw  [line width=3] [line join = round][line cap = round] (365.67,75.67) .. controls (365.67,82.67) and (364.67,89.66) .. (364.67,96.67) ;
%Shape: Free Drawing [id:dp4374155684729413] 
\draw  [line width=3] [line join = round][line cap = round] (361.67,76.67) .. controls (365.99,76.67) and (366.87,73.7) .. (372.67,74.67) .. controls (375.14,75.08) and (376.91,82.67) .. (367.67,82.67) ;
%Shape: Free Drawing [id:dp4282129051110005] 
\draw  [line width=3] [line join = round][line cap = round] (428.67,159.67) .. controls (428.67,166.67) and (427.67,173.66) .. (427.67,180.67) ;
%Shape: Free Drawing [id:dp9010077846506572] 
\draw  [line width=3] [line join = round][line cap = round] (424.67,160.67) .. controls (428.99,160.67) and (429.87,157.7) .. (435.67,158.67) .. controls (438.14,159.08) and (439.91,166.67) .. (430.67,166.67) ;
%Shape: Free Drawing [id:dp5893638318697865] 
\draw  [line width=3] [line join = round][line cap = round] (397.67,52.67) .. controls (397.67,59.67) and (396.67,66.66) .. (396.67,73.67) ;
%Shape: Free Drawing [id:dp792083493004141] 
\draw  [line width=3] [line join = round][line cap = round] (393.67,53.67) .. controls (397.99,53.67) and (398.87,50.7) .. (404.67,51.67) .. controls (407.14,52.08) and (408.91,59.67) .. (399.67,59.67) ;
%Shape: Free Drawing [id:dp6413672896654575] 
\draw  [line width=3] [line join = round][line cap = round] (450.67,102.67) .. controls (450.67,109.67) and (449.67,116.66) .. (449.67,123.67) ;
%Shape: Free Drawing [id:dp10308300750157728] 
\draw  [line width=3] [line join = round][line cap = round] (446.67,103.67) .. controls (450.99,103.67) and (451.87,100.7) .. (457.67,101.67) .. controls (460.14,102.08) and (461.91,109.67) .. (452.67,109.67) ;
%Shape: Free Drawing [id:dp08289879281518953] 
\draw  [line width=3] [line join = round][line cap = round] (77.67,24.67) .. controls (77.67,31.67) and (76.67,38.66) .. (76.67,45.67) ;
%Shape: Free Drawing [id:dp10858731212607853] 
\draw  [line width=3] [line join = round][line cap = round] (73.67,25.67) .. controls (77.99,25.67) and (78.87,22.7) .. (84.67,23.67) .. controls (87.14,24.08) and (88.91,31.67) .. (79.67,31.67) ;
%Shape: Free Drawing [id:dp08760687272016587] 
\draw  [line width=3] [line join = round][line cap = round] (195.67,190.67) .. controls (203.77,190.67) and (196.18,201.71) .. (197.67,207.67) .. controls (198,209) and (201.05,207.9) .. (201.67,206.67) .. controls (203.82,202.35) and (204.3,198.15) .. (205.67,192.67) .. controls (206,191.33) and (206.29,195.35) .. (206.67,196.67) .. controls (207.51,199.61) and (208.35,206.02) .. (209.67,208.67) .. controls (210.54,210.41) and (212.2,205.55) .. (212.67,203.67) .. controls (213.48,200.42) and (213.67,197.02) .. (213.67,193.67) ;
%Shape: Free Drawing [id:dp35406026346059216] 
\draw  [line width=3] [line join = round][line cap = round] (103.67,4.67) .. controls (111.77,4.67) and (104.18,15.71) .. (105.67,21.67) .. controls (106,23) and (109.05,21.9) .. (109.67,20.67) .. controls (111.82,16.35) and (112.3,12.15) .. (113.67,6.67) .. controls (114,5.33) and (114.29,9.35) .. (114.67,10.67) .. controls (115.51,13.61) and (116.35,20.02) .. (117.67,22.67) .. controls (118.54,24.41) and (120.2,19.55) .. (120.67,17.67) .. controls (121.48,14.42) and (121.67,11.02) .. (121.67,7.67) ;
%Shape: Free Drawing [id:dp583824346618351] 
\draw  [line width=3] [line join = round][line cap = round] (446.67,35.67) .. controls (454.77,35.67) and (447.18,46.71) .. (448.67,52.67) .. controls (449,54) and (452.05,52.9) .. (452.67,51.67) .. controls (454.82,47.35) and (455.3,43.15) .. (456.67,37.67) .. controls (457,36.33) and (457.29,40.35) .. (457.67,41.67) .. controls (458.51,44.61) and (459.35,51.02) .. (460.67,53.67) .. controls (461.54,55.41) and (463.2,50.55) .. (463.67,48.67) .. controls (464.48,45.42) and (464.67,42.02) .. (464.67,38.67) ;
%Shape: Free Drawing [id:dp24732105485782763] 
\draw  [line width=3] [line join = round][line cap = round] (82.67,44.67) .. controls (82.67,46.33) and (82.67,48) .. (82.67,49.67) ;
%Shape: Free Drawing [id:dp7100834575313187] 
\draw  [line width=3] [line join = round][line cap = round] (79.67,83.67) .. controls (79.67,89.67) and (79.67,95.67) .. (79.67,101.67) ;
%Shape: Free Drawing [id:dp7861206571146936] 
\draw  [line width=3] [line join = round][line cap = round] (76.67,83.67) .. controls (87.89,83.67) and (91.1,89.67) .. (81.67,89.67) ;
%Shape: Free Drawing [id:dp632328983369215] 
\draw  [line width=3] [line join = round][line cap = round] (85.67,96.67) .. controls (86.36,96.67) and (88.52,93.38) .. (89.67,95.67) .. controls (90.57,97.48) and (89.57,99.85) .. (88.67,101.67) .. controls (88.33,102.33) and (86.92,103.67) .. (87.67,103.67) .. controls (89.33,103.67) and (91,103.67) .. (92.67,103.67) ;
%Shape: Free Drawing [id:dp766601140209197] 
\draw  [line width=3] [line join = round][line cap = round] (74.67,77.67) .. controls (78.33,77.67) and (82,77.67) .. (85.67,77.67) ;
%Shape: Free Drawing [id:dp7144633267055633] 
\draw  [line width=3] [line join = round][line cap = round] (239.67,21.67) .. controls (239.67,26.68) and (240.67,31.66) .. (240.67,36.67) ;
%Shape: Free Drawing [id:dp9288788147766728] 
\draw  [line width=3] [line join = round][line cap = round] (236.67,21.67) .. controls (239.15,21.67) and (245.09,19.78) .. (245.67,22.67) .. controls (247.19,30.29) and (240.67,28.14) .. (240.67,27.67) ;
%Shape: Free Drawing [id:dp25148651457422466] 
\draw  [line width=3] [line join = round][line cap = round] (246.67,32.67) .. controls (246.67,34.33) and (246.67,36) .. (246.67,37.67) ;
%Shape: Free Drawing [id:dp44839631932709334] 
\draw  [line width=3] [line join = round][line cap = round] (233.67,15.67) .. controls (238,15.67) and (242.33,15.67) .. (246.67,15.67) ;
%Shape: Free Drawing [id:dp8289389117404196] 
\draw  [line width=3] [line join = round][line cap = round] (239.67,53.67) .. controls (239.67,58) and (239.67,62.33) .. (239.67,66.67) ;
%Shape: Free Drawing [id:dp5551691899952512] 
\draw  [line width=3] [line join = round][line cap = round] (238.67,52.67) .. controls (244.07,52.67) and (249.81,59.67) .. (240.67,59.67) ;
%Shape: Free Drawing [id:dp0437004096525524] 
\draw  [line width=3] [line join = round][line cap = round] (243.67,63.67) .. controls (244.28,63.67) and (247.38,64.1) .. (247.67,64.67) .. controls (249.7,68.72) and (239.96,70.67) .. (247.67,70.67) ;
%Shape: Free Drawing [id:dp7569060616126647] 
\draw  [line width=3] [line join = round][line cap = round] (124.67,20.67) .. controls (124.67,22.67) and (124.67,24.67) .. (124.67,26.67) ;
%Shape: Free Drawing [id:dp06423643011104452] 
\draw  [line width=3] [line join = round][line cap = round] (143.67,45.67) .. controls (145.9,45.67) and (149.26,43.77) .. (151.67,43.67) .. controls (159.33,43.35) and (167.01,43.99) .. (174.67,43.67) .. controls (176.04,43.61) and (177.36,42.23) .. (178.67,42.67) .. controls (179,42.78) and (177.68,44.66) .. (177.67,44.67) .. controls (177.58,44.75) and (173.67,49.67) .. (173.67,49.67) .. controls (173.67,49.67) and (174.57,46.76) .. (174.67,46.67) .. controls (176.06,45.28) and (177.89,44.45) .. (179.67,42.67) .. controls (180.51,41.83) and (176.01,39.01) .. (175.67,38.67) .. controls (174.91,37.91) and (175.23,36.67) .. (173.67,36.67) ;
%Shape: Free Drawing [id:dp5838446642380375] 
\draw  [line width=3] [line join = round][line cap = round] (143.67,38.67) .. controls (143.67,39.24) and (139.21,47.21) .. (139.67,47.67) .. controls (140.64,48.64) and (143.67,49.8) .. (143.67,52.67) ;
%Shape: Free Drawing [id:dp290796159545051] 
\draw  [line width=3] [line join = round][line cap = round] (200.67,152.67) .. controls (202.91,152.67) and (215.88,139.56) .. (219.67,137.67) .. controls (221.72,136.64) and (222.21,137.12) .. (223.67,135.67) .. controls (223.67,135.67) and (225.62,134.37) .. (225.67,134.67) .. controls (226,136.64) and (225.67,138.67) .. (225.67,140.67) .. controls (225.67,141.41) and (226.56,139.4) .. (226.67,138.67) .. controls (227,136.36) and (228.02,133.57) .. (226.67,131.67) .. controls (225.49,130.02) and (222.69,130.67) .. (220.67,130.67) ;
%Shape: Free Drawing [id:dp005892952896811132] 
\draw  [line width=3] [line join = round][line cap = round] (198.67,147.67) .. controls (198.67,154.16) and (193.21,158.67) .. (201.67,158.67) ;
%Shape: Free Drawing [id:dp1720520976570279] 
\draw  [line width=3] [line join = round][line cap = round] (368.67,137.67) .. controls (368.67,135.67) and (369,133.64) .. (368.67,131.67) .. controls (368.49,130.63) and (366.72,129.14) .. (367.67,128.67) .. controls (370.07,127.46) and (372.98,129.67) .. (375.67,129.67) ;
%Shape: Free Drawing [id:dp9177493138118612] 
\draw  [line width=3] [line join = round][line cap = round] (368.67,129.67) .. controls (374.03,135.03) and (381.67,145.68) .. (384.67,151.67) .. controls (386.16,154.66) and (391.66,158.66) .. (393.67,160.67) .. controls (394.19,161.19) and (395.38,162.46) .. (394.67,162.67) .. controls (392.42,163.31) and (390,162.67) .. (387.67,162.67) .. controls (386.67,162.67) and (389.72,162.35) .. (390.67,162.67) .. controls (392.59,163.31) and (394.67,164) .. (396.67,163.67) .. controls (396.9,163.63) and (394.71,158.9) .. (394.67,158.67) .. controls (394.2,156.36) and (393.67,154.02) .. (393.67,151.67) ;
%Shape: Free Drawing [id:dp9639157187114484] 
\draw  [line width=3] [line join = round][line cap = round] (368.67,89.67) .. controls (370.5,89.67) and (373.59,87.51) .. (374.67,89.67) .. controls (375.81,91.95) and (371.67,93.67) .. (371.67,93.67) .. controls (371.67,93.67) and (376.68,98.67) .. (370.67,98.67) ;
%Shape: Free Drawing [id:dp34435988346840996] 
\draw  [line width=3] [line join = round][line cap = round] (434.67,179) .. controls (434.67,177.68) and (438.04,172.72) .. (440.67,176) .. controls (445.16,181.62) and (438.67,180.02) .. (438.67,183) .. controls (438.67,184.2) and (441.29,183.86) .. (441.67,185) .. controls (443.28,189.85) and (435.67,190.6) .. (435.67,189) ;
%Shape: Free Drawing [id:dp33630725971149766] 
\draw  [line width=3] [line join = round][line cap = round] (426.67,152) .. controls (430,152) and (433.33,152) .. (436.67,152) ;
%Shape: Free Drawing [id:dp2473078138302086] 
\draw  [line width=3] [line join = round][line cap = round] (458.67,121) .. controls (458.67,124.25) and (456.58,126.46) .. (459.67,128) .. controls (460.33,128.33) and (461.14,127.53) .. (461.67,127) .. controls (461.88,126.78) and (463.57,126.19) .. (463.67,126) .. controls (463.96,125.4) and (463.67,123.33) .. (463.67,124) .. controls (463.67,128.01) and (462.67,131.99) .. (462.67,136) ;
%Shape: Free Drawing [id:dp447674584142862] 
\draw  [line width=3] [line join = round][line cap = round] (401.67,69) .. controls (401.67,70.68) and (400,80.39) .. (407.67,74) .. controls (408.95,72.93) and (407.67,67.33) .. (407.67,69) .. controls (407.67,72.67) and (407.67,76.33) .. (407.67,80) ;
%Shape: Free Drawing [id:dp3078997468136341] 
\draw  [line width=3] [line join = round][line cap = round] (394.67,45) .. controls (399.35,45) and (403.99,46) .. (408.67,46) ;
%Shape: Free Drawing [id:dp09094236987476789] 
\draw  [line width=3] [line join = round][line cap = round] (216.67,207) .. controls (216.67,206.36) and (219.37,202.71) .. (220.67,204) .. controls (226.51,209.85) and (213.59,214) .. (223.67,214) ;
%Shape: Free Drawing [id:dp24019582062888056] 
\draw  [line width=3] [line join = round][line cap = round] (438.67,54) .. controls (443.59,54) and (451.61,60.25) .. (454.67,62) .. controls (459.11,64.54) and (460.31,65.32) .. (463.67,67) .. controls (464.74,67.54) and (467.87,69) .. (466.67,69) .. controls (464.64,69) and (462.69,70) .. (460.67,70) .. controls (458.64,70) and (465.66,72.76) .. (466.67,71) .. controls (467.82,68.97) and (466.67,66.33) .. (466.67,64) ;
%Shape: Free Drawing [id:dp841231811120101] 
\draw  [line width=3] [line join = round][line cap = round] (436.67,58) .. controls (436.67,57.46) and (431.28,45.81) .. (435.67,48) .. controls (436.61,48.47) and (437.61,49) .. (438.67,49) ;
%Shape: Free Drawing [id:dp02538918374732846] 
\draw  [line width=3] [line join = round][line cap = round] (335.67,171) .. controls (335.81,172.18) and (339.39,184.91) .. (338.67,186) .. controls (337.43,187.86) and (331.67,181) .. (332.67,183) .. controls (334.98,187.63) and (344.67,191.14) .. (344.67,185) ;
%Shape: Free Drawing [id:dp9246270163214151] 
\draw  [line width=3] [line join = round][line cap = round] (329.67,172) .. controls (332.46,172) and (331.97,166.69) .. (333.67,165) .. controls (334.97,163.7) and (339.75,169) .. (341.67,169) ;
%Shape: Free Drawing [id:dp10434941275470688] 
\draw  [line width=3] [line join = round][line cap = round] (354.67,165.33) .. controls (349.35,165.33) and (359.86,171.71) .. (356.67,165.33) .. controls (355.56,163.11) and (350.31,164.26) .. (351.67,168.33) .. controls (352.44,170.66) and (359.61,171.22) .. (356.67,165.33) .. controls (354.16,160.32) and (347.23,169.33) .. (352.67,169.33) ;
%Shape: Free Drawing [id:dp2891039838502667] 
\draw  [line width=3] [line join = round][line cap = round] (420.67,99.33) .. controls (420.67,104.77) and (426.26,100.93) .. (423.67,98.33) .. controls (419.42,94.09) and (417.9,102.87) .. (418.67,103.33) .. controls (420.52,104.45) and (427.25,104.11) .. (423.67,99.33) .. controls (422.02,97.14) and (418.01,100.68) .. (420.67,103.33) ;
%Shape: Free Drawing [id:dp3479057008583146] 
\draw  [line width=3] [line join = round][line cap = round] (296.67,175.33) .. controls (296.67,177.5) and (296,179.67) .. (300.67,177.33) .. controls (302.07,176.63) and (297.38,171.9) .. (295.67,175.33) .. controls (293.7,179.26) and (302.18,180.37) .. (299.67,175.33) .. controls (298.31,172.61) and (293.48,176.15) .. (295.67,178.33) .. controls (298.27,180.94) and (299.67,176.28) .. (299.67,175.33) ;
%Shape: Free Drawing [id:dp4193866538920821] 
\draw  [line width=3] [line join = round][line cap = round] (278.67,165.33) .. controls (276.66,165.33) and (274.95,172.01) .. (281.67,170.33) .. controls (285.21,169.45) and (279.62,163.43) .. (278.67,165.33) .. controls (277.02,168.63) and (280.67,171.43) .. (280.67,169.33) ;
%Shape: Free Drawing [id:dp9616840062281544] 
\draw  [line width=3] [line join = round][line cap = round] (241.67,164.33) .. controls (244.63,161.37) and (247.26,148.74) .. (238.67,157.33) .. controls (237.49,158.51) and (238.05,160.79) .. (238.67,162.33) .. controls (240.4,166.66) and (247.39,158.33) .. (242.67,158.33) ;
%Shape: Free Drawing [id:dp5246481934250727] 
\draw  [line width=3] [line join = round][line cap = round] (174.67,99.33) .. controls (174.67,93.89) and (163.56,103.97) .. (170.67,106.33) .. controls (176.24,108.19) and (175.35,97.68) .. (172.67,96.33) .. controls (167.42,93.71) and (168.53,108.4) .. (174.67,105.33) .. controls (179.41,102.96) and (171.67,95.01) .. (171.67,100.33) ;
%Shape: Free Drawing [id:dp058313520436557575] 
\draw  [line width=3] [line join = round][line cap = round] (192.67,46.33) .. controls (192.67,58.17) and (199.38,45.19) .. (195.67,43.33) .. controls (192.64,41.82) and (192.67,48) .. (192.67,49.33) ;
%Shape: Free Drawing [id:dp4488888502192355] 
\draw  [line width=3] [line join = round][line cap = round] (124.67,56.33) .. controls (121.16,56.33) and (120.11,69.04) .. (128.67,63.33) .. controls (130.61,62.04) and (123.67,49.93) .. (123.67,61.33) ;
%Shape: Free Drawing [id:dp846895710811987] 
\draw  [line width=3] [line join = round][line cap = round] (254.67,158.33) .. controls (254.67,161.79) and (252.19,168.77) .. (250.67,173.33) .. controls (250.45,173.98) and (247.7,166.3) .. (246.67,167.33) .. controls (245.92,168.08) and (247.33,169.33) .. (247.67,170.33) .. controls (249.48,175.77) and (251.03,173.33) .. (255.67,173.33) ;
%Shape: Free Drawing [id:dp5598571479284624] 
\draw  [line width=3] [line join = round][line cap = round] (258.67,162.33) .. controls (258.67,161.77) and (258.02,160.04) .. (257.67,159.33) .. controls (257,158) and (255.67,155.33) .. (255.67,155.33) .. controls (255.67,155.33) and (251.36,159.33) .. (250.67,159.33) ;
%Shape: Free Drawing [id:dp7339760579979543] 
\draw  [line width=3] [line join = round][line cap = round] (298.67,139.33) .. controls (290.93,139.33) and (272.15,129.33) .. (268.67,129.33) ;
%Shape: Free Drawing [id:dp9691106613882157] 
\draw  [line width=3] [line join = round][line cap = round] (270.67,136.33) .. controls (270.67,130.91) and (263.7,129.33) .. (271.67,127.33) .. controls (271.99,127.25) and (273.69,125.33) .. (274.67,125.33) ;
%Shape: Free Drawing [id:dp898714298255041] 
\draw  [line width=3] [line join = round][line cap = round] (296.67,133.33) .. controls (296.67,135.35) and (300.11,137.23) .. (300.67,138.33) .. controls (301.48,139.97) and (298.94,141.79) .. (298.67,142.33) .. controls (298.33,143) and (298.41,144.33) .. (297.67,144.33) ;
%Shape: Free Drawing [id:dp7817903922631395] 
\draw  [line width=3] [line join = round][line cap = round] (315.67,171.33) .. controls (315.67,166.75) and (313.02,175.98) .. (316.67,172.33) .. controls (317.33,171.67) and (315.61,170.33) .. (314.67,170.33) ;
%Shape: Free Drawing [id:dp43696926330876396] 
\draw  [line width=3] [line join = round][line cap = round] (317.67,188.33) .. controls (317.67,189.28) and (319,191) .. (319.67,190.33) .. controls (320.85,189.15) and (318.67,186.78) .. (318.67,189.33) ;
%Shape: Free Drawing [id:dp6019278480277002] 
\draw  [line width=3] [line join = round][line cap = round] (313.67,210.33) .. controls (313.67,217.01) and (312.67,223.66) .. (312.67,230.33) ;
%Shape: Free Drawing [id:dp020899917206700924] 
\draw  [line width=3] [line join = round][line cap = round] (311.67,209.33) .. controls (319.42,209.33) and (326.8,220.33) .. (315.67,220.33) ;
%Shape: Free Drawing [id:dp19344427614298743] 
\draw  [line width=3] [line join = round][line cap = round] (321.67,226.33) .. controls (321.67,227.67) and (320.87,229.27) .. (321.67,230.33) .. controls (322.36,231.26) and (326.98,230.96) .. (327.67,232.33) .. controls (329.45,235.91) and (321.67,239.19) .. (321.67,235.33) ;
%Shape: Free Drawing [id:dp584204177508646] 
\draw  [line width=3] [line join = round][line cap = round] (321.67,225.33) .. controls (324,225.33) and (326.33,225.33) .. (328.67,225.33) ;
%Shape: Free Drawing [id:dp6666172738958803] 
\draw  [line width=3] [line join = round][line cap = round] (320.67,135.33) .. controls (320.67,140.33) and (320.67,145.33) .. (320.67,150.33) ;
%Shape: Free Drawing [id:dp04253846934546346] 
\draw  [line width=3] [line join = round][line cap = round] (318.67,133.33) .. controls (330.72,133.33) and (330.13,141.33) .. (321.67,141.33) ;
%Shape: Free Drawing [id:dp4250306206542781] 
\draw  [line width=3] [line join = round][line cap = round] (328.67,146.33) .. controls (328.67,151.83) and (337.12,146.42) .. (333.67,153.33) .. controls (332.91,154.85) and (330.37,154.33) .. (328.67,154.33) ;
%Shape: Free Drawing [id:dp9911491289103436] 
\draw  [line width=3] [line join = round][line cap = round] (328.67,144.33) .. controls (331.02,144.33) and (333.31,143.33) .. (335.67,143.33) ;
%Shape: Free Drawing [id:dp22265419080764381] 
\draw  [line width=3] [line join = round][line cap = round] (318.67,127.33) .. controls (323.51,127.33) and (327.4,130.33) .. (331.67,130.33) ;
%Shape: Free Drawing [id:dp8282314556698435] 
\draw  [line width=3] [line join = round][line cap = round] (467.67,52) .. controls (468.44,51.23) and (470.18,49.02) .. (471.67,51) .. controls (472.49,52.1) and (471.64,54.03) .. (470.67,55) .. controls (470.33,55.33) and (469.2,56) .. (469.67,56) .. controls (473.69,56) and (473.69,61) .. (469.67,61) ;

\end{tikzpicture}

%% file: bipartmap.tex
\tikzset{every picture/.style={line width=0.75pt}} %set default line width to 0.75pt        

\begin{tikzpicture}[x=0.75pt,y=0.75pt,yscale=-1,xscale=1]
%uncomment if require: \path (0,467); %set diagram left start at 0, and has height of 467

%Shape: Circle [id:dp39138686001250333] 
\draw  [fill={rgb, 255:red, 0; green, 0; blue, 0 }  ,fill opacity=1 ] (117.67,192.17) .. controls (117.67,188.76) and (120.43,186) .. (123.83,186) .. controls (127.24,186) and (130,188.76) .. (130,192.17) .. controls (130,195.57) and (127.24,198.33) .. (123.83,198.33) .. controls (120.43,198.33) and (117.67,195.57) .. (117.67,192.17) -- cycle ;
%Shape: Circle [id:dp7064539801509641] 
\draw  [fill={rgb, 255:red, 0; green, 0; blue, 0 }  ,fill opacity=1 ] (117.67,240.17) .. controls (117.67,236.76) and (120.43,234) .. (123.83,234) .. controls (127.24,234) and (130,236.76) .. (130,240.17) .. controls (130,243.57) and (127.24,246.33) .. (123.83,246.33) .. controls (120.43,246.33) and (117.67,243.57) .. (117.67,240.17) -- cycle ;
%Shape: Circle [id:dp47072156165560786] 
\draw  [fill={rgb, 255:red, 0; green, 0; blue, 0 }  ,fill opacity=1 ] (117.67,286.17) .. controls (117.67,282.76) and (120.43,280) .. (123.83,280) .. controls (127.24,280) and (130,282.76) .. (130,286.17) .. controls (130,289.57) and (127.24,292.33) .. (123.83,292.33) .. controls (120.43,292.33) and (117.67,289.57) .. (117.67,286.17) -- cycle ;
%Shape: Circle [id:dp11870692028773355] 
\draw  [fill={rgb, 255:red, 0; green, 0; blue, 0 }  ,fill opacity=1 ] (117.67,100.17) .. controls (117.67,96.76) and (120.43,94) .. (123.83,94) .. controls (127.24,94) and (130,96.76) .. (130,100.17) .. controls (130,103.57) and (127.24,106.33) .. (123.83,106.33) .. controls (120.43,106.33) and (117.67,103.57) .. (117.67,100.17) -- cycle ;
%Shape: Circle [id:dp770194487547789] 
\draw  [fill={rgb, 255:red, 0; green, 0; blue, 0 }  ,fill opacity=1 ] (118.67,329.17) .. controls (118.67,325.76) and (121.43,323) .. (124.83,323) .. controls (128.24,323) and (131,325.76) .. (131,329.17) .. controls (131,332.57) and (128.24,335.33) .. (124.83,335.33) .. controls (121.43,335.33) and (118.67,332.57) .. (118.67,329.17) -- cycle ;
%Shape: Circle [id:dp29235381708572794] 
\draw  [fill={rgb, 255:red, 0; green, 0; blue, 0 }  ,fill opacity=1 ] (118.67,380.17) .. controls (118.67,376.76) and (121.43,374) .. (124.83,374) .. controls (128.24,374) and (131,376.76) .. (131,380.17) .. controls (131,383.57) and (128.24,386.33) .. (124.83,386.33) .. controls (121.43,386.33) and (118.67,383.57) .. (118.67,380.17) -- cycle ;
%Shape: Circle [id:dp22966095303045275] 
\draw  [fill={rgb, 255:red, 0; green, 0; blue, 0 }  ,fill opacity=1 ] (117.67,147.17) .. controls (117.67,143.76) and (120.43,141) .. (123.83,141) .. controls (127.24,141) and (130,143.76) .. (130,147.17) .. controls (130,150.57) and (127.24,153.33) .. (123.83,153.33) .. controls (120.43,153.33) and (117.67,150.57) .. (117.67,147.17) -- cycle ;
%Shape: Circle [id:dp803367703085351] 
\draw  [fill={rgb, 255:red, 0; green, 0; blue, 0 }  ,fill opacity=1 ] (218.67,58.17) .. controls (218.67,54.76) and (221.43,52) .. (224.83,52) .. controls (228.24,52) and (231,54.76) .. (231,58.17) .. controls (231,61.57) and (228.24,64.33) .. (224.83,64.33) .. controls (221.43,64.33) and (218.67,61.57) .. (218.67,58.17) -- cycle ;
%Shape: Circle [id:dp6390290655938118] 
\draw  [fill={rgb, 255:red, 0; green, 0; blue, 0 }  ,fill opacity=1 ] (218.67,192.17) .. controls (218.67,188.76) and (221.43,186) .. (224.83,186) .. controls (228.24,186) and (231,188.76) .. (231,192.17) .. controls (231,195.57) and (228.24,198.33) .. (224.83,198.33) .. controls (221.43,198.33) and (218.67,195.57) .. (218.67,192.17) -- cycle ;
%Shape: Circle [id:dp20037953436767253] 
\draw  [fill={rgb, 255:red, 0; green, 0; blue, 0 }  ,fill opacity=1 ] (218.67,240.17) .. controls (218.67,236.76) and (221.43,234) .. (224.83,234) .. controls (228.24,234) and (231,236.76) .. (231,240.17) .. controls (231,243.57) and (228.24,246.33) .. (224.83,246.33) .. controls (221.43,246.33) and (218.67,243.57) .. (218.67,240.17) -- cycle ;
%Shape: Circle [id:dp46992098195107435] 
\draw  [fill={rgb, 255:red, 0; green, 0; blue, 0 }  ,fill opacity=1 ] (218.67,286.17) .. controls (218.67,282.76) and (221.43,280) .. (224.83,280) .. controls (228.24,280) and (231,282.76) .. (231,286.17) .. controls (231,289.57) and (228.24,292.33) .. (224.83,292.33) .. controls (221.43,292.33) and (218.67,289.57) .. (218.67,286.17) -- cycle ;
%Shape: Circle [id:dp628977050081013] 
\draw  [fill={rgb, 255:red, 0; green, 0; blue, 0 }  ,fill opacity=1 ] (218.67,100.17) .. controls (218.67,96.76) and (221.43,94) .. (224.83,94) .. controls (228.24,94) and (231,96.76) .. (231,100.17) .. controls (231,103.57) and (228.24,106.33) .. (224.83,106.33) .. controls (221.43,106.33) and (218.67,103.57) .. (218.67,100.17) -- cycle ;
%Shape: Circle [id:dp7899915070530644] 
\draw  [fill={rgb, 255:red, 0; green, 0; blue, 0 }  ,fill opacity=1 ] (219.67,329.17) .. controls (219.67,325.76) and (222.43,323) .. (225.83,323) .. controls (229.24,323) and (232,325.76) .. (232,329.17) .. controls (232,332.57) and (229.24,335.33) .. (225.83,335.33) .. controls (222.43,335.33) and (219.67,332.57) .. (219.67,329.17) -- cycle ;
%Shape: Circle [id:dp7853485455503675] 
\draw  [fill={rgb, 255:red, 0; green, 0; blue, 0 }  ,fill opacity=1 ] (219.67,380.17) .. controls (219.67,376.76) and (222.43,374) .. (225.83,374) .. controls (229.24,374) and (232,376.76) .. (232,380.17) .. controls (232,383.57) and (229.24,386.33) .. (225.83,386.33) .. controls (222.43,386.33) and (219.67,383.57) .. (219.67,380.17) -- cycle ;
%Shape: Circle [id:dp31972409861498785] 
\draw  [fill={rgb, 255:red, 0; green, 0; blue, 0 }  ,fill opacity=1 ] (218.67,147.17) .. controls (218.67,143.76) and (221.43,141) .. (224.83,141) .. controls (228.24,141) and (231,143.76) .. (231,147.17) .. controls (231,150.57) and (228.24,153.33) .. (224.83,153.33) .. controls (221.43,153.33) and (218.67,150.57) .. (218.67,147.17) -- cycle ;
%Shape: Circle [id:dp8393774396769696] 
\draw  [fill={rgb, 255:red, 0; green, 0; blue, 0 }  ,fill opacity=1 ] (119.67,429.17) .. controls (119.67,425.76) and (122.43,423) .. (125.83,423) .. controls (129.24,423) and (132,425.76) .. (132,429.17) .. controls (132,432.57) and (129.24,435.33) .. (125.83,435.33) .. controls (122.43,435.33) and (119.67,432.57) .. (119.67,429.17) -- cycle ;
%Shape: Circle [id:dp9122671404367513] 
\draw  [fill={rgb, 255:red, 0; green, 0; blue, 0 }  ,fill opacity=1 ] (220.67,429.17) .. controls (220.67,425.76) and (223.43,423) .. (226.83,423) .. controls (230.24,423) and (233,425.76) .. (233,429.17) .. controls (233,432.57) and (230.24,435.33) .. (226.83,435.33) .. controls (223.43,435.33) and (220.67,432.57) .. (220.67,429.17) -- cycle ;
%Straight Lines [id:da9730608195272499] 
\draw [color={rgb, 255:red, 208; green, 2; blue, 27 }  ,draw opacity=1 ][line width=3]    (123.83,100.17) -- (224.83,100.17) ;
%Straight Lines [id:da7285639839609988] 
\draw [color={rgb, 255:red, 208; green, 2; blue, 27 }  ,draw opacity=1 ][line width=3]    (123.83,192.17) -- (224.83,192.17) ;
%Straight Lines [id:da15378467961255593] 
\draw [color={rgb, 255:red, 208; green, 2; blue, 27 }  ,draw opacity=1 ][line width=3]    (123.83,240.17) -- (224.83,240.17) ;
%Straight Lines [id:da32488308053600357] 
\draw [color={rgb, 255:red, 208; green, 2; blue, 27 }  ,draw opacity=1 ][line width=3]    (124.83,380.17) -- (226.83,429.17) ;
%Straight Lines [id:da5713934778232975] 
\draw [color={rgb, 255:red, 208; green, 2; blue, 27 }  ,draw opacity=1 ][line width=3]    (125.83,429.17) -- (225.83,380.17) ;
%Straight Lines [id:da5167672956156528] 
\draw [color={rgb, 255:red, 208; green, 2; blue, 27 }  ,draw opacity=1 ][line width=3]    (124.83,329.17) -- (224.83,286.17) ;
%Straight Lines [id:da7135455164369313] 
\draw [color={rgb, 255:red, 74; green, 144; blue, 226 }  ,draw opacity=1 ][line width=3]    (123.83,100.17) -- (224.83,58.17) ;
%Straight Lines [id:da6817798150754134] 
\draw [color={rgb, 255:red, 74; green, 144; blue, 226 }  ,draw opacity=1 ][line width=3]    (123.83,147.17) -- (224.83,147.17) ;
%Straight Lines [id:da9588900717684055] 
\draw [color={rgb, 255:red, 74; green, 144; blue, 226 }  ,draw opacity=1 ][line width=3]    (123.83,198.33) -- (224.83,198.33) ;
%Straight Lines [id:da153465805912531] 
\draw [color={rgb, 255:red, 74; green, 144; blue, 226 }  ,draw opacity=1 ][line width=3]    (123.83,246.33) -- (224.83,246.33) ;
%Straight Lines [id:da15078448721396742] 
\draw [color={rgb, 255:red, 74; green, 144; blue, 226 }  ,draw opacity=1 ][line width=3]    (124.83,380.17) -- (225.83,380.17) ;
%Straight Lines [id:da19411220419182174] 
\draw [color={rgb, 255:red, 74; green, 144; blue, 226 }  ,draw opacity=1 ][line width=3]    (125.83,429.17) -- (226.83,429.17) ;
%Straight Lines [id:da41580313320901685] 
\draw [color={rgb, 255:red, 74; green, 144; blue, 226 }  ,draw opacity=1 ][line width=3]    (123.83,286.17) -- (224.83,286.17) ;
%Shape: Free Drawing [id:dp8115343691729312] 
\draw  [line width=3] [line join = round][line cap = round] (73.67,92) .. controls (73.67,100.33) and (73.67,108.67) .. (73.67,117) ;
%Shape: Free Drawing [id:dp5652013101314327] 
\draw  [line width=3] [line join = round][line cap = round] (73.67,92) .. controls (82.55,92) and (84.35,102) .. (75.67,102) ;
%Shape: Free Drawing [id:dp09638673838844214] 
\draw  [line width=3] [line join = round][line cap = round] (80.67,110) .. controls (80.67,113.35) and (81.67,116.65) .. (81.67,120) ;
%Shape: Free Drawing [id:dp16779380553630863] 
\draw  [line width=3] [line join = round][line cap = round] (76.67,143) .. controls (76.67,149) and (76.67,155) .. (76.67,161) ;
%Shape: Free Drawing [id:dp4944842948023188] 
\draw  [line width=3] [line join = round][line cap = round] (75.67,141) .. controls (84.35,141) and (89.14,152) .. (78.67,152) ;
%Shape: Free Drawing [id:dp5715476750673988] 
\draw  [line width=3] [line join = round][line cap = round] (83.67,159) .. controls (83.67,157.77) and (87.21,155.26) .. (87.67,158) .. controls (88.01,160.08) and (87.16,162.51) .. (85.67,164) .. controls (85.67,164) and (84.37,165.95) .. (84.67,166) .. controls (86.64,166.33) and (88.67,166) .. (90.67,166) ;
%Shape: Free Drawing [id:dp5183191137922568] 
\draw  [line width=3] [line join = round][line cap = round] (74.67,192) .. controls (74.67,191.53) and (75.33,191.33) .. (75.67,191) .. controls (81.14,185.53) and (75.47,201.27) .. (80.67,203) .. controls (84.8,204.38) and (85.11,196.25) .. (85.67,194) .. controls (86.91,189.01) and (86.28,205.69) .. (91.67,203) .. controls (95.66,201) and (94.67,193.94) .. (94.67,191) ;
%Shape: Free Drawing [id:dp9231006744893211] 
\draw  [line width=3] [line join = round][line cap = round] (99.67,199) .. controls (99.67,202) and (99.67,205) .. (99.67,208) ;
%Shape: Free Drawing [id:dp7604321625438055] 
\draw  [line width=3] [line join = round][line cap = round] (67.67,236.33) .. controls (67.67,234.06) and (71.35,230.02) .. (73.67,232.33) .. controls (75.91,234.58) and (73.51,243.76) .. (76.67,245.33) .. controls (81.23,247.62) and (81.4,238.69) .. (81.67,237.33) .. controls (81.87,236.3) and (82.67,233.28) .. (82.67,234.33) .. controls (82.67,253.28) and (91.67,246.41) .. (91.67,231.33) ;
%Shape: Free Drawing [id:dp9266323930140261] 
\draw  [line width=3] [line join = round][line cap = round] (93.67,244.33) .. controls (94.56,244.33) and (96.02,242.69) .. (96.67,243.33) .. controls (97.89,244.56) and (98.34,245.97) .. (97.67,249.33) .. controls (97.34,250.97) and (93.09,252.81) .. (94.67,253.33) .. controls (98.02,254.45) and (98.12,253.88) .. (99.67,252.33) ;
%Shape: Free Drawing [id:dp277125743077116] 
\draw  [line width=3] [line join = round][line cap = round] (78.67,279.33) .. controls (78.67,282.12) and (78.27,303.33) .. (73.67,303.33) ;
%Shape: Free Drawing [id:dp25619849956295904] 
\draw  [line width=3] [line join = round][line cap = round] (73.67,278.33) .. controls (76.69,278.33) and (79.67,276.96) .. (82.67,277.33) .. controls (85.43,277.68) and (87.49,283.78) .. (80.67,288.33) ;
%Shape: Free Drawing [id:dp17923482179520778] 
\draw  [line width=3] [line join = round][line cap = round] (85.67,296.67) .. controls (85.67,296.06) and (88.82,293.82) .. (89.67,294.67) .. controls (91.55,296.55) and (90.78,299.61) .. (88.67,300.67) .. controls (88.34,300.83) and (93.95,306.67) .. (86.67,306.67) ;
%Shape: Free Drawing [id:dp6252609398726262] 
\draw  [line width=3] [line join = round][line cap = round] (79.67,327.67) .. controls (79.67,330.38) and (80.07,346.67) .. (74.67,346.67) ;
%Shape: Free Drawing [id:dp30753370220933673] 
\draw  [line width=3] [line join = round][line cap = round] (75.67,325.67) .. controls (78.21,325.67) and (82.93,322.93) .. (85.67,325.67) .. controls (90.36,330.36) and (85.12,334.67) .. (80.67,334.67) ;
%Shape: Free Drawing [id:dp3002950883271912] 
\draw  [line width=3] [line join = round][line cap = round] (84.67,345.67) .. controls (86,345.67) and (87.72,344.72) .. (88.67,345.67) .. controls (91.78,348.78) and (85.43,351.55) .. (85.67,351.67) .. controls (89.69,353.68) and (92.14,356.67) .. (83.67,356.67) ;
%Shape: Free Drawing [id:dp96697503929667] 
\draw  [line width=3] [line join = round][line cap = round] (72.67,318.67) .. controls (77.35,318.67) and (81.99,319.67) .. (86.67,319.67) ;
%Shape: Free Drawing [id:dp9157285032450893] 
\draw  [line width=3] [line join = round][line cap = round] (88.67,378.67) .. controls (88.67,383.51) and (92.22,397.67) .. (84.67,397.67) ;
%Shape: Free Drawing [id:dp03079008816529949] 
\draw  [line width=3] [line join = round][line cap = round] (86.67,377.67) .. controls (93.48,377.67) and (98.44,385.67) .. (90.67,385.67) ;
%Shape: Free Drawing [id:dp5973088897847608] 
\draw  [line width=3] [line join = round][line cap = round] (95.67,391.67) .. controls (95.67,393.33) and (95.67,395) .. (95.67,396.67) .. controls (95.67,397) and (95.34,395.73) .. (95.67,395.67) .. controls (97.3,395.34) and (99.24,396.52) .. (100.67,395.67) .. controls (101.52,395.15) and (100.67,391.67) .. (100.67,392.67) .. controls (100.67,395.33) and (100.67,398) .. (100.67,400.67) ;
%Shape: Free Drawing [id:dp6951087831249594] 
\draw  [line width=3] [line join = round][line cap = round] (88.67,424.67) .. controls (93.18,424.67) and (88.67,436.96) .. (88.67,440.67) ;
%Shape: Free Drawing [id:dp36845536185546424] 
\draw  [line width=3] [line join = round][line cap = round] (87.67,424.67) .. controls (94.05,424.67) and (103.07,430.67) .. (92.67,430.67) ;
%Shape: Free Drawing [id:dp02265560859629201] 
\draw  [line width=3] [line join = round][line cap = round] (94.67,436.67) .. controls (94.67,444.9) and (99.51,437.39) .. (99.67,438.67) .. controls (100,441.31) and (99.67,444) .. (99.67,446.67) ;
%Shape: Free Drawing [id:dp14598298473151516] 
\draw  [line width=3] [line join = round][line cap = round] (97.67,419.67) .. controls (93.67,419.67) and (89.67,419.67) .. (85.67,419.67) ;
%Shape: Free Drawing [id:dp5711682961038158] 
\draw  [line width=3] [line join = round][line cap = round] (247.67,372.33) .. controls (247.67,377.18) and (251.22,391.33) .. (243.67,391.33) ;
%Shape: Free Drawing [id:dp14754915316340156] 
\draw  [line width=3] [line join = round][line cap = round] (245.67,371.33) .. controls (252.48,371.33) and (257.44,379.33) .. (249.67,379.33) ;
%Shape: Free Drawing [id:dp3481558481530256] 
\draw  [line width=3] [line join = round][line cap = round] (254.67,385.33) .. controls (254.67,387) and (254.67,388.67) .. (254.67,390.33) .. controls (254.67,390.67) and (254.34,389.4) .. (254.67,389.33) .. controls (256.3,389.01) and (258.24,390.19) .. (259.67,389.33) .. controls (260.52,388.82) and (259.67,385.33) .. (259.67,386.33) .. controls (259.67,389) and (259.67,391.67) .. (259.67,394.33) ;
%Shape: Free Drawing [id:dp506753801076842] 
\draw  [line width=3] [line join = round][line cap = round] (247.67,418.33) .. controls (252.18,418.33) and (247.67,430.63) .. (247.67,434.33) ;
%Shape: Free Drawing [id:dp6852970681901286] 
\draw  [line width=3] [line join = round][line cap = round] (246.67,418.33) .. controls (253.05,418.33) and (262.07,424.33) .. (251.67,424.33) ;
%Shape: Free Drawing [id:dp9839888611815865] 
\draw  [line width=3] [line join = round][line cap = round] (253.67,430.33) .. controls (253.67,438.57) and (258.51,431.05) .. (258.67,432.33) .. controls (259,434.98) and (258.67,437.67) .. (258.67,440.33) ;
%Shape: Free Drawing [id:dp5045941870852074] 
\draw  [line width=3] [line join = round][line cap = round] (256.67,413.33) .. controls (252.67,413.33) and (248.67,413.33) .. (244.67,413.33) ;
%Shape: Free Drawing [id:dp5852154924212069] 
\draw  [line width=3] [line join = round][line cap = round] (243.67,192) .. controls (243.67,191.53) and (244.33,191.33) .. (244.67,191) .. controls (250.14,185.53) and (244.47,201.27) .. (249.67,203) .. controls (253.8,204.38) and (254.11,196.25) .. (254.67,194) .. controls (255.91,189.01) and (255.28,205.69) .. (260.67,203) .. controls (264.66,201) and (263.67,193.94) .. (263.67,191) ;
%Shape: Free Drawing [id:dp6674314503863402] 
\draw  [line width=3] [line join = round][line cap = round] (268.67,199) .. controls (268.67,202) and (268.67,205) .. (268.67,208) ;
%Shape: Free Drawing [id:dp1674760681704487] 
\draw  [line width=3] [line join = round][line cap = round] (236.67,236.33) .. controls (236.67,234.06) and (240.35,230.02) .. (242.67,232.33) .. controls (244.91,234.58) and (242.51,243.76) .. (245.67,245.33) .. controls (250.23,247.62) and (250.4,238.69) .. (250.67,237.33) .. controls (250.87,236.3) and (251.67,233.28) .. (251.67,234.33) .. controls (251.67,253.28) and (260.67,246.41) .. (260.67,231.33) ;
%Shape: Free Drawing [id:dp5796927308647466] 
\draw  [line width=3] [line join = round][line cap = round] (262.67,244.33) .. controls (263.56,244.33) and (265.02,242.69) .. (265.67,243.33) .. controls (266.89,244.56) and (267.34,245.97) .. (266.67,249.33) .. controls (266.34,250.97) and (262.09,252.81) .. (263.67,253.33) .. controls (267.02,254.45) and (267.12,253.88) .. (268.67,252.33) ;
%Shape: Free Drawing [id:dp18604784568260058] 
\draw  [line width=3] [line join = round][line cap = round] (243.67,90) .. controls (243.67,98.33) and (243.67,106.67) .. (243.67,115) ;
%Shape: Free Drawing [id:dp9505307039446799] 
\draw  [line width=3] [line join = round][line cap = round] (243.67,90) .. controls (252.55,90) and (254.35,100) .. (245.67,100) ;
%Shape: Free Drawing [id:dp519585651328846] 
\draw  [line width=3] [line join = round][line cap = round] (250.67,108) .. controls (250.67,111.35) and (251.67,114.65) .. (251.67,118) ;
%Shape: Free Drawing [id:dp4954911510264087] 
\draw  [line width=3] [line join = round][line cap = round] (246.67,141) .. controls (246.67,147) and (246.67,153) .. (246.67,159) ;
%Shape: Free Drawing [id:dp6888833397566129] 
\draw  [line width=3] [line join = round][line cap = round] (245.67,139) .. controls (254.35,139) and (259.14,150) .. (248.67,150) ;
%Shape: Free Drawing [id:dp5926851449630373] 
\draw  [line width=3] [line join = round][line cap = round] (253.67,157) .. controls (253.67,155.77) and (257.21,153.26) .. (257.67,156) .. controls (258.01,158.08) and (257.16,160.51) .. (255.67,162) .. controls (255.67,162) and (254.37,163.95) .. (254.67,164) .. controls (256.64,164.33) and (258.67,164) .. (260.67,164) ;
%Shape: Free Drawing [id:dp45362616368289055] 
\draw  [line width=3] [line join = round][line cap = round] (247.67,273.33) .. controls (247.67,276.12) and (247.27,297.33) .. (242.67,297.33) ;
%Shape: Free Drawing [id:dp3853807802467195] 
\draw  [line width=3] [line join = round][line cap = round] (242.67,272.33) .. controls (245.69,272.33) and (248.67,270.96) .. (251.67,271.33) .. controls (254.43,271.68) and (256.49,277.78) .. (249.67,282.33) ;
%Shape: Free Drawing [id:dp205039994305506] 
\draw  [line width=3] [line join = round][line cap = round] (254.67,290.67) .. controls (254.67,290.06) and (257.82,287.82) .. (258.67,288.67) .. controls (260.55,290.55) and (259.78,293.61) .. (257.67,294.67) .. controls (257.34,294.83) and (262.95,300.67) .. (255.67,300.67) ;
%Shape: Free Drawing [id:dp7957813818594729] 
\draw  [line width=3] [line join = round][line cap = round] (245.67,329) .. controls (247.65,330.98) and (247.37,347) .. (242.67,347) ;
%Shape: Free Drawing [id:dp011642300287247997] 
\draw  [line width=3] [line join = round][line cap = round] (244.67,327) .. controls (246.69,327) and (248.7,325.51) .. (250.67,326) .. controls (254.53,326.96) and (257.28,334.76) .. (251.67,337) .. controls (250.39,337.51) and (249.04,336) .. (247.67,336) ;
%Shape: Free Drawing [id:dp39060047625029337] 
\draw  [line width=3] [line join = round][line cap = round] (256.67,341) .. controls (256.67,343) and (256.67,345) .. (256.67,347) .. controls (256.67,347.33) and (256.37,346.15) .. (256.67,346) .. controls (258,345.34) and (261.54,347.62) .. (261.67,348) .. controls (263.27,352.8) and (258.39,352) .. (256.67,352) ;
%Shape: Free Drawing [id:dp9858955361610706] 
\draw  [line width=3] [line join = round][line cap = round] (256.67,340) .. controls (258.69,340) and (260.64,339) .. (262.67,339) ;
%Shape: Free Drawing [id:dp26329311865593774] 
\draw  [line width=3] [line join = round][line cap = round] (243.67,51) .. controls (243.67,57) and (243.67,63) .. (243.67,69) ;
%Shape: Free Drawing [id:dp43222161502815293] 
\draw  [line width=3] [line join = round][line cap = round] (242.67,50) .. controls (252.33,50) and (256.39,57) .. (245.67,57) ;
%Shape: Free Drawing [id:dp5344981978065677] 
\draw  [line width=3] [line join = round][line cap = round] (251.67,65) .. controls (251.67,67) and (251.67,69) .. (251.67,71) ;
%Shape: Free Drawing [id:dp8150905062942994] 
\draw  [line width=3] [line join = round][line cap = round] (241.67,83) .. controls (245.02,83) and (248.32,84) .. (251.67,84) ;
%Shape: Free Drawing [id:dp19206691270823328] 
\draw  [line width=1.5] [line join = round][line cap = round] (406.67,47) .. controls (399.15,47) and (395.77,59) .. (406.67,59) ;
%Shape: Free Drawing [id:dp18293218257475452] 
\draw  [line width=1.5] [line join = round][line cap = round] (414.67,58) .. controls (414.67,60.06) and (411.67,61.94) .. (411.67,64) ;
%Shape: Free Drawing [id:dp3598259732753044] 
\draw  [line width=1.5] [line join = round][line cap = round] (419.67,53) .. controls (421.37,53) and (420.04,56.42) .. (420.67,58) .. controls (421.44,59.94) and (424.67,53) .. (424.67,53) .. controls (424.67,53) and (427.42,59) .. (427.67,58) .. controls (428.39,55.11) and (428.67,53.69) .. (428.67,51) ;
%Shape: Free Drawing [id:dp9043206023529841] 
\draw  [line width=1.5] [line join = round][line cap = round] (431.67,57) .. controls (431.67,58.33) and (431.67,59.67) .. (431.67,61) ;
%Shape: Free Drawing [id:dp37983257660019143] 
\draw  [line width=1.5] [line join = round][line cap = round] (438.67,60) .. controls (438.67,61.27) and (437.9,65) .. (436.67,65) ;
%Shape: Free Drawing [id:dp5875639602985142] 
\draw  [line width=1.5] [line join = round][line cap = round] (439.67,53) .. controls (440.11,53) and (442.14,50.68) .. (442.67,52) .. controls (443.3,53.58) and (442.46,55.8) .. (443.67,57) .. controls (444,57.33) and (444.33,58.33) .. (444.67,58) .. controls (445.08,57.59) and (446.67,52) .. (446.67,52) .. controls (446.67,52) and (448.72,60.9) .. (449.67,59) .. controls (450.79,56.75) and (451.67,52.61) .. (451.67,50) ;
%Shape: Free Drawing [id:dp39185551295647203] 
\draw  [line width=1.5] [line join = round][line cap = round] (453.67,58) .. controls (454.26,57.4) and (455.43,55.76) .. (456.67,57) .. controls (458.04,58.37) and (451.72,62) .. (453.67,62) .. controls (454.67,62) and (455.67,62) .. (456.67,62) ;
%Shape: Free Drawing [id:dp09991604834235557] 
\draw  [line width=1.5] [line join = round][line cap = round] (460.67,59) .. controls (460.67,60.31) and (459.99,63) .. (458.67,63) ;
%Shape: Free Drawing [id:dp21467987349707462] 
\draw  [line width=1.5] [line join = round][line cap = round] (466.67,51) .. controls (463.63,54.04) and (461.58,61) .. (466.67,61) ;
%Shape: Free Drawing [id:dp3947144142303016] 
\draw  [line width=1.5] [line join = round][line cap = round] (469.67,55) .. controls (469.67,58.67) and (469.67,62.33) .. (469.67,66) ;
%Shape: Free Drawing [id:dp8020730867940765] 
\draw  [line width=1.5] [line join = round][line cap = round] (469.67,54) .. controls (473.73,54) and (475.51,59) .. (470.67,59) ;
%Shape: Free Drawing [id:dp8113913828449837] 
\draw  [line width=1.5] [line join = round][line cap = round] (474.67,60) .. controls (474.67,60.67) and (474.67,61.33) .. (474.67,62) ;
%Shape: Free Drawing [id:dp9285011892484396] 
\draw  [line width=1.5] [line join = round][line cap = round] (479.67,59) .. controls (479.67,59.79) and (478.66,62) .. (477.67,62) ;
%Shape: Free Drawing [id:dp05638676042789381] 
\draw  [line width=1.5] [line join = round][line cap = round] (482.67,56) .. controls (482.67,58.4) and (482.67,62.08) .. (482.67,64) ;
%Shape: Free Drawing [id:dp15998998250440244] 
\draw  [line width=1.5] [line join = round][line cap = round] (481.67,53) .. controls (486.78,53) and (490.44,59) .. (484.67,59) ;
%Shape: Free Drawing [id:dp3328583979494384] 
\draw  [line width=1.5] [line join = round][line cap = round] (490.67,59) .. controls (490.67,60) and (490.67,61) .. (490.67,62) ;
%Shape: Free Drawing [id:dp5437333647129085] 
\draw  [line width=1.5] [line join = round][line cap = round] (480.67,49) .. controls (482.67,49) and (484.67,49) .. (486.67,49) ;
%Shape: Free Drawing [id:dp30200851870978085] 
\draw  [line width=1.5] [line join = round][line cap = round] (492.67,51) .. controls (496.33,51) and (494.67,58.85) .. (494.67,61) ;
%Shape: Free Drawing [id:dp6396806990180974] 
\draw  [line width=1.5] [line join = round][line cap = round] (500.67,58) .. controls (500.67,59.28) and (499.49,63) .. (497.67,63) ;
%Shape: Free Drawing [id:dp6465296198601902] 
\draw  [line width=1.5] [line join = round][line cap = round] (505.67,60) .. controls (505.67,57.65) and (503.31,60) .. (505.67,60) ;
%Shape: Free Drawing [id:dp7965021574452498] 
\draw  [line width=1.5] [line join = round][line cap = round] (509.67,60) .. controls (510,60) and (510.33,60) .. (510.67,60) ;
%Shape: Free Drawing [id:dp6261815460818462] 
\draw  [line width=1.5] [line join = round][line cap = round] (513.67,59) .. controls (513.67,60.5) and (514.67,60.5) .. (514.67,59) ;
%Shape: Free Drawing [id:dp8029209987451678] 
\draw  [line width=1.5] [line join = round][line cap = round] (519.67,58) .. controls (519.67,59.27) and (518.9,63) .. (517.67,63) ;
%Shape: Free Drawing [id:dp0407022324875137] 
\draw  [line width=1.5] [line join = round][line cap = round] (528.67,50) .. controls (528.67,51.54) and (526.38,52.51) .. (525.67,55) .. controls (525.11,56.95) and (526.67,58.97) .. (526.67,61) ;
%Shape: Free Drawing [id:dp9918228594513934] 
\draw  [line width=1.5] [line join = round][line cap = round] (537.67,54) .. controls (537.67,57.37) and (536.69,61.97) .. (534.67,64) ;
%Shape: Free Drawing [id:dp5685207242710275] 
\draw  [line width=1.5] [line join = round][line cap = round] (534.67,52) .. controls (540.51,52) and (546.23,58) .. (539.67,58) ;
%Shape: Free Drawing [id:dp24816507418839528] 
\draw  [line width=1.5] [line join = round][line cap = round] (544.67,59) .. controls (544.67,60.13) and (550.4,61.64) .. (545.67,64) .. controls (545,64.33) and (544.41,63) .. (543.67,63) ;
%Shape: Free Drawing [id:dp621506684559653] 
\draw  [line width=1.5] [line join = round][line cap = round] (544.67,58) .. controls (545.67,58) and (546.67,58) .. (547.67,58) ;
%Shape: Free Drawing [id:dp5456826740657755] 
\draw  [line width=1.5] [line join = round][line cap = round] (553.67,58) .. controls (553.67,59.05) and (552.67,59.95) .. (552.67,61) ;
%Shape: Free Drawing [id:dp30518970014465985] 
\draw  [line width=1.5] [line join = round][line cap = round] (558.67,53) .. controls (558.67,58.24) and (556.67,60.91) .. (556.67,64) ;
%Shape: Free Drawing [id:dp4311220873848257] 
\draw  [line width=1.5] [line join = round][line cap = round] (556.67,53) .. controls (558.04,53) and (559.36,51.57) .. (560.67,52) .. controls (564.16,53.16) and (563.07,58) .. (558.67,58) ;
%Shape: Free Drawing [id:dp8712974105245386] 
\draw  [line width=1.5] [line join = round][line cap = round] (564.67,58) .. controls (564.67,59.05) and (562.67,60.67) .. (563.67,61) .. controls (566.93,62.09) and (568.58,65) .. (563.67,65) ;
%Shape: Free Drawing [id:dp20161886624229175] 
\draw  [line width=1.5] [line join = round][line cap = round] (564.67,58) .. controls (565.41,58) and (566.67,57.75) .. (566.67,57) ;
%Shape: Free Drawing [id:dp5291836024972196] 
\draw  [line width=1.5] [line join = round][line cap = round] (556.67,48) .. controls (559.02,48) and (561.31,49) .. (563.67,49) ;
%Shape: Free Drawing [id:dp4586128874588673] 
\draw  [line width=1.5] [line join = round][line cap = round] (569.67,50) .. controls (572.41,52.74) and (572.27,55.59) .. (571.67,58) .. controls (571.33,59.33) and (572.04,62) .. (570.67,62) ;
%Shape: Free Drawing [id:dp7339753520127316] 
\draw  [line width=1.5] [line join = round][line cap = round] (571.67,38) .. controls (571.67,39.64) and (575.87,42.4) .. (576.67,44) .. controls (579.25,49.16) and (579.1,59.14) .. (576.67,64) .. controls (575.95,65.44) and (576.6,69) .. (574.67,69) ;
%Shape: Free Drawing [id:dp6659589824435268] 
\draw  [line width=1.5] [line join = round][line cap = round] (403.67,36) .. controls (393.38,36) and (386.03,59.36) .. (394.67,68) ;
%Shape: Free Drawing [id:dp16812703499070591] 
\draw  [line width=1.5] [line join = round][line cap = round] (416.57,82.05) .. controls (416.6,82.45) and (416.9,82.97) .. (416.66,83.26) .. controls (411.37,89.55) and (404.52,97.16) .. (399.63,107.46) .. controls (391.93,123.69) and (382.81,149.32) .. (384.1,166.74) .. controls (384.16,167.55) and (384.3,165.04) .. (383.92,164.32) .. controls (382.56,161.77) and (382.72,158.68) .. (382.18,155.85) .. controls (381.64,152.94) and (380.29,150.21) .. (379.34,147.38) .. controls (379.05,146.52) and (377.99,144.06) .. (378.06,144.96) .. controls (378.51,151.12) and (381.82,158.73) .. (382.63,161.9) .. controls (383.14,163.88) and (381.41,168.66) .. (383.08,167.95) .. controls (388.93,165.48) and (392.58,154.64) .. (398.7,154.64) ;
%Shape: Free Drawing [id:dp016953438012783373] 
\draw  [line width=1.5] [line join = round][line cap = round] (504.67,80) .. controls (508.15,82.32) and (513.66,82.6) .. (516.67,85) .. controls (533.04,98.1) and (536.98,114.57) .. (540.67,133) .. controls (541.28,136.06) and (546.04,147.62) .. (544.67,149) .. controls (544,149.67) and (543.33,147.67) .. (542.67,147) .. controls (541.5,145.83) and (540.83,144.17) .. (539.67,143) .. controls (538.83,142.17) and (536.32,140.65) .. (535.67,140) .. controls (534.92,139.25) and (531.61,139) .. (532.67,139) .. controls (539.94,139) and (541.72,146.05) .. (545.67,150) .. controls (546.34,150.67) and (549.56,140.21) .. (549.67,140) .. controls (550.52,138.3) and (552.67,135.93) .. (552.67,134) ;
%Shape: Free Drawing [id:dp004734063190034554] 
\draw  [line width=1.5] [line join = round][line cap = round] (369.67,89) .. controls (378.54,89) and (373.24,80.71) .. (370.67,82) .. controls (363.8,85.43) and (369.67,97.33) .. (369.67,105) ;
%Shape: Free Drawing [id:dp2514414407659582] 
\draw  [line width=1.5] [line join = round][line cap = round] (364.67,97) .. controls (367.67,97) and (370.67,97) .. (373.67,97) ;
%Shape: Free Drawing [id:dp5910814741825017] 
\draw  [line width=1.5] [line join = round][line cap = round] (374.67,104) .. controls (374.67,104.47) and (375.55,104.54) .. (375.67,105) .. controls (377.44,112.11) and (377.15,106.52) .. (380.67,103) ;
%Shape: Free Drawing [id:dp45071639265067387] 
\draw  [line width=1.5] [line join = round][line cap = round] (548.67,84) .. controls (547.4,84) and (539.77,84.2) .. (541.67,88) .. controls (542.1,88.87) and (545.16,89.75) .. (546.67,89) .. controls (548.19,88.24) and (547.67,82.3) .. (547.67,84) .. controls (547.67,89.68) and (549.23,95.35) .. (548.67,101) .. controls (548.52,102.48) and (545.22,104.38) .. (544.67,103) .. controls (541.4,94.83) and (552.67,93.84) .. (552.67,89) ;
%Shape: Free Drawing [id:dp9140803545493887] 
\draw  [line width=1.5] [line join = round][line cap = round] (552.67,98) .. controls (554.54,99.87) and (553.72,101.11) .. (554.67,103) .. controls (554.9,103.47) and (556.95,99) .. (557.67,99) ;
%Shape: Free Drawing [id:dp4378502227932739] 
\draw  [line width=1.5] [line join = round][line cap = round] (316.67,175.33) .. controls (315.29,176.71) and (310.82,180.14) .. (310.67,181.33) .. controls (310.26,184.62) and (311.28,197.33) .. (315.67,197.33) ;
%Shape: Free Drawing [id:dp19147293958806255] 
\draw  [line width=1.5] [line join = round][line cap = round] (325.67,179.33) .. controls (316.2,179.33) and (316.55,195.33) .. (326.67,195.33) ;
%Shape: Free Drawing [id:dp3714542360301637] 
\draw  [line width=1.5] [line join = round][line cap = round] (332.67,192.33) .. controls (332.67,195.4) and (330.67,197.55) .. (330.67,199.33) ;
%Shape: Free Drawing [id:dp6528936113808104] 
\draw  [line width=1.5] [line join = round][line cap = round] (340.67,186) .. controls (340.67,192.29) and (342.46,207) .. (338.67,207) ;
%Shape: Free Drawing [id:dp6222108847256875] 
\draw  [line width=1.5] [line join = round][line cap = round] (338.67,186) .. controls (347.49,186) and (349.35,193) .. (340.67,193) ;
%Shape: Free Drawing [id:dp34610759500514] 
\draw  [line width=1.5] [line join = round][line cap = round] (347.67,196) .. controls (347.67,198.33) and (347.67,200.67) .. (347.67,203) ;
%Shape: Free Drawing [id:dp27814447132397824] 
\draw  [line width=1.5] [line join = round][line cap = round] (354.67,197) .. controls (354.67,198.39) and (352.67,200.61) .. (352.67,202) ;
%Shape: Free Drawing [id:dp6875020465663666] 
\draw  [line width=1.5] [line join = round][line cap = round] (359.67,190) .. controls (359.67,195) and (359.67,200) .. (359.67,205) ;
%Shape: Free Drawing [id:dp8317643428646235] 
\draw  [line width=1.5] [line join = round][line cap = round] (356.67,188) .. controls (359.67,188) and (362.68,187.7) .. (365.67,188) .. controls (368.49,188.28) and (362.39,194) .. (360.67,194) ;
%Shape: Free Drawing [id:dp15206520000770052] 
\draw  [line width=1.5] [line join = round][line cap = round] (365.67,196) .. controls (366.49,196) and (368.26,192.66) .. (369.67,195) .. controls (372.4,199.56) and (364.71,202) .. (371.67,202) ;
%Shape: Free Drawing [id:dp9483745329462451] 
\draw  [line width=1.5] [line join = round][line cap = round] (379.67,196) .. controls (377.8,197.86) and (378.63,199) .. (375.67,199) ;
%Shape: Free Drawing [id:dp465804446039885] 
\draw  [line width=1.5] [line join = round][line cap = round] (382.67,189) .. controls (387.08,184.59) and (386.86,193.59) .. (387.67,196) .. controls (388.37,198.12) and (389.09,191.58) .. (390.67,190) .. controls (390.84,189.82) and (392.97,198.54) .. (394.67,196) .. controls (395.56,194.65) and (395.67,188.08) .. (395.67,187) ;
%Shape: Free Drawing [id:dp12292313327821747] 
\draw  [line width=1.5] [line join = round][line cap = round] (397.67,195) .. controls (397.67,196.33) and (397.67,197.67) .. (397.67,199) ;
%Shape: Free Drawing [id:dp059329621696750046] 
\draw  [line width=1.5] [line join = round][line cap = round] (403.67,196) .. controls (403.67,199.57) and (400.67,200.59) .. (400.67,203) ;
%Shape: Free Drawing [id:dp8913571365460969] 
\draw  [line width=1.5] [line join = round][line cap = round] (404.67,189) .. controls (410.04,183.63) and (408.74,189.52) .. (409.67,196) .. controls (409.84,197.2) and (412.62,192.29) .. (412.67,192) .. controls (412.94,190.36) and (412.67,188.67) .. (412.67,187) .. controls (412.67,186.25) and (413.49,188.28) .. (413.67,189) .. controls (413.97,190.21) and (414,195.01) .. (414.67,196) .. controls (416.39,198.58) and (418.67,188.77) .. (418.67,188) ;
%Shape: Free Drawing [id:dp8446766214082211] 
\draw  [line width=1.5] [line join = round][line cap = round] (418.67,194) .. controls (420,194) and (421.83,192.96) .. (422.67,194) .. controls (426.75,199.1) and (415.99,201) .. (423.67,201) ;
%Shape: Free Drawing [id:dp31921658080126836] 
\draw  [line width=1.5] [line join = round][line cap = round] (429.67,196) .. controls (429.67,197.63) and (429.19,199) .. (427.67,199) ;
%Shape: Free Drawing [id:dp014872878058225458] 
\draw  [line width=1.5] [line join = round][line cap = round] (434.67,189) .. controls (434.67,191.86) and (439.57,207) .. (432.67,207) ;
%Shape: Free Drawing [id:dp8027979608232121] 
\draw  [line width=1.5] [line join = round][line cap = round] (433.67,190) .. controls (442.4,190) and (444.48,199) .. (436.67,199) ;
%Shape: Free Drawing [id:dp7223952474859204] 
\draw  [line width=1.5] [line join = round][line cap = round] (442.67,200) .. controls (443.67,200) and (444.96,199.29) .. (445.67,200) .. controls (449.61,203.94) and (444.04,202.81) .. (443.67,203) .. controls (443.57,203.05) and (447.08,205.88) .. (445.67,208) .. controls (444.59,209.61) and (442.67,206.42) .. (442.67,206) ;
%Shape: Free Drawing [id:dp4544285363548677] 
\draw  [line width=1.5] [line join = round][line cap = round] (453.67,199) .. controls (453.67,201.74) and (451.67,202.4) .. (451.67,204) ;
%Shape: Free Drawing [id:dp21108281900260406] 
\draw  [line width=1.5] [line join = round][line cap = round] (457.67,191) .. controls (457.67,195.05) and (459.5,204.67) .. (459.67,207) ;
%Shape: Free Drawing [id:dp11955192155071492] 
\draw  [line width=1.5] [line join = round][line cap = round] (455.67,190) .. controls (465.57,190) and (462.91,198) .. (457.67,198) ;
%Shape: Free Drawing [id:dp34081462651359296] 
\draw  [line width=1.5] [line join = round][line cap = round] (462.67,201) .. controls (462.67,199.21) and (462.33,197.83) .. (466.67,200) .. controls (468.19,200.76) and (464.48,204.63) .. (464.67,205) .. controls (465.44,206.54) and (472.15,209) .. (464.67,209) ;
%Shape: Free Drawing [id:dp942033956440704] 
\draw  [line width=1.5] [line join = round][line cap = round] (454.67,185) .. controls (458.09,185) and (460.59,187) .. (464.67,187) ;
%Shape: Free Drawing [id:dp49331468227469977] 
\draw  [line width=1.5] [line join = round][line cap = round] (472.67,199) .. controls (470.84,199) and (471.15,202.52) .. (469.67,204) ;
%Shape: Free Drawing [id:dp1101738485466125] 
\draw  [line width=1.5] [line join = round][line cap = round] (479.67,192) .. controls (479.67,196.68) and (480.67,201.32) .. (480.67,206) ;
%Shape: Free Drawing [id:dp6136440150155494] 
\draw  [line width=1.5] [line join = round][line cap = round] (474.67,193) .. controls (477.02,193) and (479.33,191.67) .. (481.67,192) .. controls (484.55,192.41) and (484.62,199) .. (480.67,199) ;
%Shape: Free Drawing [id:dp3454399314083071] 
\draw  [line width=1.5] [line join = round][line cap = round] (484.67,203) .. controls (484.67,204.08) and (483,209.83) .. (486.67,208) .. controls (488.8,206.93) and (488.67,202.98) .. (488.67,203) .. controls (488.67,205.36) and (489.67,207.64) .. (489.67,210) ;
%Shape: Free Drawing [id:dp3782860094919015] 
\draw  [line width=1.5] [line join = round][line cap = round] (494.67,204) .. controls (494.67,205.05) and (494.72,207) .. (493.67,207) ;
%Shape: Free Drawing [id:dp43292070353359724] 
\draw  [line width=1.5] [line join = round][line cap = round] (497.67,195) .. controls (500.78,195) and (500.67,210.22) .. (500.67,211) .. controls (500.67,213.33) and (501,206.31) .. (500.67,204) .. controls (500.19,200.63) and (496.63,196.72) .. (498.67,194) .. controls (499.81,192.48) and (503.62,195.83) .. (503.67,196) .. controls (504.18,198.05) and (501.67,199.89) .. (501.67,202) ;
%Shape: Free Drawing [id:dp22737424219332314] 
\draw  [line width=1.5] [line join = round][line cap = round] (504.67,203) .. controls (504.67,203.88) and (506.83,206.83) .. (508.67,205) .. controls (509.37,204.29) and (508.67,202) .. (508.67,202) .. controls (508.67,202) and (509.67,206.64) .. (509.67,209) ;
%Shape: Free Drawing [id:dp14872144143955013] 
\draw  [line width=1.5] [line join = round][line cap = round] (492.67,190) .. controls (496.35,190) and (499.98,189) .. (503.67,189) ;
%Shape: Free Drawing [id:dp6873641019525796] 
\draw  [line width=1.5] [line join = round][line cap = round] (511.67,187) .. controls (511.67,193.35) and (526.18,204.48) .. (516.67,214) ;
%Shape: Free Drawing [id:dp6925218679102185] 
\draw  [line width=1.5] [line join = round][line cap = round] (429.67,153.33) .. controls (428.29,154.71) and (423.82,158.14) .. (423.67,159.33) .. controls (423.26,162.62) and (424.28,175.33) .. (428.67,175.33) ;
%Shape: Free Drawing [id:dp9678621451892139] 
\draw  [line width=1.5] [line join = round][line cap = round] (438.67,157.33) .. controls (429.2,157.33) and (429.55,173.33) .. (439.67,173.33) ;
%Shape: Free Drawing [id:dp3474997423933385] 
\draw  [line width=1.5] [line join = round][line cap = round] (445.67,170.33) .. controls (445.67,173.4) and (443.67,175.55) .. (443.67,177.33) ;
%Shape: Free Drawing [id:dp4012491449402261] 
\draw  [line width=1.5] [line join = round][line cap = round] (468.67,161) .. controls (468.67,167.29) and (470.46,182) .. (466.67,182) ;
%Shape: Free Drawing [id:dp5659240105551916] 
\draw  [line width=1.5] [line join = round][line cap = round] (466.67,161) .. controls (475.49,161) and (477.35,168) .. (468.67,168) ;
%Shape: Free Drawing [id:dp37253363850686994] 
\draw  [line width=1.5] [line join = round][line cap = round] (475.67,171) .. controls (475.67,173.33) and (475.67,175.67) .. (475.67,178) ;
%Shape: Free Drawing [id:dp6269208713269812] 
\draw  [line width=1.5] [line join = round][line cap = round] (482.67,172) .. controls (482.67,173.39) and (480.67,175.61) .. (480.67,177) ;
%Shape: Free Drawing [id:dp5014351007954245] 
\draw  [line width=1.5] [line join = round][line cap = round] (487.67,165) .. controls (487.67,170) and (487.67,175) .. (487.67,180) ;
%Shape: Free Drawing [id:dp11072602704271906] 
\draw  [line width=1.5] [line join = round][line cap = round] (484.67,163) .. controls (487.67,163) and (490.68,162.7) .. (493.67,163) .. controls (496.49,163.28) and (490.39,169) .. (488.67,169) ;
%Shape: Free Drawing [id:dp3911553584694718] 
\draw  [line width=1.5] [line join = round][line cap = round] (493.67,171) .. controls (494.49,171) and (496.26,167.66) .. (497.67,170) .. controls (500.4,174.56) and (492.71,177) .. (499.67,177) ;
%Shape: Free Drawing [id:dp678245167180748] 
\draw  [line width=1.5] [line join = round][line cap = round] (507.67,171) .. controls (505.8,172.86) and (506.63,174) .. (503.67,174) ;
%Shape: Free Drawing [id:dp6172461917158508] 
\draw  [line width=1.5] [line join = round][line cap = round] (510.67,164) .. controls (515.08,159.59) and (514.86,168.59) .. (515.67,171) .. controls (516.37,173.12) and (517.09,166.58) .. (518.67,165) .. controls (518.84,164.82) and (520.97,173.54) .. (522.67,171) .. controls (523.56,169.65) and (523.67,163.08) .. (523.67,162) ;
%Shape: Free Drawing [id:dp5061868066983207] 
\draw  [line width=1.5] [line join = round][line cap = round] (525.67,170) .. controls (525.67,171.33) and (525.67,172.67) .. (525.67,174) ;
%Shape: Free Drawing [id:dp5887573404150941] 
\draw  [line width=1.5] [line join = round][line cap = round] (531.67,171) .. controls (531.67,174.57) and (528.67,175.59) .. (528.67,178) ;
%Shape: Free Drawing [id:dp8680683118655731] 
\draw  [line width=1.5] [line join = round][line cap = round] (532.67,164) .. controls (538.04,158.63) and (536.74,164.52) .. (537.67,171) .. controls (537.84,172.2) and (540.62,167.29) .. (540.67,167) .. controls (540.94,165.36) and (540.67,163.67) .. (540.67,162) .. controls (540.67,161.25) and (541.49,163.28) .. (541.67,164) .. controls (541.97,165.21) and (542,170.01) .. (542.67,171) .. controls (544.39,173.58) and (546.67,163.77) .. (546.67,163) ;
%Shape: Free Drawing [id:dp9488029155871764] 
\draw  [line width=1.5] [line join = round][line cap = round] (546.67,169) .. controls (548,169) and (549.83,167.96) .. (550.67,169) .. controls (554.75,174.1) and (543.99,176) .. (551.67,176) ;
%Shape: Free Drawing [id:dp9446687535544175] 
\draw  [line width=1.5] [line join = round][line cap = round] (557.67,171) .. controls (557.67,172.63) and (557.19,174) .. (555.67,174) ;
%Shape: Free Drawing [id:dp28812233123440467] 
\draw  [line width=1.5] [line join = round][line cap = round] (562.67,164) .. controls (562.67,166.86) and (567.57,182) .. (560.67,182) ;
%Shape: Free Drawing [id:dp738715889419021] 
\draw  [line width=1.5] [line join = round][line cap = round] (561.67,165) .. controls (570.4,165) and (572.48,174) .. (564.67,174) ;
%Shape: Free Drawing [id:dp7928468278376402] 
\draw  [line width=1.5] [line join = round][line cap = round] (570.67,175) .. controls (571.67,175) and (572.96,174.29) .. (573.67,175) .. controls (577.61,178.94) and (572.04,177.81) .. (571.67,178) .. controls (571.57,178.05) and (575.08,180.88) .. (573.67,183) .. controls (572.59,184.61) and (570.67,181.42) .. (570.67,181) ;
%Shape: Free Drawing [id:dp6111223602616063] 
\draw  [line width=1.5] [line join = round][line cap = round] (581.67,174) .. controls (581.67,176.74) and (579.67,177.4) .. (579.67,179) ;
%Shape: Free Drawing [id:dp2946419855125195] 
\draw  [line width=1.5] [line join = round][line cap = round] (585.67,166) .. controls (585.67,170.05) and (587.5,179.67) .. (587.67,182) ;
%Shape: Free Drawing [id:dp08964046925744595] 
\draw  [line width=1.5] [line join = round][line cap = round] (583.67,165) .. controls (593.57,165) and (590.91,173) .. (585.67,173) ;
%Shape: Free Drawing [id:dp5360813609913107] 
\draw  [line width=1.5] [line join = round][line cap = round] (600.67,174) .. controls (598.84,174) and (599.15,177.52) .. (597.67,179) ;
%Shape: Free Drawing [id:dp19494536690197617] 
\draw  [line width=1.5] [line join = round][line cap = round] (607.67,167) .. controls (607.67,171.68) and (608.67,176.32) .. (608.67,181) ;
%Shape: Free Drawing [id:dp13436002784583745] 
\draw  [line width=1.5] [line join = round][line cap = round] (602.67,168) .. controls (605.02,168) and (607.33,166.67) .. (609.67,167) .. controls (612.55,167.41) and (612.62,174) .. (608.67,174) ;
%Shape: Free Drawing [id:dp42493265834298566] 
\draw  [line width=1.5] [line join = round][line cap = round] (612.67,178) .. controls (612.67,179.08) and (611,184.83) .. (614.67,183) .. controls (616.8,181.93) and (616.67,177.98) .. (616.67,178) .. controls (616.67,180.36) and (617.67,182.64) .. (617.67,185) ;
%Shape: Free Drawing [id:dp14074831206603955] 
\draw  [line width=1.5] [line join = round][line cap = round] (622.67,179) .. controls (622.67,180.05) and (622.72,182) .. (621.67,182) ;
%Shape: Free Drawing [id:dp06756274101571291] 
\draw  [line width=1.5] [line join = round][line cap = round] (625.67,170) .. controls (628.78,170) and (628.67,185.22) .. (628.67,186) .. controls (628.67,188.33) and (629,181.31) .. (628.67,179) .. controls (628.19,175.63) and (624.63,171.72) .. (626.67,169) .. controls (627.81,167.48) and (631.62,170.83) .. (631.67,171) .. controls (632.18,173.05) and (629.67,174.89) .. (629.67,177) ;
%Shape: Free Drawing [id:dp8253850816102716] 
\draw  [line width=1.5] [line join = round][line cap = round] (632.67,178) .. controls (632.67,178.88) and (634.83,181.83) .. (636.67,180) .. controls (637.37,179.29) and (636.67,177) .. (636.67,177) .. controls (636.67,177) and (637.67,181.64) .. (637.67,184) ;
%Shape: Free Drawing [id:dp8866143380272151] 
\draw  [line width=1.5] [line join = round][line cap = round] (620.67,165) .. controls (624.35,165) and (627.98,164) .. (631.67,164) ;
%Shape: Free Drawing [id:dp4966797002475256] 
\draw  [line width=1.5] [line join = round][line cap = round] (639.67,162) .. controls (639.67,168.35) and (654.18,179.48) .. (644.67,189) ;
%Shape: Free Drawing [id:dp388391802384364] 
\draw  [line width=1.5] [line join = round][line cap = round] (449.67,163) .. controls (449.67,166.67) and (449.67,170.33) .. (449.67,174) ;
%Shape: Free Drawing [id:dp5439030990390685] 
\draw  [line width=1.5] [line join = round][line cap = round] (446.67,163) .. controls (451.43,163) and (462.31,168) .. (450.67,168) ;
%Shape: Free Drawing [id:dp8508914801958598] 
\draw  [line width=1.5] [line join = round][line cap = round] (454.67,172) .. controls (454.67,173.67) and (454.67,175.33) .. (454.67,177) ;
%Shape: Free Drawing [id:dp45587409297939496] 
\draw  [line width=1.5] [line join = round][line cap = round] (461.67,171) .. controls (461.67,173.14) and (460.84,177) .. (458.67,177) ;
%Shape: Free Drawing [id:dp8666048748259093] 
\draw  [line width=1.5] [line join = round][line cap = round] (463.67,157) .. controls (468.01,157) and (472.32,158) .. (476.67,158) ;
%Shape: Free Drawing [id:dp7113019422721956] 
\draw  [line width=1.5] [line join = round][line cap = round] (591.67,178) .. controls (591.67,179) and (590.96,180.29) .. (591.67,181) .. controls (592.22,181.55) and (595.33,180.99) .. (595.67,182) .. controls (597.04,186.12) and (591.67,186.84) .. (591.67,185) ;
%Shape: Free Drawing [id:dp5216871612734904] 
\draw  [line width=1.5] [line join = round][line cap = round] (591.67,178) .. controls (592.67,178) and (593.67,178) .. (594.67,178) ;

\end{tikzpicture}

%% file: marks.tex
\tikzset{every picture/.style={line width=0.75pt}} %set default line width to 0.75pt        

\begin{tikzpicture}[x=0.75pt,y=0.75pt,yscale=-1,xscale=1]

%Straight Lines [id:da5273846178943351] 
\draw    (431.67,127) -- (431.33,191.33) ;
%Straight Lines [id:da15812634883306875] 
\draw    (341.67,30) -- (341.33,92.33) ;
%Straight Lines [id:da9656593018689767] 
\draw [color={rgb, 255:red, 28; green, 138; blue, 248 }  ,draw opacity=1 ][line width=2.25]    (410.67,102) -- (431.67,127) ;
%Straight Lines [id:da222218978998605] 
\draw [color={rgb, 255:red, 248; green, 231; blue, 28 }  ,draw opacity=1 ][line width=2.25]    (341.5,61.17) -- (341.33,92.33) ;
%Straight Lines [id:da5295913555984131] 
\draw [color={rgb, 255:red, 248; green, 231; blue, 28 }  ,draw opacity=1 ][line width=2.25]    (431.5,159.17) -- (431.33,191.33) ;
%Straight Lines [id:da3298787972938375] 
\draw [color={rgb, 255:red, 28; green, 138; blue, 248 }  ,draw opacity=1 ][line width=2.25]    (264.67,164) -- (284.33,188.33) ;
%Straight Lines [id:da6125275265354747] 
\draw [color={rgb, 255:red, 28; green, 138; blue, 248 }  ,draw opacity=1 ][line width=2.25]    (322.67,5) -- (341.67,30) ;
%Shape: Circle [id:dp26338130785896496] 
\draw  [fill={rgb, 255:red, 0; green, 0; blue, 0 }  ,fill opacity=1 ] (276,188.33) .. controls (276,183.73) and (279.73,180) .. (284.33,180) .. controls (288.94,180) and (292.67,183.73) .. (292.67,188.33) .. controls (292.67,192.94) and (288.94,196.67) .. (284.33,196.67) .. controls (279.73,196.67) and (276,192.94) .. (276,188.33) -- cycle ;
%Shape: Circle [id:dp6394153219056924] 
\draw  [fill={rgb, 255:red, 0; green, 0; blue, 0 }  ,fill opacity=1 ] (333,92.33) .. controls (333,87.73) and (336.73,84) .. (341.33,84) .. controls (345.94,84) and (349.67,87.73) .. (349.67,92.33) .. controls (349.67,96.94) and (345.94,100.67) .. (341.33,100.67) .. controls (336.73,100.67) and (333,96.94) .. (333,92.33) -- cycle ;
%Shape: Circle [id:dp8683641248228599] 
\draw  [fill={rgb, 255:red, 0; green, 0; blue, 0 }  ,fill opacity=1 ] (333.33,30) .. controls (333.33,25.4) and (337.06,21.67) .. (341.67,21.67) .. controls (346.27,21.67) and (350,25.4) .. (350,30) .. controls (350,34.6) and (346.27,38.33) .. (341.67,38.33) .. controls (337.06,38.33) and (333.33,34.6) .. (333.33,30) -- cycle ;
%Shape: Circle [id:dp6445915517368376] 
\draw  [fill={rgb, 255:red, 0; green, 0; blue, 0 }  ,fill opacity=1 ] (423,191.33) .. controls (423,186.73) and (426.73,183) .. (431.33,183) .. controls (435.94,183) and (439.67,186.73) .. (439.67,191.33) .. controls (439.67,195.94) and (435.94,199.67) .. (431.33,199.67) .. controls (426.73,199.67) and (423,195.94) .. (423,191.33) -- cycle ;
%Shape: Circle [id:dp16888359955401366] 
\draw  [fill={rgb, 255:red, 0; green, 0; blue, 0 }  ,fill opacity=1 ] (423.33,127) .. controls (423.33,122.4) and (427.06,118.67) .. (431.67,118.67) .. controls (436.27,118.67) and (440,122.4) .. (440,127) .. controls (440,131.6) and (436.27,135.33) .. (431.67,135.33) .. controls (427.06,135.33) and (423.33,131.6) .. (423.33,127) -- cycle ;
%Shape: Free Drawing [id:dp715250995539353] 
\draw  [line width=1.5] [line join = round][line cap = round] (266.67,221) .. controls (267.36,221) and (269.52,217.71) .. (270.67,220) .. controls (271.08,220.83) and (270.44,233.11) .. (274.67,231) .. controls (276.52,230.08) and (278.77,224.68) .. (279.67,222) .. controls (279.88,221.37) and (279.67,220) .. (279.67,220) .. controls (279.67,220) and (280.09,224.71) .. (280.67,227) .. controls (282.43,234.07) and (287.5,233.69) .. (289.67,225) .. controls (290.16,223.03) and (290.67,221.03) .. (290.67,219) ;
%Shape: Free Drawing [id:dp3542166680799381] 
\draw  [line width=1.5] [line join = round][line cap = round] (420.67,221) .. controls (420.67,217.91) and (423.05,219.76) .. (423.67,221) .. controls (424.74,223.14) and (424.14,229.24) .. (425.67,230) .. controls (429.03,231.68) and (432.18,227.43) .. (432.67,225) .. controls (433,223.33) and (433.67,218.3) .. (433.67,220) .. controls (433.67,232.55) and (444.67,230.87) .. (444.67,220) ;
%Shape: Free Drawing [id:dp6656884289654509] 
\draw  [line width=1.5] [line join = round][line cap = round] (453.67,222) .. controls (453.67,220.85) and (456.53,215.14) .. (457.67,214) ;
%Straight Lines [id:da6445316475710703] 
\draw    (350.67,9) -- (341.67,30) ;
%Straight Lines [id:da5681091090756817] 
\draw    (360.67,21) -- (341.67,30) ;
%Straight Lines [id:da32822041232751975] 
\draw    (450.67,118) -- (431.67,127) ;
%Straight Lines [id:da1530527930150114] 
\draw    (437.67,105) -- (431.67,127) ;
%Shape: Free Drawing [id:dp6213788894140269] 
\draw  [color={rgb, 255:red, 28; green, 138; blue, 248 }  ,draw opacity=1 ][line width=1.5] [line join = round][line cap = round] (239.67,153) .. controls (242.38,153) and (243.57,147.7) .. (243.67,147) .. controls (244,144.67) and (242.76,137.64) .. (242.67,140) .. controls (242.33,148.33) and (242.67,156.67) .. (242.67,165) .. controls (242.67,166) and (242.42,162.97) .. (242.67,162) .. controls (243.12,160.19) and (245.63,155.52) .. (246.67,155) .. controls (250.14,153.26) and (250.03,161.74) .. (250.67,163) .. controls (251.19,164.05) and (255.67,167.33) .. (255.67,163) ;
%Shape: Free Drawing [id:dp08103116439259483] 
\draw  [color={rgb, 255:red, 248; green, 231; blue, 28 }  ,draw opacity=1 ][line width=1.5] [line join = round][line cap = round] (448.67,169) .. controls (452.31,169) and (457.86,163.76) .. (456.67,159) .. controls (454.97,152.2) and (455.67,172.99) .. (455.67,180) .. controls (455.67,182.92) and (455.68,176.97) .. (456.67,174) .. controls (457.15,172.56) and (459.99,169.34) .. (460.67,169) .. controls (464.79,166.94) and (464.12,178.64) .. (464.67,180) .. controls (465.2,181.32) and (467.22,179) .. (467.67,179) ;
%Shape: Free Drawing [id:dp17151005578587375] 
\draw  [color={rgb, 255:red, 248; green, 231; blue, 28 }  ,draw opacity=1 ][line width=1.5] [line join = round][line cap = round] (463.67,159) .. controls (466.02,155.48) and (466.67,154.3) .. (466.67,150) ;
%Shape: Free Drawing [id:dp03588345203360055] 
\draw  [color={rgb, 255:red, 248; green, 231; blue, 28 }  ,draw opacity=1 ][line width=1.5] [line join = round][line cap = round] (421.67,162) .. controls (421.67,155.91) and (415.05,148.53) .. (413.67,143) .. controls (408.15,120.95) and (393.15,91.74) .. (373.67,82) .. controls (371.64,80.99) and (369.13,81.23) .. (366.67,80) .. controls (362.68,78) and (358.31,75) .. (352.67,75) .. controls (351.92,75) and (354.33,75.33) .. (354.67,76) .. controls (355.2,77.07) and (356.13,77.92) .. (356.67,79) .. controls (357.09,79.85) and (359.37,85.59) .. (359.67,85) .. controls (360.74,82.85) and (352.52,73.15) .. (354.67,71) .. controls (356.6,69.06) and (361,71.67) .. (363.67,69) ;
%Shape: Free Drawing [id:dp8633475302558202] 
\draw  [color={rgb, 255:red, 28; green, 138; blue, 248 }  ,draw opacity=1 ][line width=1.5] [line join = round][line cap = round] (261.67,149) .. controls (269.3,118.46) and (272.7,86.93) .. (286.67,59) .. controls (292.72,46.89) and (303.57,35.09) .. (311.67,27) .. controls (315.54,23.12) and (324.33,17.04) .. (322.67,27) .. controls (322.02,30.89) and (322.94,20.41) .. (322.67,20) .. controls (321.86,18.78) and (313.93,18) .. (312.67,18) ;
%Shape: Free Drawing [id:dp4988345706807059] 
\draw  [color={rgb, 255:red, 0; green, 0; blue, 0 }  ,draw opacity=1 ][line width=1.5] [line join = round][line cap = round] (359.67,4) .. controls (359.67,5.83) and (357.41,12) .. (355.67,12) ;
%Shape: Free Drawing [id:dp5982456484657177] 
\draw  [color={rgb, 255:red, 0; green, 0; blue, 0 }  ,draw opacity=1 ][line width=1.5] [line join = round][line cap = round] (353.67,4) .. controls (356.6,6.94) and (359.82,11) .. (363.67,11) ;
%Shape: Free Drawing [id:dp9007830620047649] 
\draw  [color={rgb, 255:red, 0; green, 0; blue, 0 }  ,draw opacity=1 ][line width=1.5] [line join = round][line cap = round] (368.67,16) .. controls (368.67,17.25) and (367.61,28.53) .. (372.67,26) .. controls (374.53,25.07) and (375.15,14.93) .. (375.67,18) .. controls (376.87,25.22) and (381.83,49.72) .. (373.67,47) .. controls (360.53,42.62) and (385.79,28.88) .. (387.67,27) ;
%Shape: Free Drawing [id:dp9896083487726527] 
\draw  [color={rgb, 255:red, 0; green, 0; blue, 0 }  ,draw opacity=1 ][line width=1.5] [line join = round][line cap = round] (315.67,38) .. controls (309.97,38) and (314.18,51.71) .. (316.67,43) .. controls (317.03,41.72) and (316.67,40.33) .. (316.67,39) .. controls (316.67,37.95) and (317.52,40.96) .. (317.67,42) .. controls (317.76,42.62) and (318.95,50.86) .. (318.67,52) .. controls (318.37,53.18) and (313.67,54.95) .. (313.67,51) ;
%Shape: Free Drawing [id:dp18021860249702615] 
\draw  [color={rgb, 255:red, 0; green, 0; blue, 0 }  ,draw opacity=1 ][line width=1.5] [line join = round][line cap = round] (318.67,41) .. controls (320.14,41) and (321.67,40.02) .. (321.67,38) ;
%Shape: Free Drawing [id:dp7458103213451893] 
\draw  [color={rgb, 255:red, 0; green, 0; blue, 0 }  ,draw opacity=1 ][line width=1.5] [line join = round][line cap = round] (320.67,44) .. controls (322.15,44) and (323.19,42) .. (324.67,42) ;
%Shape: Free Drawing [id:dp42566358866977083] 
\draw  [color={rgb, 255:red, 0; green, 0; blue, 0 }  ,draw opacity=1 ][line width=1.5] [line join = round][line cap = round] (326.67,33) .. controls (326.67,34.55) and (324.14,43.26) .. (328.67,41) .. controls (334.27,38.2) and (326.67,28.27) .. (326.67,34) ;
%Shape: Free Drawing [id:dp5566192103648613] 
\draw  [color={rgb, 255:red, 0; green, 0; blue, 0 }  ,draw opacity=1 ][line width=1.5] [line join = round][line cap = round] (326.67,46) .. controls (326.67,49.25) and (332.47,50.55) .. (334.67,50) .. controls (335.11,49.89) and (337.15,46.52) .. (337.67,46) .. controls (340,43.67) and (338.67,55) .. (338.67,50) .. controls (338.67,47.67) and (340.19,44.77) .. (338.67,43) .. controls (337.37,41.48) and (334.67,43) .. (332.67,43) ;
%Shape: Free Drawing [id:dp6457197365848099] 
\draw  [color={rgb, 255:red, 0; green, 0; blue, 0 }  ,draw opacity=1 ][line width=1.5] [line join = round][line cap = round] (259.67,39.67) .. controls (259.67,33.67) and (259.67,27.67) .. (259.67,21.67) ;
%Shape: Free Drawing [id:dp5691958480588513] 
\draw  [color={rgb, 255:red, 0; green, 0; blue, 0 }  ,draw opacity=1 ][line width=1.5] [line join = round][line cap = round] (259.67,20.67) .. controls (263,20.67) and (266.33,20.67) .. (269.67,20.67) ;
%Shape: Free Drawing [id:dp7627839584068541] 
\draw  [color={rgb, 255:red, 0; green, 0; blue, 0 }  ,draw opacity=1 ][line width=1.5] [line join = round][line cap = round] (332.67,10.67) .. controls (332.67,9.08) and (336.35,8.92) .. (336.67,7.67) .. controls (336.99,6.37) and (336.67,2.33) .. (336.67,3.67) .. controls (336.67,6.55) and (333.83,14.67) .. (337.67,14.67) ;

\end{tikzpicture}

%% file: branchcuts.tex
\tikzset{every picture/.style={line width=0.75pt}} %set default line width to 0.75pt        

\begin{tikzpicture}[x=0.75pt,y=0.75pt,yscale=-1,xscale=1]

%Straight Lines [id:da805488959334586] 
\draw [line width=3]    (100.67,150) -- (200.67,150) ;
%Straight Lines [id:da44453403982580575] 
\draw [line width=3]    (299.67,150) -- (399.67,150) ;
%Straight Lines [id:da8283958373277821] 
\draw [line width=3]    (499.67,151) -- (599.67,151) ;
%Shape: Free Drawing [id:dp3983947296977519] 
\draw  [line width=3] [line join = round][line cap = round] (99.67,148) .. controls (96.67,148) and (97.28,153.46) .. (101.67,152) .. controls (103.07,151.53) and (105.23,144.15) .. (99.67,146) .. controls (96.13,147.18) and (100.73,153.58) .. (102.67,151) .. controls (103.8,149.49) and (99.84,145.53) .. (98.67,147) .. controls (96.27,150) and (100.85,154.82) .. (101.67,154) .. controls (107.32,148.34) and (98.67,144.86) .. (98.67,149) ;
%Shape: Free Drawing [id:dp130229109670857] 
\draw  [line width=3] [line join = round][line cap = round] (200.67,150) .. controls (205.1,150) and (202.21,145.15) .. (199.67,146) .. controls (197.47,146.73) and (197.09,151.47) .. (198.67,152) .. controls (199.5,152.28) and (204.8,151.13) .. (202.67,149) .. controls (200.67,147.01) and (198.64,152) .. (201.67,152) ;
%Shape: Free Drawing [id:dp6935518780888864] 
\draw  [line width=3] [line join = round][line cap = round] (300.67,150) .. controls (304.92,145.75) and (299.14,145.66) .. (298.67,148) .. controls (297.06,156.02) and (303.55,153.24) .. (304.67,151) .. controls (306.04,148.25) and (299.67,145.97) .. (299.67,149) ;
%Shape: Free Drawing [id:dp9118887168880514] 
\draw  [line width=3] [line join = round][line cap = round] (400.67,148) .. controls (396.38,148) and (400.28,154.39) .. (401.67,153) .. controls (406.42,148.25) and (399.67,142.81) .. (399.67,149) ;
%Shape: Free Drawing [id:dp287964875047404] 
\draw  [line width=3] [line join = round][line cap = round] (501.67,151) .. controls (499.72,151) and (500.81,148.57) .. (499.67,148) .. controls (497.74,147.04) and (495.51,151.84) .. (495.67,152) .. controls (496.8,153.14) and (504.79,154.5) .. (500.67,149) .. controls (498.95,146.71) and (495.38,150.29) .. (497.67,152) .. controls (502.67,155.75) and (501.67,150.52) .. (501.67,149) ;
%Shape: Free Drawing [id:dp8363716705481832] 
\draw  [line width=3] [line join = round][line cap = round] (599.67,147) .. controls (599.67,148.23) and (599.58,151.64) .. (600.67,152) .. controls (602.71,152.68) and (606.1,149.58) .. (604.67,149) .. controls (596.42,145.7) and (600.35,152.69) .. (601.67,154) .. controls (602.58,154.91) and (606.8,147.78) .. (603.67,147) .. controls (598.03,145.59) and (596.53,153) .. (600.67,153) ;
%Shape: Free Drawing [id:dp7294701503221115] 
\draw  [line width=3] [line join = round][line cap = round] (99.67,175) .. controls (99.67,180.3) and (97.67,185.13) .. (97.67,190) ;
%Shape: Free Drawing [id:dp46338677278809903] 
\draw  [line width=3] [line join = round][line cap = round] (195.67,173) .. controls (196.69,173) and (198.73,168.61) .. (200.67,172) .. controls (204.33,178.41) and (196.65,182.02) .. (194.67,184) .. controls (194.57,184.1) and (193.67,186) .. (193.67,186) .. controls (193.67,186) and (193.43,185.24) .. (193.67,185) .. controls (196.25,182.41) and (198.64,184) .. (202.67,184) ;
%Shape: Free Drawing [id:dp25623962893453867] 
\draw  [line width=3] [line join = round][line cap = round] (293.67,170) .. controls (293.67,168.4) and (293.33,163.67) .. (298.67,169) .. controls (303.34,173.67) and (295.26,174.6) .. (296.67,176) .. controls (298.24,177.58) and (300.96,177.11) .. (299.67,181) .. controls (298.43,184.72) and (293.82,181) .. (291.67,181) ;
%Shape: Free Drawing [id:dp000968403135439444] 
\draw  [line width=3] [line join = round][line cap = round] (393.67,166) .. controls (393.67,168.82) and (391.67,178.43) .. (400.67,172) .. controls (402.02,171.03) and (400.67,165.33) .. (400.67,167) .. controls (400.67,170.67) and (400.67,174.33) .. (400.67,178) ;
%Shape: Free Drawing [id:dp8833670075020228] 
\draw  [line width=3] [line join = round][line cap = round] (408.67,175) .. controls (411.69,175) and (414.65,176) .. (417.67,176) ;
%Shape: Free Drawing [id:dp6749580264650984] 
\draw  [line width=3] [line join = round][line cap = round] (413.67,171) .. controls (413.67,174.02) and (415.69,180) .. (412.67,180) ;
%Shape: Free Drawing [id:dp7658481344198259] 
\draw  [line width=3] [line join = round][line cap = round] (431.67,172) .. controls (431.67,173.49) and (427.07,176.6) .. (425.67,178) ;
%Shape: Free Drawing [id:dp8844299937233236] 
\draw  [line width=3] [line join = round][line cap = round] (423.67,171) .. controls (425.85,173.18) and (431.67,178.39) .. (431.67,180) ;
%Shape: Free Drawing [id:dp5049627160058853] 
\draw  [line width=3] [line join = round][line cap = round] (435.67,177) .. controls (435.67,184.31) and (440.46,177.11) .. (440.67,179) .. controls (441,181.98) and (440.67,185) .. (440.67,188) ;
%Shape: Free Drawing [id:dp255107199429609] 
\draw  [line width=3] [line join = round][line cap = round] (500.67,165) .. controls (500.67,166.7) and (498.15,169.24) .. (499.67,170) .. controls (501.26,170.79) and (504.94,170.55) .. (505.67,172) .. controls (507.13,174.92) and (499.67,178.75) .. (499.67,176) ;
%Shape: Free Drawing [id:dp7240849976198279] 
\draw  [line width=3] [line join = round][line cap = round] (501.67,164) .. controls (502.78,165.11) and (509.75,167) .. (510.67,167) ;
%Shape: Free Drawing [id:dp03135756949464974] 
\draw  [line width=3] [line join = round][line cap = round] (512.67,176) .. controls (516.05,176) and (517.38,180) .. (519.67,180) ;
%Shape: Free Drawing [id:dp0707528127133652] 
\draw  [line width=3] [line join = round][line cap = round] (519.67,173) .. controls (517.76,175.86) and (514.67,179.2) .. (514.67,182) ;
%Shape: Free Drawing [id:dp6750433721968808] 
\draw  [line width=3] [line join = round][line cap = round] (535.67,181) .. controls (534.61,181) and (525.67,187.16) .. (525.67,188) ;
%Shape: Free Drawing [id:dp085598746618746] 
\draw  [line width=3] [line join = round][line cap = round] (526.67,178) .. controls (528.04,178) and (532.67,190.89) .. (532.67,192) ;
%Shape: Free Drawing [id:dp6595437705069837] 
\draw  [line width=3] [line join = round][line cap = round] (537.67,190) .. controls (537.67,193.93) and (544.25,200) .. (536.67,200) ;
%Shape: Free Drawing [id:dp8892998844520816] 
\draw  [line width=3] [line join = round][line cap = round] (538.67,189) .. controls (540.94,189) and (542.88,191) .. (544.67,191) ;
%Shape: Free Drawing [id:dp5449413304675385] 
\draw  [line width=3] [line join = round][line cap = round] (599.67,172) .. controls (599.67,174.07) and (581.51,184.28) .. (595.67,189) .. controls (596.37,189.24) and (597.14,188.53) .. (597.67,188) .. controls (605.13,180.53) and (594.67,177) .. (594.67,181) ;
%Shape: Free Drawing [id:dp4030762481811653] 
\draw  [line width=3] [line join = round][line cap = round] (608.67,181) .. controls (610.67,181) and (612.67,181) .. (614.67,181) ;
%Shape: Free Drawing [id:dp522980238758559] 
\draw  [line width=3] [line join = round][line cap = round] (612.67,176) .. controls (612.67,179) and (612.67,182) .. (612.67,185) ;
%Shape: Free Drawing [id:dp7375673953761185] 
\draw  [line width=3] [line join = round][line cap = round] (631.67,176) .. controls (631.67,181.84) and (623.67,186.38) .. (623.67,191) ;
%Shape: Free Drawing [id:dp40322223531522516] 
\draw  [line width=3] [line join = round][line cap = round] (622.67,175) .. controls (624.94,178.41) and (627.88,189) .. (631.67,189) ;
%Shape: Free Drawing [id:dp17299162526759826] 
\draw  [line width=3] [line join = round][line cap = round] (640.67,184) .. controls (640.67,186.36) and (635.74,193.02) .. (641.67,195) .. controls (648.1,197.14) and (641.67,185.4) .. (641.67,191) ;
%Shape: Free Drawing [id:dp27366304664008523] 
\draw  [color={rgb, 255:red, 74; green, 144; blue, 226 }  ,draw opacity=1 ][line width=3] [line join = round][line cap = round] (377.67,147) .. controls (377.67,126.96) and (396.66,110.33) .. (415.67,104) .. controls (419.42,102.75) and (422.48,103.4) .. (426.67,102) .. controls (432.01,100.22) and (440.65,95.44) .. (447.67,95) .. controls (459.82,94.24) and (480.02,96.23) .. (490.67,98) .. controls (501.13,99.74) and (511.5,105.09) .. (521.67,108) .. controls (540.2,113.3) and (570.67,130.8) .. (570.67,149) ;
%Shape: Free Drawing [id:dp5485341730350162] 
\draw  [color={rgb, 255:red, 74; green, 144; blue, 226 }  ,draw opacity=1 ][line width=3] [line join = round][line cap = round] (465.67,85) .. controls (469.04,88.37) and (474.58,90.91) .. (477.67,94) .. controls (478.5,94.83) and (479.83,95.17) .. (480.67,96) .. controls (481.19,96.53) and (483.41,97) .. (482.67,97) .. controls (476.55,97) and (462.67,105.24) .. (462.67,110) ;
%Shape: Free Drawing [id:dp3758442882765207] 
\draw  [color={rgb, 255:red, 74; green, 144; blue, 226 }  ,draw opacity=1 ][line width=3] [line join = round][line cap = round] (586.67,148) .. controls (586.67,132.58) and (579.14,119.71) .. (570.67,107) .. controls (565.4,99.1) and (559.64,89.97) .. (552.67,83) .. controls (546.41,76.74) and (538.89,75.98) .. (532.67,71) .. controls (499.26,44.28) and (445.76,44.09) .. (404.67,45) .. controls (392.33,45.27) and (378.86,49.96) .. (368.67,52) .. controls (357.68,54.2) and (345.65,55.22) .. (334.67,56) .. controls (307.26,57.96) and (278.92,58.9) .. (253.67,69) .. controls (216.77,83.76) and (169.67,105.71) .. (169.67,146) ;
%Shape: Free Drawing [id:dp6355987443640776] 
\draw  [color={rgb, 255:red, 74; green, 144; blue, 226 }  ,draw opacity=1 ][line width=3] [line join = round][line cap = round] (375.67,41) .. controls (378.79,41) and (390.35,46.84) .. (392.67,48) .. controls (393.74,48.54) and (391.46,50.1) .. (390.67,51) .. controls (388.47,53.51) and (386.52,56.23) .. (384.67,59) .. controls (383.81,60.28) and (383.18,62) .. (381.67,62) ;
%Shape: Free Drawing [id:dp6590884269697695] 
\draw  [color={rgb, 255:red, 74; green, 144; blue, 226 }  ,draw opacity=1 ][line width=3] [line join = round][line cap = round] (330.67,14) .. controls (330.67,19.68) and (331.67,25.32) .. (331.67,31) ;
%Shape: Free Drawing [id:dp5826358932450346] 
\draw  [color={rgb, 255:red, 74; green, 144; blue, 226 }  ,draw opacity=1 ][line width=3] [line join = round][line cap = round] (321.67,16) .. controls (326.69,16) and (332.71,7.04) .. (338.67,13) .. controls (339.93,14.26) and (339.62,17.05) .. (338.67,18) .. controls (338.09,18.58) and (335.25,19.42) .. (334.67,20) .. controls (334.14,20.53) and (331.92,21) .. (332.67,21) .. controls (346.69,21) and (343.88,32) .. (331.67,32) ;
%Shape: Free Drawing [id:dp4849874255071567] 
\draw  [color={rgb, 255:red, 74; green, 144; blue, 226 }  ,draw opacity=1 ][line width=3] [line join = round][line cap = round] (347.67,25) .. controls (347.67,28.71) and (350.67,32.44) .. (350.67,35) ;
%Shape: Free Drawing [id:dp4499541607276184] 
\draw  [color={rgb, 255:red, 74; green, 144; blue, 226 }  ,draw opacity=1 ][line width=3] [line join = round][line cap = round] (438.67,66) .. controls (438.67,68.8) and (440.97,82) .. (437.67,82) ;
%Shape: Free Drawing [id:dp07290340433748821] 
\draw  [color={rgb, 255:red, 74; green, 144; blue, 226 }  ,draw opacity=1 ][line width=3] [line join = round][line cap = round] (435.67,68) .. controls (435.67,65.62) and (445.3,63.86) .. (446.67,65) .. controls (451.96,69.41) and (441.86,72) .. (440.67,72) .. controls (439.29,72) and (443.44,72.39) .. (444.67,73) .. controls (445.12,73.23) and (449.69,76.98) .. (448.67,78) .. controls (447.94,78.72) and (439.67,84) .. (439.67,80) ;
%Shape: Free Drawing [id:dp5192456297849535] 
\draw  [color={rgb, 255:red, 74; green, 144; blue, 226 }  ,draw opacity=1 ][line width=3] [line join = round][line cap = round] (452.67,78) .. controls (463.65,78) and (451.67,85) .. (451.67,85) .. controls (451.67,85) and (456.67,85.64) .. (456.67,87) ;
%Shape: Free Drawing [id:dp491762402221692] 
\draw  [color={rgb, 255:red, 208; green, 2; blue, 27 }  ,draw opacity=1 ][line width=3] [line join = round][line cap = round] (153.67,218) .. controls (134.55,218) and (117.24,218.43) .. (103.67,213) .. controls (83.56,204.96) and (78.37,196.82) .. (73.67,178) .. controls (72.12,171.83) and (69.8,169.5) .. (70.67,160) .. controls (70.97,156.69) and (74.5,150.67) .. (75.67,146) .. controls (76.26,143.63) and (90.08,138.17) .. (91.67,135) .. controls (95.25,127.82) and (104.38,127.29) .. (108.67,123) .. controls (110.7,120.97) and (115.34,121.11) .. (118.67,120) .. controls (133.46,115.07) and (152.56,110.99) .. (170.67,112) .. controls (173.09,112.13) and (176.31,114.41) .. (178.67,115) .. controls (198.89,120.06) and (210.98,127.93) .. (217.67,148) .. controls (218.31,149.93) and (220.65,150.98) .. (221.67,152) .. controls (225.55,155.88) and (225.6,162.67) .. (226.67,168) .. controls (229.33,181.31) and (229.86,198.81) .. (219.67,209) .. controls (207.68,220.98) and (187.52,220.62) .. (169.67,220) .. controls (164.07,219.81) and (158.16,217) .. (152.67,217) ;
%Shape: Free Drawing [id:dp3210219503489091] 
\draw  [color={rgb, 255:red, 208; green, 2; blue, 27 }  ,draw opacity=1 ][line width=3] [line join = round][line cap = round] (129.67,106) .. controls (134.05,106) and (139.63,110.99) .. (143.67,112) .. controls (144.96,112.32) and (146.47,111.4) .. (147.67,112) .. controls (148.26,112.3) and (147.96,113.4) .. (147.67,114) .. controls (146.14,117.06) and (140.61,127) .. (138.67,127) ;
%Shape: Free Drawing [id:dp11248056473505719] 
\draw  [color={rgb, 255:red, 208; green, 2; blue, 27 }  ,draw opacity=1 ][line width=3] [line join = round][line cap = round] (102.67,88) .. controls (102.67,85.54) and (106.58,83.71) .. (107.67,81) .. controls (109.85,75.53) and (111.44,70.11) .. (112.67,64) .. controls (112.73,63.67) and (113.52,63.7) .. (113.67,64) .. controls (114.9,66.46) and (115.85,70.54) .. (116.67,73) .. controls (118.22,77.67) and (121.67,80.06) .. (121.67,84) ;
%Shape: Free Drawing [id:dp3494308782441735] 
\draw  [color={rgb, 255:red, 208; green, 2; blue, 27 }  ,draw opacity=1 ][line width=3] [line join = round][line cap = round] (107.67,77) .. controls (110.69,77) and (113.65,76) .. (116.67,76) ;
%Shape: Free Drawing [id:dp3062517199350142] 
\draw  [color={rgb, 255:red, 208; green, 2; blue, 27 }  ,draw opacity=1 ][line width=3] [line join = round][line cap = round] (127.67,78) .. controls (128,82.67) and (128.33,87.33) .. (128.67,92) ;
%Shape: Free Drawing [id:dp7949766925740347] 
\draw  [color={rgb, 255:red, 208; green, 2; blue, 27 }  ,draw opacity=1 ][line width=3] [line join = round][line cap = round] (325.67,108) .. controls (317.54,108) and (310.94,111.09) .. (303.67,114) .. controls (282.34,122.53) and (277.99,127.7) .. (272.67,149) .. controls (269.07,163.39) and (269.61,178.94) .. (279.67,189) .. controls (287.01,196.35) and (299.2,197.18) .. (310.67,201) .. controls (325.82,206.05) and (343.55,210.74) .. (359.67,211) .. controls (380.33,211.33) and (401.61,215.01) .. (421.67,210) .. controls (431.25,207.6) and (450.66,204.02) .. (454.67,196) .. controls (460.96,183.41) and (458.86,165.26) .. (449.67,153) .. controls (433.22,131.07) and (399.94,106.2) .. (368.67,105) .. controls (352.29,104.37) and (339.09,107) .. (324.67,107) ;
%Shape: Free Drawing [id:dp839422670509609] 
\draw  [color={rgb, 255:red, 208; green, 2; blue, 27 }  ,draw opacity=1 ][line width=3] [line join = round][line cap = round] (339.67,96) .. controls (339.67,99.94) and (349.73,101.12) .. (350.67,103) .. controls (351.65,104.97) and (340.67,112.58) .. (340.67,117) ;
%Shape: Free Drawing [id:dp32367571331390754] 
\draw  [color={rgb, 255:red, 208; green, 2; blue, 27 }  ,draw opacity=1 ][line width=3] [line join = round][line cap = round] (273.67,102) .. controls (273.67,97.13) and (272.47,77.71) .. (274.67,81) .. controls (278.37,86.56) and (287.67,96.02) .. (287.67,101) ;
%Shape: Free Drawing [id:dp5061307996550082] 
\draw  [color={rgb, 255:red, 208; green, 2; blue, 27 }  ,draw opacity=1 ][line width=3] [line join = round][line cap = round] (274.67,95) .. controls (277.67,95) and (280.67,95) .. (283.67,95) ;
%Shape: Free Drawing [id:dp37939438468541564] 
\draw  [color={rgb, 255:red, 208; green, 2; blue, 27 }  ,draw opacity=1 ][line width=3] [line join = round][line cap = round] (290.67,92) .. controls (292.69,92) and (296.27,91.01) .. (296.67,93) .. controls (297.14,95.38) and (295.75,97.83) .. (294.67,100) .. controls (294.2,100.94) and (292.63,102.83) .. (293.67,103) .. controls (295.64,103.33) and (297.67,103) .. (299.67,103) ;
%Shape: Free Drawing [id:dp3531628454716612] 
\draw  [color={rgb, 255:red, 74; green, 144; blue, 226 }  ,draw opacity=1 ][line width=3] [line join = round][line cap = round] (173.67,160) .. controls (173.67,164.02) and (176.67,167.62) .. (176.67,170) ;
%Shape: Free Drawing [id:dp34291887684699607] 
\draw  [color={rgb, 255:red, 74; green, 144; blue, 226 }  ,draw opacity=1 ][line width=3] [line join = round][line cap = round] (181.67,181) .. controls (183.69,181) and (190.67,193.21) .. (190.67,194) ;
%Shape: Free Drawing [id:dp05929125373985278] 
\draw  [color={rgb, 255:red, 74; green, 144; blue, 226 }  ,draw opacity=1 ][line width=3] [line join = round][line cap = round] (195.67,200) .. controls (198.44,201.85) and (201.31,203.64) .. (203.67,206) ;
%Shape: Free Drawing [id:dp030334490058586017] 
\draw  [color={rgb, 255:red, 74; green, 144; blue, 226 }  ,draw opacity=1 ][line width=3] [line join = round][line cap = round] (221.67,216) .. controls (224.52,216) and (234.08,222.42) .. (235.67,224) ;
%Shape: Free Drawing [id:dp7081941814000148] 
\draw  [color={rgb, 255:red, 74; green, 144; blue, 226 }  ,draw opacity=1 ][line width=3] [line join = round][line cap = round] (246.67,228) .. controls (250.7,228) and (260.61,235) .. (266.67,235) ;
%Shape: Free Drawing [id:dp7963153724578717] 
\draw  [color={rgb, 255:red, 74; green, 144; blue, 226 }  ,draw opacity=1 ][line width=3] [line join = round][line cap = round] (276.67,237) .. controls (282.41,239.87) and (309.61,246) .. (315.67,246) ;
%Shape: Free Drawing [id:dp863581266487046] 
\draw  [color={rgb, 255:red, 74; green, 144; blue, 226 }  ,draw opacity=1 ][line width=3] [line join = round][line cap = round] (327.67,248) .. controls (348.73,248) and (354.15,249) .. (369.67,249) ;
%Shape: Free Drawing [id:dp4222055780724724] 
\draw  [color={rgb, 255:red, 74; green, 144; blue, 226 }  ,draw opacity=1 ][line width=3] [line join = round][line cap = round] (384.67,250) .. controls (399.34,250) and (414.01,250.44) .. (428.67,251) ;
%Shape: Free Drawing [id:dp009851058353678943] 
\draw  [color={rgb, 255:red, 74; green, 144; blue, 226 }  ,draw opacity=1 ][line width=3] [line join = round][line cap = round] (442.67,251) .. controls (450.26,251) and (455.75,252.61) .. (463.67,252) .. controls (467.14,251.73) and (471.54,249) .. (474.67,249) ;
%Shape: Free Drawing [id:dp08023612369791222] 
\draw  [color={rgb, 255:red, 74; green, 144; blue, 226 }  ,draw opacity=1 ][line width=3] [line join = round][line cap = round] (485.67,250) .. controls (494.16,250) and (501.67,246) .. (509.67,246) ;
%Shape: Free Drawing [id:dp9740521764773716] 
\draw  [color={rgb, 255:red, 74; green, 144; blue, 226 }  ,draw opacity=1 ][line width=3] [line join = round][line cap = round] (518.67,243) .. controls (519,243) and (519.37,243.15) .. (519.67,243) .. controls (519.85,242.91) and (550.67,231.2) .. (550.67,228) ;
%Shape: Free Drawing [id:dp7113826459557846] 
\draw  [color={rgb, 255:red, 74; green, 144; blue, 226 }  ,draw opacity=1 ][line width=3] [line join = round][line cap = round] (559.67,220) .. controls (559.67,216.11) and (570.67,211.8) .. (570.67,206) ;
%Shape: Free Drawing [id:dp48883668320693063] 
\draw  [color={rgb, 255:red, 74; green, 144; blue, 226 }  ,draw opacity=1 ][line width=3] [line join = round][line cap = round] (573.67,199) .. controls (573.67,193.32) and (580.67,187.78) .. (580.67,181) ;
%Shape: Free Drawing [id:dp9635927756292841] 
\draw  [color={rgb, 255:red, 74; green, 144; blue, 226 }  ,draw opacity=1 ][line width=3] [line join = round][line cap = round] (584.67,173) .. controls (587.31,169.04) and (587.67,166.5) .. (587.67,162) ;
%Shape: Free Drawing [id:dp5118526151071577] 
\draw  [color={rgb, 255:red, 74; green, 144; blue, 226 }  ,draw opacity=1 ][line width=3] [line join = round][line cap = round] (378.67,158) .. controls (378.67,160.54) and (377.88,172) .. (382.67,172) ;
%Shape: Free Drawing [id:dp501028751989378] 
\draw  [color={rgb, 255:red, 74; green, 144; blue, 226 }  ,draw opacity=1 ][line width=3] [line join = round][line cap = round] (385.67,178) .. controls (385.67,180.62) and (394.07,189) .. (395.67,189) ;
%Shape: Free Drawing [id:dp7420081202915043] 
\draw  [color={rgb, 255:red, 74; green, 144; blue, 226 }  ,draw opacity=1 ][line width=3] [line join = round][line cap = round] (398.67,192) .. controls (401.79,192) and (406.29,199) .. (408.67,199) ;
%Shape: Free Drawing [id:dp5687198811086611] 
\draw  [color={rgb, 255:red, 74; green, 144; blue, 226 }  ,draw opacity=1 ][line width=3] [line join = round][line cap = round] (415.67,201) .. controls (415.67,203.05) and (419.59,205) .. (421.67,205) ;
%Shape: Free Drawing [id:dp3730752107821107] 
\draw  [color={rgb, 255:red, 74; green, 144; blue, 226 }  ,draw opacity=1 ][line width=3] [line join = round][line cap = round] (435.67,211) .. controls (437.66,212.99) and (444.81,215) .. (447.67,215) ;
%Shape: Free Drawing [id:dp7724157097132096] 
\draw  [color={rgb, 255:red, 74; green, 144; blue, 226 }  ,draw opacity=1 ][line width=3] [line join = round][line cap = round] (455.67,217) .. controls (459.28,217) and (465.61,219) .. (470.67,219) ;
%Shape: Free Drawing [id:dp5590426843962889] 
\draw  [color={rgb, 255:red, 74; green, 144; blue, 226 }  ,draw opacity=1 ][line width=3] [line join = round][line cap = round] (478.67,219) .. controls (485.67,219) and (492.67,219) .. (499.67,219) ;
%Shape: Free Drawing [id:dp7699987567582367] 
\draw  [color={rgb, 255:red, 74; green, 144; blue, 226 }  ,draw opacity=1 ][line width=3] [line join = round][line cap = round] (508.67,218) .. controls (512.67,218) and (516.67,218) .. (520.67,218) ;
%Shape: Free Drawing [id:dp22375720406234123] 
\draw  [color={rgb, 255:red, 74; green, 144; blue, 226 }  ,draw opacity=1 ][line width=3] [line join = round][line cap = round] (531.67,215) .. controls (537.44,215) and (542.76,209) .. (545.67,209) ;
%Shape: Free Drawing [id:dp8795010407555356] 
\draw  [color={rgb, 255:red, 74; green, 144; blue, 226 }  ,draw opacity=1 ][line width=3] [line join = round][line cap = round] (553.67,201) .. controls (557.69,201) and (565.67,186.56) .. (565.67,184) ;
%Shape: Free Drawing [id:dp5718084910290567] 
\draw  [color={rgb, 255:red, 74; green, 144; blue, 226 }  ,draw opacity=1 ][line width=3] [line join = round][line cap = round] (568.67,175) .. controls (568.67,171.75) and (570.67,164.69) .. (570.67,160) ;

\end{tikzpicture}